\documentclass[10pt, a4paper]{amsart}
\usepackage{amsmath,amsfonts,amsthm}
\usepackage{amssymb}
\usepackage{xypic}
\usepackage{graphicx}
\usepackage{mathrsfs}
\usepackage{booktabs}
\usepackage{verbatim,comment}
\usepackage{amsthm}
\usepackage[dvipsnames]{xcolor}
\usepackage{bm,esint}
\usepackage{marginnote}

\usepackage{mathrsfs} 

\usepackage{color}
\usepackage[colorlinks=true,urlcolor=blue,
citecolor=red,linkcolor=blue,linktocpage,pdfpagelabels,bookmarksnumbered,bookmarksopen]{hyperref}

\DeclareMathOperator*{\argmin}{arg\,min}
\DeclareMathOperator*{\argmax}{arg\,max}
\usepackage{tikz}

\theoremstyle{plain}
\newtheorem{thm}{Theorem}[section]
\newtheorem{theorem}{Theorem}[section]
\newtheorem{lem}[thm]{Lemma}
\newtheorem{lemma}[thm]{Lemma}
\newtheorem{proposition}[thm]{Proposition}

\theoremstyle{definition}

\newtheorem{definition}{Definition}[section]

\theoremstyle{remark}

\newtheorem{remark}{Remark}[section]

\renewcommand{\epsilon}{\varepsilon}

\providecommand{\RR}{\mathbb{R}}

\providecommand{\cE}{\mathcal{E}}
\providecommand{\cF}{\mathcal{F}}
\providecommand{\cG}{\mathcal{G}}

\providecommand{\cP}{\mathcal{P}}

\providecommand{\cS}{\mathcal{S}}

\renewcommand{\a}{\alpha}
\renewcommand{\b}{\beta}

\renewcommand{\d}{\delta}

\providecommand{\e}{\epsilon}
\providecommand{\ve}{\varepsilon}

\renewcommand{\l}{\lambda}

\providecommand{\s}{\sigma}

\providecommand{\Om}{\Omega}
\renewcommand{\t}{\tau}
\providecommand{\vp}{\varphi}
\providecommand{\vphi}{\varphi}

\providecommand{\norm}[1]{\left\| #1 \right\|}

\DeclareMathOperator{\spt}{spt}
\DeclareMathOperator{\tr}{Tr}

\providecommand{\norm}[1]{\lVert #1 \rVert}

\newcommand{\HC}{{\rm{HC}}}
\newcommand{\hc}{{\rm{HC}}}
\newcommand{\ov}{\overline}

\newcommand{\N}{\mathbb{N}}

\newcommand{\R}{\mathbb{R}}

\def\dd{\,{\rm d}} 
\newcommand{\ds}{\displaystyle}
\newcommand{\weakly}{\rightharpoonup}
\newcommand{\weaklys}{\stackrel{\star}{\rightharpoonup}}
\newcommand{\da}{\downarrow}

\newcommand\ac{{\rm ac}}
\renewcommand\ae{{\rm a.e.}}

\newcommand{\mres}{\mathbin{\vrule height 1.6ex depth 0pt width
0.13ex\vrule height 0.13ex depth 0pt width 1.3ex}}

\newcommand{\sD}{\mathscr{D}}

\newcommand{\sL}{\mathscr{L}}

\newcommand{\sH}{{\mathscr H}}
\newcommand{\sM}{{\mathscr M}}
\newcommand{\sP}{{\mathscr P}}

\title{Weak solutions to the Muskat problem with surface tension via optimal transport}

\author[M. Jacobs]{Matt Jacobs}  
\address{Department of Mathematics, UCLA, 520 Portola Plaza, Los Angeles, CA 90095, USA}
\email{majaco@math.ucla.edu} 

\author[I. Kim]{Inwon Kim}  
\address{Department of Mathematics, UCLA, 520 Portola Plaza, Los Angeles, CA 90095, USA}
\email{ikim@math.ucla.edu} 

\author[A.R. M\'esz\'aros]{Alp\'ar R. M\'esz\'aros}  
\address{Department of Mathematical Sciences, Durham University, South Road, DH1 3LE, Durham, UK}
\email{alpar.r.meszaros@durham.ac.uk}

\begin{document}

\begin{abstract}
   Inspired by recent works on the threshold dynamics scheme for multi-phase mean curvature flow (by Esedo\={g}lu-Otto and Laux-Otto), we introduce a novel framework to approximate solutions of the Muskat problem with surface tension.  Our approach is based on interpreting the Muskat problem as a gradient flow in a product Wasserstein space.   This perspective allows us to construct weak solutions via a minimizing movements scheme.    Rather than working directly with the singular surface tension force,  we instead relax the perimeter functional with the heat content energy approximation of Esedo\={g}lu-Otto.   The heat content energy allows us to show the convergence of the associated minimizing movement scheme in the Wasserstein space, and  makes the scheme far more tractable for numerical simulations.   Under a typical energy convergence assumption, we show that our scheme converges to weak solutions of the Muskat problem with surface tension.  We then conclude the paper with a discussion {on some numerical experiments and} on equilibrium configurations.  
\end{abstract}

\maketitle

\section{Introduction}
The Muskat problem was first introduced by Morris Muskat \cite{muskat_original} as a model for the flow of two immiscible fluids through a porous medium.  Since its introduction, this problem has received sustained attention in a variety of fields.  It is used to model flows in oil reservoirs (water is injected into the oil well to drive oil extraction), and in hydrology to model flows of groundwater through aquifers.    

In this paper we are interested in obtaining the global existence of weak solutions for the Muskat problem with surface tension, based on its gradient flow structure. We begin by introducing a variational formulation of the problem, which will motivate our subsequent analysis.  The fluid evolution can be written as Darcy's law
\begin{equation}\label{eq:velocity} v_i+b^{-1}_i\nabla \d_{\rho_i}\cE(\bm{\rho})=0,\end{equation}
coupled with the continuity equation
\begin{equation}\label{eq:density}\partial_t \rho_i +\nabla \cdot (\rho_i v_i)=0,\end{equation}
where $v_i$ is the velocity of phase $i$, $\bm{\rho}=(\rho_1,\rho_2)$ is the collection of relative concentrations for each phase,  $\nabla (\d_{\rho_i}\cE)$ denotes the {spacial gradient} of the classical first variation of the {free energy} with respect to $\rho_i$, and $b_i>0$ ($i=1,2$) denotes constant mobilities. For convenience, throughout the rest of the paper, we will refer to $\bm{\rho}$ as a collection of density functions; however, one should note that $\bm{\rho}$ only encodes information about the volume occupied by the fluids and nothing about their mass.

\medskip

The physical setting for our problem is a bounded, convex open domain $\Omega\subset \RR^d$ with smooth boundary.  We shall suppose that the two fluids fill the entire domain, and that they are confined to $\Omega$ for all time.  We then take the internal energy to be a sum of three distinct terms:
\begin{equation} \cE(\bm{\rho})=\cE_p(\bm{\rho})+\cE_s(\bm{\rho})+\Phi(\bm{\rho}). \end{equation}
The first term in the energy, describing incompressibility and containment of the fluids,  is given by
\begin{equation}
\cE_p(\bm{\rho})\label{eq:incompressibility}=
\begin{cases}
0, \quad\quad \textrm{if} \;  \rho_1(x)+\rho_2(x)=1 \; \textrm{for a.e.} \, x\in\Omega \\
+\infty, \quad \textrm{otherwise}.
\end{cases}
\end{equation}
The immiscibility of the fluids and the surface tension force arise from the highly non-convex interaction energy
\begin{equation}
\cE_s(\bm{\rho})\label{eq:interaction_energy}=
\begin{cases}
\frac{\sigma}{2}|D\rho_1|(\Omega)+\frac{\sigma}{2}| D\rho_2|(\Omega), \quad \textrm{if} \;  \rho_1,\rho_2\in BV(\Om;\{0,1\})\  {\rm{and}}\ \rho_1(x)\rho_2(x)=0 \; \textrm{for a.e.} \, x\in\Omega \\
+\infty, \qquad\qquad\qquad\qquad\quad\ \  \textrm{otherwise},
\end{cases}
\end{equation}
where $|D\rho_i|(\Omega)$ denotes the total variation of $\rho_i$ in $\Omega$ and $\sigma>0$ is a surface tension constant.  Finally, $\Phi(\bm{\rho})$ denotes the potential energy of the fluid configuration, i.e. 
$$
\Phi(\bm{\rho})=\int_\Om\Phi_1\dd\rho_1+\int_\Om\Phi_2\dd\rho_2,
$$
where $\Phi_1,\Phi_2:\Om\to\R$ are given Lipschitz continuous potentials. A typical example is when one assumes these to 
be gravitational potentials, i.e.
\begin{equation}\label{eq:gravitational_potential_energy}
\Phi_i= g_ix\cdot e_d, \quad g_i >0,\ \  i=1,2  
\end{equation} 
where $g_i$'s are proportional to the specific gravity of each fluid.

Although the internal energy is singular, when $\rho_1$ and $\rho_2$ are separated by a smooth interface $\Gamma:= \partial\{\rho_1>0\}\cap \partial\{\rho_2>0\}$, one can formulate a classical solution to the Muskat problem equations (\ref{eq:velocity}-\ref{eq:density}). In the classical solution, the flow is driven by the pressure variables $p_i$ for each phase, which are Lagrange multipliers generated by $\cE_p$ above. The continuity equation becomes 
\begin{equation}\label{continuity}
\partial_t\rho_i - b^{-1}_i\nabla\cdot \big((\nabla p_i + \nabla \Phi_i)\rho_i\big) = 0\tag{MP$_1$}
 \end{equation}
and the pressure is determined by solving the free boundary problem
\begin{equation}\label{eq:Muskat}
\left\{\begin{array}{lll}
-\Delta p_i = \Delta \Phi_i &\hbox{ in } &\spt(\rho_i);\\ \\
\partial_n  (p_i + \Phi_i) =0 &\hbox{ on } &\partial\Omega; \\ \\
V = b_1^{-1}\partial_{n} (p_1 + \Phi_1) = b_2^{-1}\partial_{n}(p_2+\Phi_2) &\hbox{ on }& \Gamma ;\\ \\
\left[p\right]:=(p_1-p_2) = \frac{\sigma}{2}\kappa & \hbox{ on } &\Gamma; \\ \\
\tilde{n}=n & \hbox{on} & \partial \Gamma\cap \partial\Omega;\\
\end{array}\right.\tag{MP$_2$}
\end{equation} 
where $\kappa$ denotes the mean curvature of $\Gamma$, oriented to be positive when $\{\rho_2>0\}$ is convex at the point, $n$ denotes the outer normal along $\partial \Omega$ and along $\Gamma$, and $\tilde{n}$ denotes the co-normal vector  orthogonal to $\partial \Gamma$ and tangential to $\Gamma$.  Note that the final condition relating the co-normal vector at $\partial\Gamma\cap\partial \Omega$ to the normal vector of $\partial \Omega$ implies that $\Gamma$ must meet $\partial\Omega$ orthogonally (see Lemma \ref{lem:consistency} and Remark \ref{rmk:boundary} for the weak formulation of this condition). To summarize the ideas in the formal derivation of \eqref{continuity}-\eqref{eq:Muskat} from \eqref{eq:velocity}-\eqref{eq:density} using the definition of $\cE$, we heuristically have $\nabla\d_{\rho_i}\cE_p=\nabla p_i$, $\nabla\d_{\rho_i}\Phi=\nabla \Phi_i$ while the contribution of $\nabla\d_{\rho_i}\cE_s$ will act only on $\Gamma$ in the form of the curvature $\kappa$.

Problems like \eqref{eq:Muskat} received a lot of attention in the past decades. Most of the works focus on the zero surface tension model ($\sigma=0$) and well-posedness of regular solutions with graph property \cite{Amb,CasCorFefGan,ConCorGanStr,ConCorGanRodStr,CorCorGan}. In the presence of surface tension, the problem has stronger regularity properties in stable settings \cite {PruSim}, but still, topological singularities can occur in finite time, for instance when heavier fluid is placed on top of the lighter one \cite{EsMat}.  Thus, our aim is to construct global-in-time weak solutions to the Muskat problem \eqref{eq:Muskat}, which exist past the formation of singularities.

\medskip

To construct global-in-time solutions, we exploit the gradient flow structure of the Muskat problem.  As noted by Otto in \cite{Ott, otto_geometry_porous_media}, Darcy's law can be approximated by the Euler-Lagrange equation for the minimizing movements scheme (or JKO \cite{JKO} scheme) with time step size $\tau>0$,
 \begin{equation}
\label{eq:general_minimizing_movements} \bm{\rho}^{n+1}=\argmin_{\bm{\rho}} \left\{\cE(\bm{\rho})+\sum_{i=1}^2\frac{b_i}{2\tau} W_2^2(\rho_i, \rho_i^n)\right\}
\end{equation} 
where $W_2(\rho_i, \rho_i^n)$ denotes the 2-Wasserstein or 2-Monge-Kantorovich distance.
In this context, the squared $W_2$ distance has a physical interpretation as the energy dissipated by friction as the fluids flow through the porous media.

Let us note that Wasserstein gradient flows of energies involving total variation terms have been considered before in the literature, though only in the case of one phase models (see e.g. \cite{MatMcCSav, CarPoo}), and hence with no incompressibility or interaction constraints.  As a result, the techniques developed in those papers do not appear to be applicable here --- the constrained two phase setting adds many additional difficulties.  

Indeed, it is not easy to obtain a complete characterization of the solutions to the minimizing movements problem (\ref{eq:general_minimizing_movements}).
   The interaction energy (\ref{eq:interaction_energy}) is sufficiently non-convex that problem (\ref{eq:general_minimizing_movements}) is non-convex for \emph{any} $\tau>0$.  As a result, one must be careful in using duality to introduce the pressure as a Lagrange multiplier.  Furthermore, we are interested in developing a scheme which could be used for numerical implementations.  The formulation (\ref{eq:general_minimizing_movements}) is poorly suited for numerical methods. Optimizing over the non-convex constraint set $\{\rho_1, \rho_2\in BV(\Omega; \{0,1\}): \rho_1(x)\rho_2(x)=0\; \textrm{a.e.}\}$ is extremely difficult.  For these reasons, we instead consider a relaxed version of minimizing movements scheme inspired by \cite{EO} and \cite{ LauOtt}.   
\medskip

\medskip

\medskip

\noindent {\bf Approximation of the perimeter by the Heat Content}

\medskip

In our analysis, we replace the interaction energy $\cE_s$ in \eqref{eq:general_minimizing_movements} by the {\it heat content energy}
 \begin{equation}\label{def:HC}
\textrm{HC}_{\epsilon}(\bm{\rho}):= \sigma\sqrt{\frac{2\pi}{\e}}\int_{\Omega} (G_{\epsilon} \star\rho_1)(x)\dd\rho_2(x)=\sigma\sqrt{\frac{2\pi}{\e}}\int_{\Omega} \int_{\RR^d} G(z) \dd\rho_1(x+\sqrt{\epsilon} z)\dd\rho_2(x).
\end{equation}
Here $G_\e:\R^d\to\R$ stands for the standard heat kernel (with mean 0 and variance $\e>0$) and the densities $\rho_i$ are assumed to be defined on all of $\RR^d$ by extending them to zero off of $\Omega$. {Let us notice that in \cite{EO} and \cite{LauOtt}, for similar purposes the authors use periodic extensions. The analysis in both cases and the validity of results using both kinds of extensions is essentially the same.}

{Dating back to the work of De Giorgi, approximate perimeter energies have been used in the literature to study geometric variational problems (see for instance \cite{MirPalParPre} and \cite{AlbBel}). The use of the heat content energy to study the multi-phase mean curvature flow was first introduced by Esedo\={g}lu and Otto in \cite{EO}.}   It was observed in \cite{EO} that the threshold dynamics, a well known numerical scheme for mean curvature motion introduced by Merriman, Bence, and Osher \cite{MBO92},  is precisely a minimizing movements scheme for the heat content energy. Esedo\={g}lu and Otto also showed that $\textrm{HC}_\e$ $\Gamma$-converges (with respect to the $L^1$ topology) to $\cE_s$ as $\e\to 0$. Building off of these results, Laux and Otto showed in \cite{LauOtt} that under an energy convergence assumption, the threshold dynamics scheme produces weak solutions to the multi-phase motion by mean curvature in the limit $\epsilon\to 0$. 

Our goal is to consider such a framework in the context of the Muskat problem by studying the minimizing movements scheme 
\begin{equation}
\label{eq:hc_minimizing_movements} \bm{\rho}^{n+1}=\argmin_{\bm{\rho}} \left\{\cE_\e(\bm{\rho})+\sum_{i=1}^2\frac{b_i}{2\tau} W_2^2(\rho_i, \rho_i^n)\right\},
\end{equation} 
where we used the notation 
$$\cE_\e(\bm{\rho}):=\cE_p(\bm{\rho})+ \HC_{\epsilon}(\bm{\rho})+\Phi(\bm{\rho}).$$
\\
As we alluded above,  the scheme (\ref{eq:hc_minimizing_movements}) has a number of numerical advantages over (\ref{eq:general_minimizing_movements}). Unlike $\mathcal{E}_s$ which is neither convex nor concave, the heat content is a strictly concave functional of the densities.  This concavity can be exploited to simplify numerical implementations, along the same lines as the linearization trick noted in Subsection 5.1 of \cite{EO}.  After applying this trick, the resulting variational problem becomes convex, and thus, can be efficiently solved using the recently introduced back-and-forth method \cite{JacLeg}.   See Figures \ref{fig:small_drop}-\ref{fig:ripped_drop}, for a demonstration of the numerical performance of the scheme.

Although the heat content in principle allows mixing of the phases, we shall show that the discrete in time solutions constructed by the JKO scheme always stay unmixed with a sharp interface between the phases for all time  (see Proposition \ref{prop:characteristic} below). This phenomenon is due to the fact that $\HC_\e$ behaves like a strictly concave functional (see Lemma \ref{lem:conc}). Thus, one retains the essential properties of the Muskat problem evolution. 

 In the context of the Muskat problem, the heat content also has a natural physical interpretation.   In a discrete statistical mechanics model with $N$ particles, surface tension can be seen to arise from short range interactions between particles in different phases (cf. \cite{irving_kirkwood}). Typically, the discrete surface tension takes the form 
$$\frac{1}{N^2}\sum_{i\in P_1, j\in P_2}V(|x_i-x_j|)=\frac{1}{N^2}\sum_{i\in P_1,j\in P_2}\int_{\RR^d}\int_{\RR^d} V(|x-x'|)\delta(x-x_i)\delta(x'-x_j)$$
where $P_r$ is the particle index in each phase $r$, $V$ is some decreasing function, and $x_i$ is the location of particle $i$ (\cite{irving_kirkwood}). By taking the limit $N\to\infty$ in above formula, one obtains an analogue of the heat content energy where the kernel $G_{\e}$ is replaced with $V$.   Here it is worth noting that we choose to work with the heat kernel for computational convenience, indeed a different choice of kernel may be more physically relevant. 

Finally, let us also emphasize that the heat content approach can be naturally extended to the multiphase Muskat problem evolution (with any number of phases), without incurring any additional difficulties (just as in the case of \cite{EO} and \cite{LauOtt}).  This includes scenarios where the surface tension force depends on the phases that are interacting \cite{EO}. To present our ideas in the simplest possible way, we do not pursue the multiphase case in this paper.

\medskip
 
\noindent {\bf Statement of our main results}

\medskip

From the relaxed minimizing movements scheme (\ref{eq:hc_minimizing_movements}), we obtain a sequence of discrete in time approximations to the Muskat flow.  When we take the time step $\tau$ and the heat content approximation parameter $\epsilon$ to zero together, we hope to recover weak solutions to the Muskat problem.  Our main results show that this is indeed the case under the assumption that there is no loss of perimeter when passing from discrete to continuous solutions. For the precise statements of the convergence results we refer to Theorem \ref{thm:MAIN} and Theorem \ref{thm:limits}.  In addition, we show that our weak formulation (see Definition \ref{def:weak_sol}) encodes all of the conditions in \eqref{continuity}-\eqref{eq:Muskat} (see Lemma \ref{lem:consistency} and Remark \ref{rmk:boundary}).

\begin{theorem}\label{thm:main} Let $T>0$ be a fixed time horizon, let $\e,\t>0$ be fixed and let $(\rho_1^{\e,\t},\rho_{2}^{\e,\t})$ be the discrete in time interpolations  of the densities obtained from the minimizing movements scheme \eqref{eq:hc_minimizing_movements}. Let moreover $p^{\e,\t}$ stand for the discrete in time interpolations between the scalar pressure fields, obtained as Lagrange multipliers associated to the incompressibility constraint in \eqref{eq:hc_minimizing_movements}. Then
\begin{itemize} 
\item There exists a family $(A^{\e,\t}_t)_{t\in[0,T]}\subseteq\Om$ of measurable sets such that $\rho_{1}^{\e,\t}(t,\cdot)=\chi_{A^{\e,\t}_t}$ and $\rho_{2}^{\e,\t}(t,\cdot)=\chi_{\Om\setminus A^{\e,\t}_t}$.\\ 

\item There exists $\rho_i\in L^1([0,T];BV(\Om;\{0,1\}))\cap {\rm{AC}}^2([0,T];\sP(\Om))$ ($i=1,2$) such that as $\max\{\e,\t\}\da 0$ and along a subsequence $(\rho_{1}^{\e,\t}, \rho_{2}^{\e,\t})\to (\rho_1,\rho_2)$ strongly in $L^1([0,T]\times\Om)\times L^1([0,T]\times\Om)$. Moreover, $\rho_1, \rho_2\in L^1([0,T];BV(\Om))$  and they are also characteristic functions that sum up to one, i.e. 
\begin{equation}\label{eq:MP1}
 \rho_1(t,\cdot) =\chi_{A_t\cap \Omega}\hbox{ and }\rho_2(t,\cdot) = \chi_{\Omega\setminus A_t},
\end{equation}
for a measurable family of sets $(A_t)_{t\in[0,T]}$ which are of finite perimeter. 

\item There exists a scalar pressure field $p\in L^2([0,T]; (C^{0,\a}(\Om))^*)$ such that along a subsequence $\nabla p^{\e,\t}\weaklys\nabla p$ weakly-$\star$ in $L^2([0,T]; (C^{1}(\Om))^*)$ as $\max\{\e,\t\}\da 0$.

\item Finally, under the assumption of energy convergence
\begin{equation}\label{assumption}
\int_0^T {\rm HC}_\e(\rho_{1}^{\e,\t}, \rho_{2}^{\e,\t})\dd t \to \int_0^T \left(\sum_i \frac{\sigma}{2}\int_{\Omega} |D\rho_i|\right)\dd x\dd t,\tag{EC}
\end{equation}
$(\rho_i, v_i,p)$ with $i=1,2$ solves, in the weak sense (see Definition \ref{def:weak_sol}),  the problem \eqref{continuity}-\eqref{eq:Muskat}. Here formally $v_i$ corresponds to $-b_i^{-1}(\nabla p + \nabla \Phi_i)$.

\end{itemize}
\end{theorem}

It is challenging to study qualitative or geometric properties of our solutions. We will illustrate heuristically in Section \ref{sec:equilibrium} that, even for global minimizers of the energy, there are diverse possibilities depending on the values of the specific gravity and volumes of the two phases.   This is verified with several numerical simulations in Section \ref{sec:numer}.

\medskip

\noindent {\bf Remarks on our results}

\begin{itemize}
\item[(1)] Let us underline the fact that our discrete-time scheme (\ref{eq:hc_minimizing_movements}) produces minimizers that are characteristic functions of a partition of $\Om$. To prove this fact, we exploit the strict concavity of the heat content along the admissible set. This seems to be an interesting property in its own right, and ensures that numerical implementations of the scheme maintain a sharp interface at every time step.

\item[(2)] From the point of view of Wasserstein gradient flows, an interesting remark on our results is that while one cannot expect strong compactness for the interpolated densities $(\rho_{1}^{\e,\t},\rho_{2}^{\e,\t})$ in the case when $\e>0$ is fixed and $\t\da 0$, when sending both parameters to 0 in the same time, we regain the strong compactness. This phenomenon is mainly due to the fact that in the limit as $\max\{\t,\e\}\da 0$ we recover total variation estimates on the densities. This compactness is obtained via a standard Aubin-Lions type argument.

\item[(3)] The energy convergence assumption \eqref{assumption} is rather natural and the same as the ones given in \cite{LauOtt} and \cite{LucStu}.
 This assumption ensures that there is no sudden loss of boundary between phases in the limit $\e\to 0$.   If there is a loss of interface then one cannot obtain weak solutions.   Indeed, if two components of the support of $\rho_1$ merge into each other and remove a sizable part of its boundary, one can expect a discontinuous change of the pressure term $p$ in the entire domain $\Omega$, creating an inconsistency between the discrete and limiting evolutions.

Replacing the assumption with a direct argument has been discussed for mean curvature flows \cite{Tak, DeLau}. Unfortunately, these results rely strongly on certain properties of the mean curvature flow (especially the comparison principle), which do not hold for fourth order equations like the Muskat problem.

There are also results in the literature \cite{Luc, Rog} which eliminate the assumption for third order curvature driven flows (specifically the Stefan problem and the Mullins-Sekerka flow respectively). However, the solutions constructed in \cite{Luc} are discontinuous in time  (the interface may experience sudden jumps in time) and the formulation in \cite{Rog} only keeps track of the regular part of the interface. Let us also note that the Stefan problem and the Mullins-Sekerka flow are not similar to the Muskat problem.  In particular, the jump condition for the pressure across the interface is different for the Muskat problem, which leads to qualitatively different behavior.

\end{itemize}

\medskip

\noindent {\bf Paper summary}

\medskip

 The rest of the paper is organized as follows.    In Section \ref{sec:MM}, we derive the basic properties of the minimizing movements scheme (\ref{eq:hc_minimizing_movements}) and construct our discrete-time quantities.  We begin by showing that solutions to the minimization problem are characteristic functions at every time step of the discrete scheme.  We then derive the existence of pressure as a Lagrange multiplier for the incompressibility constraint and obtain the Euler-Lagrange equation for the minimization problem.   Our derivation of these equations were inspired by previous results from \cite{DiFFag,MauRouSan1,Lab17, KimMes}. In particular, the definition of the discrete in time pressure variable is very much inspired by \cite{MauRouSan1}.
 
 In Section \ref{sec:muskat}, we take $\tau$ and $\epsilon$ to zero together to obtain weak solutions to the Muskat problem, under the assumption that the internal energy of the discrete solutions converges to the internal energy of the limiting solutions.   The main task in this section amounts to showing that one can pass to the limit in the Euler-Lagrange equation obtained in Section \ref{sec:MM}.  This can be done using the standard theory for Wasserstein gradient flows if $\epsilon$ is held fixed.  However, the joint limit, $\tau, \epsilon_{\tau}\to 0$, requires an adaptation of the arguments of \cite{LauOtt} to the case of Wasserstein gradient flows.  Let us underline that one 
can rely entirely on the results of \cite{LauOtt} to pass to the limit the weak curvature equation. An adaptation of the Aubin-Lions type argument from \cite{LauOtt} can be done to get compactness of the density terms. The only difference here is that we are using $W_2$ as metric while in \cite{LauOtt} a different metric is used, but this does not impose crucial difficulties. An interesting link to the flow-exchange technique introduced in \cite{MatMcCSav}, which is a typical tool for $W_2$ gradient flows, is also pointed out. Last, we developed the necessary estimates and compactness results on the pressure terms. These are new and clearly were not present in the setting of mean curvature flows. The compactness that we get on the pressure is in the sense of distributions.
 
Finally, in Section \ref{sec:equilibrium}, we conclude the paper with a demonstration of the numerical method on several examples and a discussion on the global minimizers of the approximated internal energy associated to the Muskat problem. While this discussion remains at the heuristic level, our conjectures are supported by the equilibrium states attained in our numerical experiments (c.f. Figures \ref{fig:small_drop}-\ref{fig:ripped_drop}). We end the paper with an appendix section, where we recall the results from \cite{LauOtt} that are used when passing to the limit the weak curvature equation.

\medskip

 \medskip
 
 \medskip

\section{The Wasserstein minimizing movements scheme for the heat content}\label{sec:MM}

\subsection{Some preliminary results}

Recall that the setting for our problem is a smooth convex domain $\Omega\subset \RR^d$ and without loss of generality by scaling we assume that $\sL^d(\Om)=2$. By $\sP(\Omega)$ we denote the space of Borel probability measures on $\RR^d$ supported on $\overline{\Omega}$. $\sP^{\ac}(\Omega)$ stands for the elements of $\sP(\Omega)$ that are absolutely continuous with respect to $\sL^d\mres\Omega$. 

Let $\Phi_1,\Phi_2:\Omega\to\R$ be given Lipschitz potentials and let us recall the definition of the potential energy $\Phi:\sP(\Omega)\times\sP(\Omega)\to\R$, given as 
$$\Phi(\bm{\rho}):=\int_{\Omega}\Phi_1\dd\rho_1+\int_{\Omega}\Phi_2\dd\rho_2.$$
Let $\e>0$. We consider the heat content $\HC_\e:\sP(\Omega)\times\sP(\Omega)\to\R$ defied in \eqref{def:HC}, using the standard heat kernel $G_\e:\R^d\to\R$, i.e. 
$$G_\e(x)=\frac{1}{(4\pi\e)^{d/2}}e^{-\frac{|x|^2}{4\e}}.$$
We also use the notations $K_\e,G:\R^d\to\R$ to denote $K_\e(x)=\sigma\sqrt{\frac{2\pi}{\e}}G_\e(x)$ and $G(x)=G_1(x).$
We have the following preliminary results.

\begin{lemma}\label{lem:lambda-conv}
Let $\HC_\e$ be defined as in \eqref{def:HC}. We have the following properties. 
\begin{itemize}
\item[(1)] $\HC_\e$ is bounded from below and continuous w.r.t. the weak-$\star$ convergence on $\sP(\Omega)\times\sP(\Omega).$ 
\item[(2)] $\HC_\e$ displacement $\l$-convex on $\sP(\Omega)\times\sP(\Omega)$, with $\l=-\frac{\s\sqrt{2\pi}}{(4\pi)^{d/2}}\frac{1}{\e^{(d+3)/2}}$.
\end{itemize}
\end{lemma}

\begin{proof}
(1) Is immediate by the definition of $\HC_\e$. 

To show (2), it is enough to show that the function $z\mapsto K_\e(z)$ is $\l$-convex in the classical sense for some $\l\in\R$. We have
$$D^2 K_\e(x)=\frac{\s\sqrt{2\pi}}{2(4\pi)^{d/2}}\frac{1}{\e^{(d+3)/2}}e^{-\frac{|x|^2}{4\e}}\left(\frac{1}{2\e}x\otimes x- I_d\right).$$
Since the matrix $x\otimes x$ is positive semidefinite for any $x\in\R^d$, setting $\l=-\frac{1}{2(4\pi)^{d/2}}\frac{\s\sqrt{2\pi}}{\e^{(d+3)/2}}$, we have that the matrix $D^2K_\e(x)-\l I_d$ is positive semidefinite for any $x\in\R^d,$ which implies in particular that $K_\e$ is $\l$-convex.

We conclude similarly as in \cite[Lemma 2.1]{DiFFag} the displacement $2\l$-convexity of $\HC_\e$ on $\sP(\Omega)\times\sP(\Omega).$ 
\end{proof}

\begin{lemma}\label{lem:conc} If $P\subset \sP(\Omega)\times\sP(\Omega)$ denotes the set of pairs $(\rho_1, \rho_2)$ such that $\rho_1+\rho_2=1$ a.e. in $\Omega$, then the heat content is strictly concave along line segments in $P$.
\end{lemma}
\begin{proof}

For any pair $(\rho_1, \rho_2)\in P$ we may write $\rho_2(x)=1-\rho_1(x)$.  Thus, extending the densities by $0$ outside of $\Om$, we have
\small
$$\hc_{\e}(\rho_1,\rho_2)=\sigma\sqrt{\frac{2\pi}{\e}}\int_{\RR^d} (G_{\e}\star\rho_1)(x)(1-\rho_1(x))\dd x=\sigma\sqrt{\frac{2\pi}{\e}}\int_{\RR^d} (G_{\e}\star\rho_1)(x)\dd x-\sigma\sqrt{\frac{2\pi}{\e}}\int_{\RR^d} (G_{\e}\star\rho_1)(x)\rho_1(x)\dd x$$
\normalsize
Ignoring the constant multiples, both terms can be expressed conveniently in the Fourier domain:
$$ \hat{\rho}_1(0)-\int_{\RR^d} \hat{G}(\xi\sqrt{\e})|\hat{\rho}_1(\xi)|^2\dd\xi.$$
Now the strict concavity follows immediately as $\hat{G}(\xi\sqrt{\e})>0$ for all $\xi\in\RR^d$. 
\end{proof}

\subsection{The minimizing movements scheme}

Now we are ready to discuss the minimizing movements scheme.  Our first result confirms the existence of minimizers, and shows that any minimizing configuration $\bm{\rho}=(\rho_1,\rho_2)$ is a completely unmixed partition of the domain.  As we will see, the phases stay unmixed thanks to the concavity of the heat content. 

\begin{proposition}\label{prop:characteristic}
Suppose that $\rho_1^n, \rho_2^n\in \sP(\Om)$ and let $\t>0$ and $b_1,b_2>0$. 
Then the set of minimizers of the problem
\begin{equation}\label{eq:mm_unique}
\inf \left\{ \HC_{\epsilon}(\bm{\rho})+\Phi(\bm{\rho})+\frac{b_1}{2\tau}W_2^2(\rho_1, \rho_1^n)+\frac{b_2}{2\tau}W_2^2(\rho_2, \rho_2^n):\ \rho_1,\rho_2\in\sP^{\ac}(\Om), \rho_1+\rho_2=1\ \ae\right\}
\end{equation}
is non-empty and any solution $(\rho_1^*,\rho_2^*) \in\sP(\Om)\times\sP(\Om)$ is the characteristic function of a partition of $\Omega$.  
\end{proposition}

\begin{proof}
The existence of a solution of the optimization problem is an easy consequence of the weak lower semicontinuity of the objective functional and the weak-$\star$ compactness of $\sP(\Om)\times\sP(\Om)$. Let us remark that the constraint $\rho_1+\rho_2=1$ a.e. is closed under weak convergence, since $\int_{\Om}\rho_1\dd x+\int_{\Om}\rho_2\dd x=\sL^d(\Om).$

To show that an arbitrary solution $(\rho_1^*,\rho_2^*)$ is the characteristic functions of a partition of $\Om$, let us rewrite equivalently the minimization problem in terms of transport plans $\pi_i\in\sP(\Om\times\Om)$. 
Recall that $\pi_i$ is a plan between $\rho_i^n$ and $\rho_i$, whenever 
$$\int_{\Om\times\Om}\vphi(x)\dd\pi_i(x,y)=\int_{\Om}\vphi(x)\dd\rho_i^n(x)\ \ {\rm{and}}\ \ \int_{\Om\times\Om}\psi(y)\dd\pi_i(x,y)=\int_{\Om}\psi(y)\dd\rho_i(y),$$
for any $\vphi,\psi\in C(\Om).$ Since we are always working with measures $\rho_i^n,\rho_i$ that are absolutely continuous w.r.t. $\sL^d\mres\Om$, in the new minimization problem below, as we will see, we can restrict our search to plans that have absolutely continuous marginals w.r.t. $\sL^d\mres\Om.$ 
For a measure $\theta\in\sP(\Om\times\Om)$, we use the notation $\theta^1:=(P^x)_\#\theta$ and $\theta^2:=(P^y)_\#\theta$ to denote its marginals (here $P^x,P^y:\Om\times\Om\to\Om$ stand for the canonical projections from $\Om\times\Om$ onto $\Om$).

Thus, we aim to solve 
\begin{equation}\label{prob:min}
\begin{array}{l}
\ds\inf\left\{\cF(\pi_1,\pi_2)+\cG(\pi_1,\pi_2)+\sum_{i=1}^2\int_{\Om\times\Om}\frac{b_i}{2\tau}|x-y|^2\dd\pi_i(x,y)\right\}=:\inf\cS(\pi_1,\pi_2)\\[5pt]
\ds{\rm{subject\ to\ }} \pi_i^1=\rho^n_i \ {\rm{and}}\ \sum_{i=1}^2 \pi_i^2=1\ \ae
\end{array}
\end{equation}
Here we denote
$$\cG(\pi_1,\pi_2):=\int_{\Om\times\Om}\sum_{i=1}^2\Phi_i(y)\dd\pi_i(x,y).$$
We define moreover
\begin{align*}
\cF(\pi_1,\pi_2)&:=\sigma\sqrt{\frac{2\pi}{\e}}\int_{\Om}\int_{\Om\times \RR^d}G(x_1-\sqrt{\e}y)\dd\pi_1(x,y)\dd x_1\\
&-\sigma\sqrt{\frac{2\pi}{\e}}\int_{\Om\times\Om}\int_{\Om\times\RR^d} G(y_1-\sqrt{\e}y_2)\dd\pi_1(x_2,y_2)\dd \pi_1(x_1,y_1)
\end{align*}
and
$$\cS(\pi_1,\pi_2):=\cF(\pi_1,\pi_2)+\cG(\pi_1,\pi_2)+\sum_{i=1}^2\int_{\Om\times\Om}\frac{b_i}{2\tau}|x-y|^2\dd\pi_i(x,y),$$
where we have extended the second marginals of $\pi_i$ by 0 outside of $\Om$.

The minimization is carried out over a weakly compact set and $\cS$ is weakly lower semicontinuous and bounded below, thus minimizers exist.  If $\bm{\pi}^*=(\pi_1,\pi_2)$ is a minimizer of \eqref{prob:min} then we can construct a minimizer $\bm{\rho}^*=(\rho_1^*,\rho_2^*)$ of the original problem by taking $\rho_i^*=(P^y)_\#\pi^*_i$. 

Now we consider the properties of minimizers.  Clearly, $\cF$ is G\^ateaux differentiable at $\bm{\pi}^*$ in the sense that there exists $\delta\cF(\bm{\pi}^*)\in C(\Om\times\Om)$ such that  
\begin{align}\label{def:Gateaux}
\langle\delta\cF(\bm{\pi}^*), \bm{\theta}\rangle&:=\sigma\sqrt{\frac{2\pi}{\e}}\int_{\Om}\int_{\Om\times \RR^d}G(x_1-\sqrt{\e}y)\dd \theta_1(x,y)\dd x_1\\
\nonumber&-\sigma\sqrt{\frac{2\pi}{\e}}\int_{\Om\times\Om}\int_{\Om\times\RR^d} G(y_1-\sqrt{\e}y_2)\dd\pi_1(x_2,y_2)\dd \theta_1(x_1,y_1)\\
\nonumber&-\sigma\sqrt{\frac{2\pi}{\e}}\int_{\Om\times\Om}\int_{\Om\times\RR^d} G(y_1-\sqrt{\e}y_2)\dd \theta_1(x_2,y_2)\dd \pi_1(x_1,y_1),
\end{align}
where $\bm{\pi}+t\bm{\theta}$ is any admissible perturbation of $\bm{\pi}.$ Similarly, as the other terms in the definition of $\cS$ are linear in $\pi$, these are in the same way differentiable, therefore $\cS$ is G\^ateaux differentiable in this sense.

From Lemma \ref{lem:conc} it follows that $\cS$ is concave along line segments $\bm{\pi}+t\bm{\theta}$ ($t\in(-1,1)$), where $\bm{\pi}$ is a feasible point and $\bm{\theta}=(\theta_1,\theta_2)$ is a feasible direction at $\bm{\pi}$ i.e.
\begin{equation}\label{eq:variation}
\theta_i^1(x)=0\ \ae\ x\in\Om,\ \ \int_{\Om\times\Om}\dd \theta_i(x,y)=0\ \ {\rm{and}}\ \ \sum_{i=1}^2 \theta_i^2(y)=0 \ \ae\ y\in\Om,
\end{equation}
and for some $\delta>0$
\begin{equation}\label{eq:valid_perturbation}
\pi_i+t\theta_i\geq 0
\end{equation}
in the sense of signed measures, for all $t\in [0,\delta)$.
Furthermore, it follows from Lemma \ref{lem:conc} that $\cS$ is strictly concave on line segments $\bm{\pi}+t\bm{\theta}$, if for some $i$ the marginal $\theta_i^2(y) $ is not $0$ for almost every $y$.
Therefore, if $\bm{\pi}^*=(\pi_1^*, \pi_2^*)$ is a minimizer and $\bm{\theta}$ is a feasible direction at $\bm{\pi^*}$ with at least one non-trivial marginal then 
$$\langle\delta\cS(\bm{\pi}^*), \bm{\theta}\rangle>0,$$   
where $\delta\cS(\bm{\pi}^*)$ stands for the first variation of $\cS$ at $\bm{\pi}^*$ defined as is \eqref{def:Gateaux}.

Let $\bm{\pi}$ be a feasible solution and let $\rho_i(y)= \pi_i^2(y)$.  Now suppose that there exists $0<\alpha<1$ such that the set $\Omega_{\alpha}=\{y\in \Om: \rho_1(y), \rho_2(y) \in (\alpha, 1-\alpha)\}$ has positive measure.  Partition $\Omega_{\alpha}$ into two sets $E_1, E_2$ of equal measure. Then there exist measure preserving maps $T:E_1\to E_2$ and $S:E_2\to E_1$ such that $S\circ T$ and $T\circ S$ are the identity almost everywhere on their respective domains (for example one may choose $T=\nabla \psi$ to be the optimal transport map between the densities $\bm{1}_{E_1}$ and $\bm{1}_{E_2}$ and $S=\nabla \psi^*$).  Now we construct a feasible direction $\bm{\theta}$ at $\bm{\pi}$ as follows.  Let $r(y)=\frac{\rho_1(T(y))}{\rho_2(y)}$ for $y\in E_1$ and we define the signed measures $\theta_1,\theta_2\in\sM(\Om\times\Om)$ as
$$\theta_1:=(\rho_1^n\otimes(\rho_1\circ T))\mres(\Om\times E_1)-(\rho_1^n\otimes\rho_1)\mres(\Om\times E_2),$$
i.e.
$$\int_{\Om\times\Om}\vphi(x,y)\dd \theta_1(x,y):=\int_{\Om\times E_1}\vphi(x,y)\dd\rho_1^n(x)\dd\rho_1(T(y))-\int_{\Om\times E_2}\vphi(x,y)\dd\rho_1^n(x)\dd\rho_1(y)$$
and
$$\theta_2:=-(\rho_2^n\otimes r\rho_2)\mres(\Om\times E_1)+(\rho_2^n\otimes (r\circ S) (\rho_2\circ S))\mres (\Om\times E_2),$$
i.e.
$$\int_{\Om\times\Om}\vphi(x,y)\dd \theta_2(x,y):=-\int_{\Om\times E_1}\vphi(x,y)r(y)\dd\rho_2^n(x)\dd\rho_2(y)+\int_{\Om\times E_2}\vphi(x,y)r(S(y))\dd\rho_2^n(x)\dd\rho_2(S(y)),$$
for any $\vphi\in C(\Om\times\Om).$

%
Since $\theta_1^2(y)=\rho_1(T(y))>\alpha$ a.e. on $E_1$ we see that $\bm{\theta}$ has a nontrivial marginal.  
 
Let us now check that $\bm{\theta}$ is feasible, i.e. it satisfies \eqref{eq:variation} and \eqref{eq:valid_perturbation}. If $\beta<\min\{1-\a,\alpha\}$ then
$\bm{\pi}\pm \beta\bm{\theta}$ defines a non-negative measure. Next, we check that $\bm{\theta}$ satisfies $\theta_i^1(x)=0$ for a.e.  $x\in\Om$ and $i=1,2$ and $\sum_{i=1}^2 \theta_i^2(y) =0$ for a.e. $y\in\Om$.   
For $\vphi\in C(\Om)$, we have
$$\int_\Om\vphi(x)\dd \theta^1_1(x)= \int_{\Om\times\Om}\vphi(x) \dd \theta_1(x,y)= \int_\Om\vphi(x)\dd\rho_1^n(x)\int_{E_1}\rho_1(T(y))\dd y - \int_\Om\vphi(x)\rho_1^n(x)\int_{E_2} \rho_1(y)\dd y=0 $$
and
\begin{align*}
\int_{\Om}\vphi(x) \dd \theta_2^1(x)&= \int_{\Om\times\Om}\vphi(x)\dd \theta_1(x,y)\\
&=-\int_{\Om}\vphi(x)\dd\rho_2^n(x)\int_{E_1} r(y)\rho_2(y)\dd y + \int_\Om\vphi(x)\dd\rho_2^n(x)\int_{E_2} r(S(y))\rho_2(S(y))\dd y=0
\end{align*}
where we have used that both $T$ and $S$ are measure preserving between $E_1$ and $E_2$ and vice-versa, respectively. Let us notice that these arguments also show (by taking $\vphi\equiv 1$) that 
$$\int_{\Om\times\Om} \dd \theta_i(x,y) =0,\ \ i=1,2.$$ 
Now, for $\vphi\in C(\Om),$ we have
\begin{align*}
\sum_{i=1}^2\int_{\Om}\vphi(y)\dd \theta_i^2(y)&=\sum_{i=1}^2\int_{\Om\times\Om}\vphi(y)\dd \theta_i(x,y)\\
&=\int_{\Om}\dd\rho_1^n(x)\int_{E_1}\vphi(y)\rho_1(T(y))\dd y-\int_{\Om}\dd \rho_1^n(x)\int_{E_2}\vphi(y)\dd\rho_1(y)\\
&-\int_\Om\dd\rho_2^n(x)\int_{E_1}\vphi(y)r(y)\dd\rho_2(y)+\int_\Om\dd\rho_2^n(x)\int_{E_2}\vphi(y)r(S(y))\rho_2(S(y))\dd y\\
&=\int_{E_1}\phi(y)[\rho_1(T(y))-r(y)\rho_2(y)]\dd y + \int_{E_2}\phi(y)[r(S(y))\rho_2(S(y))-\rho_1(y)]\dd y = 0.
\end{align*} 

Note that our arguments in fact show that $-\bm{\theta}$ is also a feasible direction at $\bm{\pi}$.  $\pm\bm{\theta}$ have nontrivial marginals, and it is not possible to have both $\langle\delta\cS(\bm{\pi}^*), \bm{\theta}\rangle>0$ and $-\langle\delta\cS(\bm{\pi}^*), \bm{\theta}\rangle>0$.  Therefore, $\bm{\pi}$ cannot be a minimizer.  This allows us to conclude that for any minimizer $\bm{\rho}^*$ of the original problem each density $\rho_i^*$ takes values $\{0,1\}$ almost everywhere. 

\end{proof}


\subsection{Optimality conditions and construction of the pressure variables}

In the next Lemma, we give a more complete characterization of the minimizers in terms of certain necessary inequalities.  In particular, this is the first place where we see the appearance of the pressure variable, which plays an essential role in all of the subsequent analysis.   Note that for convenience we express this result using the notation $K_\e:=\sigma\sqrt{\frac{2\pi}{\e}}G_{\e}.$

\begin{lemma}\label{lem:opt_cond}
Let $(\rho_1^*,\rho_2^*)$ be an optimizer in \eqref{eq:mm_unique} and let $(\rho_1,\rho_2)$ a pair of probability measures such that $\rho_1+\rho_2=1$ a.e. Then, 
\begin{itemize}
\item[(i)] we have the following optimality condition
\begin{equation}\label{ineq:optimality}
\int_{\Om}\left[K_\e\star\rho_2^*+\Phi_1+b_1\vphi_1/\t\right](\rho_1-\rho_1^*)\dd x+\int_{\Om}\left[K_\e\star\rho_1^*+\Phi_2+b_2\vphi_2/\t\right](\rho_2-\rho_2^*)\dd x\ge 0,
\end{equation}
for a suitable pair of Kantorovich potentials $(\vphi_1,\vphi_2)$ in the optimal transport of $\rho_1^*$ onto $\rho_1^n$ and $\rho_2^*$ onto $\rho_2^n$, respectively. 
\item[(ii)] there exists a function $p:\Om\to\R$ that is Lipschitz continuous on $\spt(\rho^*_i)$, $i=1,2$ and is such that for any probability densities $\rho_1, \rho_2$ with $\rho_1+\rho_2=1$ a.e. we have
\begin{equation}\label{ineq:pressure_optimality}
\int_{\Om}\left[K_\e\star\rho_2^*+\Phi_1+b_1\vphi_1/\t +p \right](\rho_1-\rho_1^*)\dd x+\int_{\Om}\left[K_\e\star\rho_1^*+\Phi_2+b_2\vphi_2/\t  +p \right](\rho_2-\rho_2^*)\dd x\ge 0,
\end{equation}
moreover, we have 
\begin{equation}\label{eq:pressure-gradient}
\begin{array}{l}
\nabla p=-\nabla K_\e\star\rho_2^*-\nabla\Phi_1-b_1\nabla\vphi_1/\t\ \ {\rm a.e.\ in \ spt}(\rho_1^*)\\[5pt]
 {\rm{and}}\\[5pt]
\nabla p=-\nabla K_\e\star\rho_1^*-\nabla\Phi_2-b_2\nabla\vphi_2/\t\ \ {\rm a.e.\ in \ spt}(\rho_2^*).
\end{array}
\end{equation}
\end{itemize}
\end{lemma}

\begin{proof}
Let $(\rho_1,\rho_2)$ be a pair of probability measures such that $\rho_1+\rho_2=1$ a.e. For $\d\in[0,1]$ let us consider the competitors $(\rho_1^\d,\rho_2^\d):=(\rho_1^*+\d(\rho_1-\rho_1^*),\rho_2^*+\d(\rho_2-\rho_2^*))$, which by construction satisfy the constraint. 

By optimality, we have 
\begin{align*}
&\HC_\e(\rho_1^\d,\rho_2^\d)-\HC_\e(\rho_1^*,\rho_2^*)+\Phi(\rho_1^\d,\rho_2^\d)-\Phi(\rho_1^*,\rho_2^*)\\
&+\frac{1}{2\t}\left(W_2^2(\rho_1^\d,\rho_1^n)-W_2^2(\rho_1^*,\rho_1^n)+W_2^2(\rho_2^\d,\rho_2^n)-W_2^2(\rho_2^*,\rho_2^n)\right)\ge 0.
\end{align*}

Using the exact same argument as in \cite[Lemma 3.1]{MauRouSan1} to develop the Wasserstein part on the one hand and the first variations of $\HC_\e$ and $\Phi$ on the other hand, we find \eqref{ineq:optimality}.

For (ii) (similarly as in \cite[Proposition 4.7]{Lab17}), let us notice first that \eqref{ineq:optimality} can be written in the form  
$$\int_{\Om}\left[K_\e\star\rho_2^*+\Phi_1+b_1\vphi_1/\t\right]h_1\dd x+\int_{\Om}\left[K_\e\star\rho_1^*+\Phi_2+b_2\vphi_2/\t\right]h_2\dd x\ge 0,$$
where $(h_1,h_2)\in L^\infty(\Om)\times L^\infty(\Om)$ is such that $\rho_1^*+\d h_1+\rho_2^*+\d h_2=1$ and $\rho_i^*+\d h_i\in [0,1]$ for $i=1, 2$.  We know from Proposition \ref{prop:characteristic} that $\rho_1^*, \rho_2^*$ forms a partition of $\Omega$. Therefore, we must take $h_1\leq 0$ on $\spt(\rho_1^*)$, and $h_1\geq 0$ on $\spt(\rho_2^*)$ (and vice-versa for $h_2$).   To preserve the constraint  $\rho_1^*+\d h_1+\rho_2^*+\d h_2=1$ and the mass of each density, we must also take  $h_1+h_2=0$ a.e. and $\int_{\Om}h_1\dd x=\int_{\Om}h_2\dd x=0$.

Now if we set $h_2=-h_1$, we find that 
$$\int_{\Om}\left[K_\e\star\rho_2^*+\Phi_1+b_1\vphi_1/\t-\Big(K_\e\star\rho_1^*+\Phi_2+b_2\vphi_2/\t \Big) \right]h_1\dd x \geq 0,$$
for any $h_1\in L^\infty(\Om)$ with 0 mean such that $h_1\leq 0$ a.e. on $\spt(\rho_1)$. This implies that there exist constants $C_1, C_2\in\R$ such that 
$$K_\e\star\rho_2^*+\Phi_1+b_1\vphi_1/\t-C_1 \leq K_\e\star\rho_1^*+\Phi_2+b_2\vphi_2/\t -C_2\ \ {\rm a.e.}\;  \spt(\rho_1)$$
$$K_\e\star\rho_1^*+\Phi_2+b_2\vphi_2/\t-C_2\leq K_\e\star\rho_2^*+\Phi_1+b_1\vphi_1/\t-C_1 \ \ {\rm a.e.}\; \spt(\rho_2)$$
From here, we can define the pressure variable as
\begin{equation}
p:=C_1-K_\e\star\rho_2^*-\Phi_1-b_1\vphi_1/\t  \quad \textrm{a.e. on}\ \spt(\rho_1^*)\ \ {\rm{and}}\ \ p:=C_2-K_\e\star\rho_1^*-\Phi_2-b_2\vphi_2/\t  \quad \textrm{a.e. on}\ \spt(\rho_2^*).
\end{equation}

Since the Kantorovich potentials are Lipschitz continuous and the other terms are smooth, we find that $p$ is Lipschitz continuous (note this regularity may degenerate as $\t\da 0$). 
By construction, $p$ clearly satisfies the inequality in \eqref{ineq:pressure_optimality}, and in particular the value of the l.h.s. is equal to zero. 
Since the functions under consideration are all Lipschitz continuous, we obtain \eqref{eq:pressure-gradient}.
\end{proof}

\subsection{Continuous in time solutions for $\ve>0$ fixed}

Now we are ready to begin constructing time interpolations from the discrete scheme.  In this section we restrict ourselves to the case where $\epsilon$ is held fixed.  In this special case, we can use standard arguments from the theory of Wasserstein gradient flows to obtain continuous in time equations in the limit $\tau \da 0$.  When $\epsilon$ is held fixed, we have strong compactness on the pressure term. However, similarly to the models in \cite{KimMes,Lab17}, we will lack strong compactness on the density variables, which would mean that in the limit when $\t\da 0$, only a weaker version of the system will be available (see \eqref{eq:eps-fixed}).  Let us also note that later on in Section \ref{sec:muskat}, we will need some of the pressure estimates provided below when we take $\epsilon$ to zero along with $\tau$. 

Let $T>0$ be a given time horizon and $N\in\N$ and $\t>0$ such that $N\t=T.$ Let $\e>0$ be fixed. By now, it is standard how to construct weak solutions to PDEs that have gradient flow structure, using  minimizing movement schemes as
\begin{equation} \label{eq:mm}
\begin{array}{l}
\ds\left({\rho}^{n+1, \tau}_1,{\rho}^{n+1, \tau}_2\right)\\
\ds\in\argmin \left\{ \HC_{\epsilon}(\rho_1,\rho_2)+\Phi(\rho_1,\rho_2)+\frac{b_1}{2\tau}W_2^2(\rho_1, \rho_1^n)+\frac{b_2}{2\tau}W_2^2(\rho_2, \rho_2^n):\ \rho_1,\rho_2\in\sP^{\ac}(\Om), \rho_1+\rho_2=1\ 
\ae\right\}.
\end{array}
\end{equation}
Let $\rho_i^\t:[0,T]\to\cP(\Om)$ be defined as
\begin{equation}\label{def:rho-tau}
\rho_i^\t(t):=\rho_i^{n+1},\ {{\rm if}}\ t\in[n\t,(n+1)\t).
\end{equation}
We define the corresponding velocities, pressures and momentum variables as 
\begin{equation}\label{def:vpE-tau}
\begin{array}{l}
v_i^\t(t,\cdot):=\frac{\nabla\vphi_i^{n+1}}{\t},\ {{\rm if}}\ t\in[n\t,(n+1)\t),\\[5pt]
p_i^\t(t,\cdot):=p_i^{n+1},\ {{\rm if}}\ t\in[n\t,(n+1)\t),\\[5pt]
E_i^\t(t,\cdot):=v_i^\t(t,\cdot)\rho_i^\t,\ {{\rm if}}\ t\in[n\t,(n+1)\t),
\end{array}
\end{equation}
where $\vphi_i^{n+1}$ and $p_i^{n+1}$ are defined in Lemma \ref{lem:opt_cond}. Following the same steps as in \cite{KimMes,MauRouSan1},  the analysis boils down to obtain a sufficient amount of uniform (in $\tau$) estimates and compactness for the previously obtained functions, then pass to the limit $\t\da 0.$ 

Also, as additional tools we construct the corresponding continuous in time (geodesic) interpolations between the densities and the corresponding velocities and momentum variables, $(\tilde\rho_i^\t,\tilde v_i^\t,\tilde E_i^\t)$. We refer to \cite[Section 3.1]{KimMes} (see also \cite{MauRouSan1,OTAM}) for the precise construction. In particular, as a consequence of this construction, we have 
$$\partial_t\tilde\rho_i^\t+\nabla\cdot \tilde E_i^\t=0,$$
in the sense of distribution on $(0,T)\times\Om$. Let us comment on the role of the geodesic interpolations. These interpolations (beside the more standard piecewise constant ones) in the context of $W_2$--gradient flows were first used by Santambrogio (see the discussion in \cite[Section 8.3]{OTAM}). Their role is the following: for any $\t$, these interpolations, since pieces of geodesic curves, by definition solve continuity equations. Since the piecewise constant interpolations match the geodesic ones at node points $n\tau$, if both of them converge, they need to converge to the same limit. Therefore, one automatically has a continuity equation as the limit of piecewise constant interpolations. An alternative way (which is more often used in the literature) would be to say that the piecewise constant interpolations solve a continuity equation up to an error term, then one would need to show that this error term is converging to zero as the time discretization parameter tends to zero.

Based on the same techniques as in \cite{MauRouSan1,KimMes}, it is easy to obtain the following estimates.

\begin{lemma}\label{lem:estimates} 
Let $(\rho_i^\t, v_i^\t,E_i^\t)$ and $(\tilde\rho_i^\t,\tilde v_i^\t,\tilde E_i^\t)$ be the previously constructed piecewise constant and continuous in time interpolations, respectively. Then there exists $C>0$ independent of $\t>0$ and depending only on $\HC_\e(\rho_{1,0},\rho_{2,0})$ such that 
\begin{itemize}
\item[(i)]$W_2(\rho_{i}^\t(t),\rho_{i}^\t(s))\le C\sqrt{t-s+\t}$ and $W_2(\tilde\rho_{i}^\t(t),\tilde\rho_{i}^\t(s))\le C\sqrt{t-s}$  for any $0\le s\le t\le T.$ Moreover, up to passing to subsequences $(\rho_i^\t)_{\t>0}$ and $(\tilde\rho_i^\t)_{\t>0}$ converge (uniformly with respect to $W_2$) as $\t\da 0$ to the same limit.
\item[(ii)] $(v_i^\t)_{\t>0}$ is uniformly bounded in $L^2([0,T]; L^2_{\rho_i^\t}).$
\item[(iii)] $(E_i^\t)_{\t>0}$ and $(\tilde E_i^\t)_{\t>0}$ are uniformly bounded in $\sM^d([0,T]\times\Om)$ and up to passing to subsequences they have the same distributional limits as $\t\da 0$.
\item[(iv)] $\ds\int_0^T\int_{\R^d}|\nabla G_\e\star\rho_i^\t|\dd x\dd t\le CT\,\HC_{\e}(\rho_{1,0},\rho_{2,0}).$
\item[(v)] $\ds\int_0^T\int_{\Om}|\rho_i^\t(t,x+\d d)-\rho_i^\t(t,x)|\dd x\dd t\le CT(\d+\sqrt{\e})\HC_{\e}(\rho_{1,0},\rho_{2,0})$ for any $d\in\R^d$ with $|d|=1$ and any $\d>0$.
\item[(vi)] $\ds \int_0^T\int_{\Om}|\nabla p^\t|^2\dd x\dd t\le C\int_0^T\int_{\Om}\frac{1}{\e^2}\left(\left(G_{2\e}\star\rho_2^\t\right)^2\rho_1^\t+\left(G_{2\e}\star\rho_1^\t\right)^2\rho_1^\t\right)\dd x\dd t + C$, and as a consequence,
$(\nabla p^\t)_{\t>0}$ is uniformly bounded in $L^2([0,T]\times\Om;\R^d)$ and $\left(p^\t-\fint_{\Om}p^\t(\cdot,x)\dd x\right)_{\t>0}$ is uniformly bounded in $L^2([0,T]\times\Om)$ by a constant of the form $TC(1+1/\e^2).$
\item[(vii)] $(\nabla p^\t)_{\t>0}$ is uniformly bounded in $L^2([0,T]; (C^1(\Om))^*),$ independently of $\t$ and $\e$. In particular, $\nabla p^\t$ is uniformly bounded in $\sD'((0,T)\times\Om;\R^d)$ and  $\nabla p^\t(t,\cdot)$ defines a uniformly bounded distribution of order one, for a.e. $t\in[0,T]$.
\end{itemize}
\end{lemma}

\begin{proof}
Let us notice that the proofs of point (i), (ii) and (iii) follow the exact same lines of the proofs of \cite[Lemma 3.3]{MauRouSan1} and \cite[Lemma 3.3]{KimMes}, so we omit them.

From Proposition \ref{prop:characteristic} we know that that $\rho^n_i$'s are characteristic functions, therefore the proofs of (iv) and (v) follow the exact same lines as the proof of \cite[Lemma 2.4.]{LauOtt}.

Let us give the details on (vi). From the identities \eqref{eq:pressure-gradient} from Lemma \ref{lem:opt_cond}, we have that there exists $C>0$ (independent of $\t>0$, which might increase from one inequality to the next) such that
\begin{align*}
\int_0^T\int_{\Om}|\nabla p^\t|^2(\rho_1^\t+\rho_2^\t)\dd x\dd t&\le C\sum_{n=1}^N\t\int_{\Om}\left((|\nabla K_\e\star\rho_2^n|^2+|\nabla\Phi_1|^2)\rho_1^n+(|\nabla K_\e\star\rho_1^n|^2+|\nabla\Phi_2|^2)\rho_2^n\right)\dd x\\
&+\frac{C\t}{\t^2}\sum_{n=1}^N\sum_{i=1}^2 b_iW_2^2(\rho_i^n,\rho_i^{n-1})\\
&\le C\int_0^T\int_{\Om}\frac{1}{\e^2}\left(\left(G_{2\e}\star\rho_2^\t\right)^2\rho_1^\t+\left(G_{2\e}\star\rho_1^\t\right)^2\rho_1^\t\right)\dd x\dd t + C.
\end{align*}
where we have used the uniform bounds on $(\rho_i^\t)_{\t>0}$ from the previous points and 
\begin{align*}
|\nabla K_\e\star\rho_i^n|^2=\frac{2\pi\sigma}{\e}\left(|\nabla G_\e|\star\rho_i^n\right)^2\le\frac{2\pi\sigma}{\e}\left(\frac{1}{\sqrt{\e}} G_{2\e}\star\rho_i^n\right)^2=\frac{2\pi\sigma}{\e^2}\left(G_{2\e}\star\rho_i^n\right)^2
\end{align*}
Since $\rho_1^n+\rho_2^n=1$ a.e., this previous bound implies that $(\nabla p^\t)_{\t>0}$ is uniformly bounded in $L^2([0,T]\times\Om;\R^d).$ As a consequence of Poincar\'e's inequality, we have that $\left(p^\t-\fint_{\Om}p^\t(\cdot,x)\dd x\right)_{\t>0}$ is uniformly bounded in $L^2([0,T]\times\Om)$. Both uniform bounds have the form $TC(1+1/\e^2).$ 

\medskip

To show (vii), let $\xi\in C^1([0,T]\times\Om;\R^d)$. Fix $t\in[n\t,(n+1)\t)$. Taking the inner product of both sides in \eqref{eq:pressure-gradient} with $\xi(t,\cdot)$ and integrating on $\Om$ w.r.t. $\rho_i^{n+1}$ (we drop the dependence on $t$ in the notation of $\xi$), we obtain
\begin{equation}\label{eq:nab_p1}
\int_{\Om}\nabla p^{n+1}\cdot\xi\rho_1^{n+1}\dd x=-\sigma\sqrt{2\pi}\int_{\Om}\frac{1}{\sqrt{\e}}\nabla G_\e\star\rho_2^{n+1}\cdot \xi\rho_1^{n+1}\dd x-\int_{\Om}\nabla\Phi_1\cdot \xi\rho_1^{n+1}\dd x - \int_{\Om}\frac{b_1}{\t}\nabla\vphi_1^{n+1}\cdot \xi \rho_1^{n+1}\dd x
\end{equation}
and interchanging the roles of $\rho_1^{n+1}$ and $\rho_2^{n+1}$, we get
\begin{equation}\label{eq:nab_p2}
\int_{\Om}\nabla p^{n+1}\cdot\xi\rho_2^{n+1}\dd x=-\sigma\sqrt{2\pi}\int_{\Om}\frac{1}{\sqrt{\e}}\nabla G_\e\star\rho_1^{n+1}\cdot \xi\rho_2^{n+1}\dd x -\int_{\Om}\nabla\Phi_2\cdot \xi\rho_2^{n+1}\dd x- \int_{\Om}\frac{b_2}{\t}\nabla\vphi_2^{n+1}\cdot \xi \rho_2^{n+1}\dd x
\end{equation}
First, we have
\begin{align*}
&\int_{\Om}\frac{1}{\sqrt{\e}}\nabla G_\e\star\rho_2^{n+1}\cdot \xi\rho_1^{n+1}\dd x+\int_{\Om}\frac{1}{\sqrt{\e}}\nabla G_\e\star\rho_1^{n+1}\cdot \xi\rho_2^{n+1}\dd x\\
&=\frac{1}{2\e}\int_{\R^d}\int_{\Om}z G(z)\left[\rho_2^{n+1}(x+\sqrt{\e}z)\cdot\xi(x)\rho_1^{n+1}(x)+\rho_1^{n+1}(x+\sqrt{\e}z)\cdot\xi(x)\rho_2^{n+1}(x)\right]\dd x\dd z\\
&=\frac{1}{2\e}\int_{\R^d}\int_{\Om}z G(z)\left[\rho_2^{n+1}(x+\sqrt{\e}z)\cdot\xi(x)\rho_1^{n+1}(x)+\rho_1^{n+1}(x)\cdot\xi(x-\sqrt{\e}z)\rho_2^{n+1}(x-\sqrt{\e}z)\right]\dd x\dd z\\
&=\frac{1}{2\e}\int_{\Om}\int_{\R^d}z G(z)\left[\rho_2^{n+1}(x+\sqrt{\e}z)\cdot\xi(x)\rho_1^{n+1}(x)-\rho_1^{n+1}(x)\cdot\xi(x+\sqrt{\e}z)\rho_2^{n+1}(x+\sqrt{\e}z)\right]\dd z\dd x\\
&=\frac{1}{2\e}\int_{\Om}\int_{\R^d}z G(z)\rho_2^{n+1}(x+\sqrt{\e}z)\rho_1^{n+1}(x)\cdot\left[\xi(x)-\xi(x+\sqrt{\e}z)\right]\dd z\dd x\\
&=\frac{1}{2\sqrt{\e}}\int_{\Om}\int_{\R^d}\int_1^0 z G(z)\rho_2^{n+1}(x+\sqrt{\e}z)\rho_1^{n+1}(x)\cdot (D\xi(x+s\sqrt{\e}z)z)\dd s\dd z\dd x
\end{align*}
where in the second and third equalities we have used the change of variables $x\mapsto x-\sqrt{\e}z$ and $z\mapsto -z$, respectively and the fundamental theorem of calculus in the last equality. Now, since 
$$|z|^2 G(z)\le \a G(z/\beta),\ \forall z\in\R^d,$$
for some universal constants $\a,\b>0$, by the previous chain of equalities, we obtain
\begin{align*}
\Bigg{|}\int_{\Om}\frac{1}{\sqrt{\e}}\nabla G_\e\star\rho_2^{n+1}\cdot \xi\rho_1^{n+1}\dd x+\int_{\Om}\frac{1}{\sqrt{\e}}\nabla G_\e\star\rho_1^{n+1}\cdot \xi\rho_2^{n+1}\dd x\Bigg{|}&\le C\HC_{2\e}(\rho_1^{n+1},\rho_2^{n+1})\|D\xi(t,\cdot)\|_{L^\infty}\\
&\le C\HC_{\e}(\rho_1^0,\rho_2^0)\|D\xi(t,\cdot)\|_{L^\infty},
\end{align*}
where, in the last inequality we used the monotonicity of $\HC_\e$  along the sequence $\left(\rho_1^n,\rho_2^n\right)_n$.

Furthermore, we have that
\begin{align*}
\int_{\Om}\frac{1}{\t}|\nabla\vphi_i^{n+1}| |\xi|\rho_i^{n+1}\dd x=\int_{\Om}|v_i^{n+1}|\rho_i^{n+1}|\xi|\dd x
\end{align*}
and
\begin{align*}
\Bigg{|}\int_{\Om}\nabla\Phi_1\cdot \xi\rho_1^{n+1}\dd x+\int_{\Om}\nabla\Phi_1\cdot \xi\rho_1^{n+1}\dd x\Bigg{|}\le C\|\xi\|_{L^1}
\end{align*}
Using the fact that we have piecewise constant interpolations, integrating the last equality in time on $[0,T]$, we get
$$\int_0^T\int_{\Om}|v_i^\t|\rho_i^\t|\xi|\dd x\dd t \le \left(\int_0^T\int_{\Om}|v_i^\t|^2\rho_i^\t\dd x\dd t\right)^{\frac12}\left(\int_0^T\int_{\Om}|\xi|^2\dd x\dd t\right)^{\frac{1}{2}}$$

Now, adding together \eqref{eq:nab_p1} and \eqref{eq:nab_p2}, then integrating in time on $[0,T]$ and using (ii), we obtain
\begin{align*}
\Bigg{|}\int_0^T\int_{\Om}\nabla p^\t\cdot\xi\dd x\dd t\Bigg{|}\le C\|D\xi\|_{L^1(C^0)}+ C\|\xi\|_{L^2(L^2)}\le C\|\xi\|_{L^2(C^{1})},
\end{align*}
where the constant $C>0$ is independent of $\t>0$ and $\e>0$. The thesis of follows.
\end{proof}

Now we can use the above pressure estimates to derive a system of continuous in time equations in the limit $\tau\to 0$ when $\epsilon$ is held fixed.

\begin{theorem}\label{thm:e_fixed_equations}
Let $\e>0$, $\rho_{1,0},\rho_{2,0}\in\sP(\Om)$ such that $\rho_{1,0}+\rho_{2,0}=1$ a.e. There exists $\rho_i\in{\rm{AC}}^2([0,T];\sP(\Om))$, $i=1,2$, $\ov p\in L^2([0,T];H^1(\Om))$ and $\zeta_1,\zeta_2\in L^2([0,T]\times\Om;\R^d)$ such that $\rho_1+\rho_2=1$ a.e. in $[0,T]\times\Om$, $b_1\zeta_1+b_2\zeta_2=\nabla\ov p$ a.e. in $[0,T]\times\Om$ and the system
\begin{equation}\label{eq:eps-fixed}
\left\{
\begin{array}{ll}
\partial_t\rho_1-\nabla\cdot\left[b_1^{-1}\rho_1(\nabla K_\e\star\rho_2+\nabla\Phi_1)+\zeta_1\right]=0, & (0,T)\times\Om,\\[5pt]
\partial_t\rho_2-\nabla\cdot\left[b_2^{-1}\rho_2(\nabla K_\e\star\rho_1+\nabla\Phi_2)+\zeta_2\right]=0, & (0,T)\times\Om,\\[5pt]
\rho_1(0,\cdot)=\rho_{1,0}, \quad \rho_2(0,\cdot)=\rho_{2,0}, & \Om.
\end{array}
\right.
\end{equation}
is satisfied in the sense of distributions on $(0,T)\times\Om.$
\end{theorem}
It remains open whether in the previous theorem we have strong convergence $\rho_i^\t\to\rho_i$ in $L^2([0,T]\times\Om)$ as $\t\da 0$, and in particular whether one can claim that $\zeta_i=b_i^{-1}\rho_i\nabla\ov p.$

\begin{proof}[Proof of Theorem \ref{thm:e_fixed_equations}]
Using the uniform (in $\t$) estimates in Lemma \ref{lem:estimates}, we have the existence of $\rho_i\in{\rm{AC}}^2([0,T];\sP(\Om))$, $i=1,2$ such that up to passing to a subsequence, both $(\rho_1^\t,\rho_2^\t)$ and $(\tilde\rho_1^\t,\tilde\rho_2^\t)$ converge to them uniformly in time, w.r.t. $W_2$. 

Since $\left(p^\t-\fint_{\Om}p^\t(\cdot,x)\dd x\right)_{\t>0}$ is uniformly bounded in $L^2([0,T];H^1(\Om))$, there exists $\ov p\in L^2([0,T];H^1(\Om))$ such that up to passing to a subsequence, $p^\t-\fint_{\Om}p^\t(\cdot,x)\dd x\weakly \ov p$, weakly in $L^2([0,T]\times\Om)$ and $\nabla p^\t\weakly \nabla \ov p$, weakly in $L^2([0,T]\times\Om;\R^d)$ as $\t\da 0$.

We only need to identify the limits of the momentum variables $(E_i^\t)_{\t>0}$. From the weak convergence of the density variables, we have that up to passing to a subsequence,
$\rho_1^\t\nabla K_\e\star\rho_2^\t\weaklys\rho_1\nabla K_\e\star\rho_2$, as $\t\da 0$ and similarly 
$\rho_2^\t\nabla K_\e\star\rho_1^\t\weaklys\rho_2\nabla K_\e\star\rho_1$, and $\rho_i^\t\nabla\Phi_i\weaklys \rho_i\nabla\Phi_i$, as $\t\da 0$ as vector measures.

Furthermore, since $(\rho_i^\t\nabla p^\t)_{\t>0}$ is uniformly bounded in $L^2([0,T]\times\Om;\R^d)$, there exists $\zeta_i\in L^2([0,T]\times\Om;\R^d)$, $i=1,2$ such that up to passing to a subsequence $b_i^{-1}\rho_i^\t\nabla p^\t\weakly \zeta_i$ weakly in $L^2([0,T]\times\Om;\R^d)$, as $\t\da 0$.

Last, by the previous arguments, we have 
$$b_1 b_1^{-1}\rho_1^\t\nabla p^\t+b_2 b_2^{-1}\rho_2^\t\nabla p^\t=\nabla p^\t\weakly \nabla \ov p=b_1\zeta_1+b_2\zeta_2,$$
as $\t\da 0$ in $L^2([0,T]\times\Om;\R^d).$ Therefore, the thesis of the theorem follows.

\end{proof}
It is open whether the continuum in time densities in Theorem \ref{thm:e_fixed_equations} are characteristic functions.  Indeed, since we can only guarantee that the discrete densities converge weakly to the continuum densities, the characteristic function property may be lost in the limit.   We point out that due to the energy bounds, the densities are ``almost"  characteristic functions, i.e. we have

\begin{lemma}
For $(\rho_1^\t,\rho_2^\t)$ given as above with $\e>0$ fixed, for any $\a\in(0,1)$ we have
$$
\sL^d( \alpha \leq \rho_i^\t \leq 1-\alpha) \leq \frac{3\sqrt{\e}}{\sigma{\sqrt{2\pi}} \alpha(1-\alpha)}\hc_{\e}(\rho_1^{\tau}, \rho_2^{\tau}) \hbox{ for all } \t.
$$
\end{lemma}

\begin{proof}
Let us set $i=1$, the other case will be parallel. We have
$$\sL^d(\alpha \leq \rho_1^\t \leq 1-\alpha)  \leq \frac{1}{\alpha(1-\alpha)}\int_{\Om} \rho^{\tau}_1(x)(1-\rho^{\t}_1(x))\dd x=\frac{1}{\alpha(1-\alpha)}\int_{\Om}\left[\rho^{\tau}_1(x)-\rho^{\t}_1(x)^2\right]\dd x.$$
By Jensen's inequality, the above is smaller than 
\begin{align*}
&\frac{1}{\alpha(1-\alpha)}\int_{\Om}\left[ \rho^{\t}_1(x)-(G_{\epsilon}\star\rho^{\t}_1)(x)^2\right]\dd x\\
&\leq \frac{\sqrt{\e}}{\sigma{\sqrt{2\pi}}\alpha(1-\alpha)}\hc_{\e}(\rho_1^{\tau}, \rho_2^{\tau})+\frac{1}{\alpha(1-\alpha)}\int_{\Om} (G_{\epsilon}\star\rho^{\tau}_1)(x)\Big(\rho^{\t}_1(x)-(G_{\epsilon}\star\rho^{\t}_1)(x)\Big)\dd x.
\end{align*}
We then estimate 
$$\frac{1}{\alpha(1-\alpha)}\int_{\Om} (G_{\epsilon}\star\rho^{\tau}_1)(x)\Big(\rho^{\t}_1(x)-(G_{\epsilon}\star\rho^{\t}_1)(x)\Big)\dd x\leq \frac{1}{\alpha(1-\alpha)}\int_{\Om} \int_{\RR^d} G(z) |\rho_1^{\t}(x)-\rho_1^{\t}(x-\sqrt{\e} z)|\dd z\dd x $$
Since $\rho^{\t}_1$ and  $\rho^{\t}_2$ take values in $\{0,1\}$, we have $  |\rho_1^{\t}(x)-\rho_1^{\t}(x-\sqrt{\e} z)|\leq \rho^{\t}_1(x)\rho^{\t}_2(x-\sqrt{\e}z)+\rho^{\t}_1(x-\sqrt{\e}z)\rho^{\t}_2(x)$.  Applying these inequalities, the result follows. 
\end{proof}

\section{Muskat flow with surface tension}\label{sec:muskat}
In this section, we complete the proof of Theorem \ref{thm:main} and show that  when $\epsilon=\epsilon_{\t}$ goes to zero along with $\tau$, the time interpolated minimizing movements scheme constructed in 
\eqref{def:rho-tau}-\eqref{def:vpE-tau} converges to a weak formulation of the Muskat problem with surface tension under the energy convergence assumption (\ref{assumption}).  This amounts to showing that each quantity in the system  of Euler-Lagrange equations given in (\ref{eq:pressure-gradient}) converges (in an appropriate sense) to the correct limiting object.  In particular, we will need strong $L^1$ convergence of the density functions in $[0,T]\times \Omega$.  To obtain the necessary compactness for strong $L^1$ convergence,  we develop an adaptation of an Aubin-Lions type lemma in Proposition \ref{prop:strong_limit}.  We then conclude our result by verifying the convergence of the Euler-Lagrange equations to the weak formulation of the Muskat problem in a similar manner to the approach in \cite{LauOtt}.

Before we introduce the weak formulation of the Muskat problem, let us recall the classical formulation of the problem.  When the two phases are separated by a smooth interface $\Gamma$, the Muskat problem is given by the continuity equation 
\begin{equation}\label{eq:continuity_repeated}\partial_t\rho_i - b^{-1}_i\nabla\cdot \big((\nabla p_i + \nabla \Phi_i)\rho_i\big) = 0\tag{MP$_1$}
\end{equation}
along with the free boundary problem for the pressure
\begin{equation}\label{eq:muskat_repeated}
\left\{\begin{array}{lll}
-\Delta p_i = \Delta \Phi_i &\hbox{ in } &\spt(\rho_i);\\ \\
\partial_n  (p_i + \Phi_i) =0 &\hbox{ on } &\partial\Omega; \\ \\
V = b_1^{-1}\partial_{n} (p_1 + \Phi_1) = b_2^{-1}\partial_{n}(p_2+\Phi_2) &\hbox{ on }& \Gamma ;\\ \\
\left[p\right]:=(p_1-p_2) = \frac{\sigma}{2}\kappa & \hbox{ on } &\Gamma,\\ \\
\tilde{n}=n & \hbox{on} & \partial \Gamma\cap \partial\Omega;\\
\end{array}\right.\tag{MP$_2$}
\end{equation}
\color{black}
where $\kappa$ is the mean curvature of $\Gamma$, oriented to be positive when $\spt(\rho_2)$ is convex at the point. The weak formulation of the Muskat problem with surface tension is provided in Definition \ref{def:weak_sol} below.  
 \color{black}
\begin{definition}\label{def:weak_sol}
We say that $(\rho_i,v_i,p)$ is a weak solution to the Muskat problem with surface tension, if for a.e. $T>0$, $\rho_i\in L^1([0,T];BV(\Om;\{0,1\}))\cap{\rm{AC}}^2([0,T];\sP(\Om))$, $\rho_1\rho_2=0$ and $\rho_1+\rho_2=1$ a.e., $v_i\in L^2([0,T];L^2_{\rho_{i,t}}(\Om;\R^d))$, $p\in L^2([0,T]; (C^{0,\a}(\Om))^*)$ and 
for any $\psi\in C^\infty([0,T]\times\ov\Om)$ and for any  vector field $\xi\in C^\infty([0,T]\times\Om;\R^d)$ with zero normal component on $\partial\Omega$,  we have
\begin{equation}\label{eq:continuity}
\int_0^T\int_{\Om} \left(\rho_i v_i  \cdot \nabla \psi + \rho_i \partial_t\psi \right)\dd x \dd t = \left[\int_{\Om}\rho_i \psi \dd x \right]^T_0
\end{equation}
with 
\begin{equation}\label{eq:goal}   
\int_0^T\int_{\Om} \sum_i \rho_i(b_iv_i+\nabla\Phi_i)\cdot \xi - p\nabla \cdot \xi + \frac{\sigma}{2}\left( \frac{D\rho_1}{|D\rho_1|}\otimes \frac{D\rho_1}{|D\rho_1|} : \nabla \xi  + \nabla \cdot \xi \right) \left(|D \rho_1| +|D \rho_2|  \right)\dd x\dd t =0 
\end{equation} 
\end{definition}

\begin{remark}
Let us remark here that $\frac{D\rho_1}{|D\rho_1|}$ stands for the $L^\infty$ density of $D\rho_1$ with respect to the total variation measure $|D\rho_1|$ (after using the Radon-Nikodym differentiation). Furthermore, let us notice that the term $\left( \frac{D\rho_1}{|D\rho_1|}\otimes \frac{D\rho_1}{|D\rho_1|} : \nabla \xi \right) \left(|D \rho_1| +|D \rho_2|  \right)$ is meaningful in the sense that $f(\nu/|\nu|)\dd|\nu|$ defines a matrix valued element of $\sM(\Om)$ for any $f:\R^d\to\R^{d\times d}$ which is continuous and 1-homogeneous (see for instance \cite[Proposition 3.15]{AmbFusPal} in the case when $\nu=D\rho$ for some $\rho\in BV(\Om)$). In particular, if $\rho\in BV(\Om;\{0,1\})$, then $D\rho/|D\rho|$ stands for the measure theoretic normal to the boundary of the set $\spt(\rho)$ and by $\nu/|\nu|$ we mean the density of $\nu$ with respect to its total variation measure $|\nu|$. 
\end{remark}
\begin{remark}\label{rmk:boundary}
Although we restrict our attention in the weak curvature equation (\ref{eq:goal}) to test functions with vanishing normal component on $\partial\Omega$, we do not lose information at the boundary.  The weak continuity equation (\ref{eq:continuity}) encodes the zero normal boundary condition for the velocities, and (\ref{eq:goal}) still encodes the condition that $\Gamma$ and $\partial\Omega$ meet orthogonally (c.f. the proof of Lemma \ref{lem:consistency}).
\end{remark}

Now we are ready to state the main result of our paper.
\begin{theorem}\label{thm:MAIN}
Given initial data $\rho_{1,0},\rho_{2,0}\in BV(\Om)$ such that $\rho_{1,0}\rho_{2,0}=0$ and $\rho_{1,0}+\rho_{2,0}=1$ a.e. in $\Om$ and Lipschitz continuous potential functions $\Phi_1,\Phi_2:\Om\to\R$, there exists $(\rho_i,v_i, p)$ with $i=1,2$ such that under the energy convergence assumption \eqref{assumption},  it solves \eqref{eq:continuity_repeated}-\eqref{eq:muskat_repeated} in the sense of Definition \ref{def:weak_sol}.
\end{theorem}

We postpone the proof of the previous theorem to the end of this section. First, let us show a consistency result, i.e. that classical solutions of the Muskat problem satisfy the weak formulation \eqref{eq:continuity} and \eqref{eq:goal}.
\begin{lemma}\label{lem:consistency}
If smooth solutions of \eqref{eq:continuity_repeated}and \eqref{eq:muskat_repeated} exist with $\rho_i = \chi_{A_i}$ and with a $C^2$ hypersurface $\Gamma = \partial A_1 = \partial A_2$, then $\rho_i$ satisfies  \eqref{eq:continuity}  and \eqref{eq:goal} with the choice of 
\begin{equation}\label{equation00}
v_i := -b_i^{-1}(\nabla p_i+\nabla\Phi_i)  \hbox{ and } p = p_1\rho_1 + p_2\rho_2 + \sigma\dd \mu,
\end{equation}
where $\mu\in (C([0,T]\times \Om))^*$ is the surface measure of $\Gamma=\cup_{t>0}(\Gamma_t\times \{t\})$, i.e. after disintegration
$$\mu=\sH^{d-1}\mres\Gamma_t\otimes\sL^1\mres[0,T]$$ or using test functions
$$\int_0^T \int_{\Om} f \dd \mu = \int_0^T\int_{\Gamma_t} f \dd\sH^{d-1}\dd t \hbox{ for any } f\in C([0,T]\times \Om).
$$
\end{lemma}

\begin{remark}
Note that our notion of solution requires adding the surface measure $\sigma\dd\mu$ to the classical pressure variable. This singular term appears from the minimizing movement scheme, where it ensures that a vacuum does not form at the interface. In general, we expect that the singular part in the weak pressure variable corresponds to the surface measure in \eqref{equation00}.
\end{remark}

\begin{proof}[Proof of Lemma \ref{lem:consistency}]

\eqref{eq:continuity} is a standard weak expression of \eqref{continuity} with $b_iv_i = -\nabla p_i-\nabla\Phi_i$. Let us write again $\Gamma = \cup_{t>0}(\Gamma_t\times \{t\})$. Then we have  

\begin{align*}
\int_0^T\int_{\Om} \left[\sum_i \rho_i(b_iv_i+\nabla\Phi_i)\cdot \xi - p\nabla\cdot \xi\right]\dd x\dd t &=\int_0^T\int_{\Om} \left[\sum_i \rho_i(b_iv_i+\nabla\Phi_i) \cdot \xi - (p_1\rho_1 + p_2\rho_2) \nabla\cdot \xi\right]\dd x\dd t\\
& - \int_0^T\int_{\Om} \sigma (\nabla\cdot \xi )\dd\mu\\ 
&= \int_0^T \int_{\Gamma_t }(p_1-p_2) \xi \cdot \nu\dd\sH^{d-1}\dd t - \sigma \int_0^T\int_{\Gamma_t}\nabla\cdot \xi \dd\sH^{d-1}\dd t \\
&{-\int_0^T\int_{\partial\Om}(p_1\rho_1+p_2\rho_2)\xi\cdot n\dd\sH^{d-1}}\\ 
&= \frac{\sigma}{2}\int_0^T\int_{\Gamma_t} \kappa \xi \cdot\nu \dd\sH^{d-1}\dd t -\sigma \int_0^T\int_{\Gamma_t}\nabla\cdot \xi \dd\sH^{d-1}\dd t\\
&=-\frac{\sigma}{2}\int_0^T \int_{\Gamma_t} [\nabla\cdot \xi  + \nu\cdot \nabla \xi  \nu] \dd\sH^{d-1}
\end{align*}
Here for the second equality we used integration by parts for the first integral, using the fact that 
$$
\nabla \left(\sum_i p_i\rho_i\right)= \sum_i \nabla p_i \rho_i + (p_1-p_2)\nu\dd \sH^{d-1}\mres{\Gamma_t},
$$
and
for the third equality we used the curvature jump condition at the interface, and the fact that $\xi$ has zero normal component on $\partial\Omega$. To conclude, let us recall that $\kappa$ is oriented to be positive when $\spt(\rho_2)$ is convex at the point. With this orientation, observe that for any $C^1$ vector field $\xi$ we have
$$
\int_{\Gamma_t} \kappa \xi\cdot \nu \dd\sH^{d-1} = \int_{\Gamma_t} [\nabla\cdot \xi  -\nu\cdot \nabla \xi  \nu] \dd\sH^{d-1}-\int_{\partial\Gamma_t\cap\partial\Om}\xi\cdot \tilde{n}\dd\sH^{d-2},
$$
where $\nu = D\rho_1/|D\rho_1|$ is normal to $\Gamma$ toward the support of $\rho_1$, $\tilde{n}$ stands for the co-normal vector (orthogonal to $\partial\Gamma_t$ and tangential to $\Gamma_t$).  Note that the lower dimensional term $\int_{\partial\Gamma_t\cap\partial\Om}\xi\cdot \tilde{n}\dd\sH^{d-2}$ must vanish since we know that the co-normal $\tilde{n}$ coincides with the boundary normal $n$.  Indeed, we can write 
\[
\int_{\partial\Gamma_t\cap\partial\Om}\xi\cdot \tilde{n}\dd\sH^{d-2}=\int_{\partial\Gamma_t\cap\partial\Om}\xi\cdot n\dd\sH^{d-2}=0
\] 
where the final equality follows from the fact that $\xi$ has zero normal component on $\partial\Omega$.

\end{proof}

\subsection{Preliminary estimates}

We present below a compactness result on the piecewise constant interpolations of the density variables $(\rho_1^\t,\rho_2^\t)_{\t>0}$, when the parameter $\e$ is vanishing together with $\t.$ 
\begin{proposition}\label{prop:strong_limit}
Let $(\rho_1^\t,\rho_2^\t)_{\t>0}$ be the piecewise constant interpolations constructed in Section \ref{sec:MM} using the minimizing movements scheme \eqref{eq:mm}. Let moreover $\rho_{1,0},\rho_{2,0}\in\sP(\Om)$ be such that $\rho_{i,0}\in BV(\Om;\{0,1\})$ with $\rho_{1,0}+\rho_{2,0}=1$ a.e. in $\Om$ and in particular $\cE(\rho_{1,0},\rho_{2,0})<+\infty.$

Then, there exists $\rho_1,\rho_2\in L^1([0,T];BV(\Om;\{0,1\}))$ such that $\rho_1+\rho_2=1$ a.e. in $[0,T]\times\Om$ and up to passing to a subsequence, $\rho_i^\t\to\rho_i$ strongly in $L^1([0,T]\times\Om)$ as $\e_\t:=\max\{\t,\e\}\da 0$.
\end{proposition}

\begin{proof} First, let us notice that by the uniform quasi-H\"older estimate from Lemma \ref{lem:estimates}(i), we have that there exists a subsequence of $(\rho_i^\t)_{\t>0}$ (that we do not relabel) and $\rho_i\in{\rm{AC}}^2([0,T];\sP(\Om))$ such that $W_2(\rho_{i,t}^\t,\rho_{i,t})\to 0$ as $\e_\t\da 0$, uniformly in $t.$ We shall work with this subsequence from now on.

\medskip

The rest of the proof relies on a careful adaptation of the Aubin-Lions lemma, developed in \cite{RosSav} and used in similar context for instance in \cite{KimMes,Lab17}. 

Let us notice that in order to use the Aubin-Lions argument, we need to show a tightness condition (cf. \cite[Definition 1.3, Remark 1.5]{RosSav}) for the time-dependent family $(\rho_i^\t)_{\t>0}$.  This is typically done by providing a compact set (of the space of probability measures), where `most' of the sequences $(\rho_i^\t(t))_{\t>0}$ lie.  Note that the estimate in Lemma \ref{lem:estimates}(v) does not provide such a compact set. Indeed, for $\tau>0$, the densities $(\rho_i^\t(t))_{\t>0}$ are not actually BV in space. Inspired by arguments in \cite{EO}, we will work first with an auxiliary sequence $\ov\rho_i^\t:[0,T]\times\Om\to[0,1]$, defined as
$$\ov\rho_i^\t:=G_{\e}\star\rho_i^\t,$$
where we have performed a convolution with the heat kernel $G_\e$ only in the spacial variable. It worth noticing that this `perturbation' of $\rho_i^\t$ is reminiscent to the one obtained via the so-called flow interchange technique introduced in \cite{MatMcCSav}. We have the following properties for this new sequence. 

\medskip

{\it Claim 1.} $(\ov\rho_i^\t)_{\t>0}$ is uniformly bounded (w.r.t. $\e$ and $\t$) in $L^1([0,T]; BV(\Om)).$

{\it Proof of Claim 1.} The uniform $L^\infty$ bounds on $(\rho_i^\t)_{\t>0}$ are also preserved for $(\ov\rho_i^\t)_{\t>0}$. Now, as in the proof of \cite[Lemma 2.4]{LauOtt} we conclude that 
\begin{align}\label{eq:BV-estim}
\int_0^T\int_{\Om}|\nabla \ov\rho_{i,t}^\t|\dd x\dd t\le C T\HC_\e(\rho_{1,0},\rho_{2,0}),
\end{align}
for a uniform constant $C$, which proves our claim.

\medskip

{\it Claim 2.} There exists a subsequence of $(\ov\rho_i^\t)_{\t>0}$ (that we do not relabel) and $\ov\rho_i\in L^1([0,T]\times\Om)$ such that $\ov\rho_i^\t\to\ov\rho_i$ strongly in $L^1([0,T]\times\Om)$ as $\e_\t\da 0.$

{\it Proof of Claim 2.} 
For $t\in(0,T)$ fixed, let us notice that $\ov\rho_{i,\t}^\t$ actually can be seen as the evolution of $\rho_{i,t}^\t$ via the heat flow for time $\e>0$. It is well-known that $W_2$ is contractive along the heat flow, so we have 
\begin{equation}\label{eq:contraction}
W_2(\ov\rho_{i,t}^\t,\ov\rho_{i,s}^\t)\le W_2(\rho_{i,t}^\t,\rho_{i,s}^\t),\ \forall\ 0\le s\le t\le T.
\end{equation}

Now let us set $g:L^1(\Om)\times L^1(\Om)\to[0,+\infty]$ and $\cF:L^1(\Om)\to[0,+\infty]$ defined as
$$g(\mu,\nu):=W_2(\mu,\nu)$$
and 
$$\cF(\mu):=\left\{
\begin{array}{ll}
|D\mu|(\Om), & {\rm{if}}\ \mu\in BV(\Om),\\
+\infty, & {\rm{otherwise}}.
\end{array}
\right.$$
By construction, $\cF$ is convex, l.s.c. in $L^1(\Om)$ and it sublevel sets are compact in $L^1(\Om),$ therefore it defines a normal coercive integrand.

\medskip

From \eqref{eq:BV-estim} and from the definition of $\cF$, we have
$$\sup_{\e_\t}\int_0^T\cF(\ov\rho_{i,t}^\t)\dd t<+\infty.$$
And similarly from Lemma \ref{lem:estimates}(i) and from \eqref{eq:contraction}, we have 
$$\lim_{h\da 0}\sup_{\e_\t}\int_0^{T-h}g(\ov\rho_{i,t+h}^\t,\ov\rho_{i,t}^\t)\dd t<+\infty,$$
therefore the assumptions of \cite[Theorem 2]{RosSav} are fulfilled and one can conclude that there exists $\ov\rho_i\in L^1([0,T]\times\Om)$ and a subsequence of $(\ov\rho_i^\t)_{\t>0}$ (that we do not relabel) which is converging in measure, and in particular pointwise a.e. to $\ov\rho_i$ as $\e_\t\da 0$. The strong convergence in $L^1([0,T]\times\Om)$ follows from Lebesgue's dominated convergence theorem, since $(\ov\rho_i^\t)_{\t>0}$ is uniformly bounded. The claim follows. 

\medskip

{\it Claim 3.} There exists a subsequence of the original sequence $(\rho_{i}^\t)_{\t>0}$ which is converging to $\ov\rho_i$ strongly in $L^1([0,T]\times\Om)$ as $\e_\t\da 0$.

{\it Proof of Claim 3.} Passing to the same subsequence in the original sequence $(\rho_{i}^\t)_{\t>0}$ (that we do not relabel) as in the last step of the proof of Claim 2, we have
$$\|\rho_{i}^\t-\ov\rho_i\|_{L^1([0,T]\times\Om)}\le\|\rho_i^\t-\ov\rho_i^\t\|_{L^1([0,T]\times\Om)}+\|\ov\rho_i^\t-\ov\rho_i\|_{L^1([0,T]\times\Om)}\le C\sqrt{\e}+\|\ov\rho_i^\t-\ov\rho_i\|_{L^1([0,T]\times\Om)},$$
for a constant $C>0$ (independent of $\t$ and $\e$) and the claim follows by taking $\e_\t\da 0.$ In the last inequality we have used
\begin{align*}
\|\rho_1^\t-\ov\rho_1^\t\|_{L^1([0,T]\times\Om)}&=\int_0^T\int_\Om|\rho_1^\t-G_\e\star\rho_1^\t|\dd x \dd t\le\int_0^T\int_{\Om} \int_{\RR^d} G(z) |\rho_1^{\t}(x)-\rho_1^{\t}(x-\sqrt{\e} z)|\dd z\dd x\dd t\\
&\le CT\sqrt{\e}\HC_\e(\rho_{1,0},\rho_{2,0}),
\end{align*}
where we relied on the fact that since $\rho^{\t}_1$ and  $\rho^{\t}_2$ take values in $\{0,1\}$, we have $  |\rho_1^{\t}(x)-\rho_1^{\t}(x-\sqrt{\e} z)|\leq \rho^{\t}_1(x)\rho^{\t}_2(x-\sqrt{\e}z)+\rho^{\t}_1(x-\sqrt{\e}z)\rho^{\t}_2(x)$.

\medskip

To conclude, let us notice that since $(\rho_i^\t)_{\t>0}$ converges uniformly w.r.t. $W_2$ to $\rho_i$, $\ov\rho_i$ and $\rho_i$ have to coincide. Also, for this last subsequence, we can pass to the limit as $\e_\t\da 0$ in the estimation of Lemma \ref{lem:estimates}(v) to obtain that actually $\rho_i\in L^1([0,T];BV(\Om;\{0,1\})).$

\end{proof}
Let $\bm{\rho}=(\rho_1,\rho_2)$ the limit point obtained in Proposition \ref{prop:strong_limit}. Later in this section, we will profit from the assumption 
\begin{equation}\label{ass:energy_conv}
\int_0^T\hc_{\epsilon_{\tau}}(\bm{\rho}^{\tau})\dd t \to \int_0^T\cE_s(\bm{\rho})\dd t, \hbox{ as } \e\da 0.\tag{EC}
\end{equation}

\subsection{Derivation of the weak curvature equation for $\epsilon$ going to zero together with $\tau$}
Let $\xi\in C^2([0,T]\times\Om;\R^d)$. Fix $t\in[n\t,(n+1)\t).$ We consider the piecewise constant interpolations $(\rho_i^\t,v_i^\t,p^\t)$. Let us take the inner product of the equations \eqref{eq:pressure-gradient} with $\xi(t,\cdot)$, multiply with the corresponding $\rho_i^\t$ and integrate over $[0,T]\times\Om$. Adding up the two equations we get
\begin{align*}
\int_0^T\int_{\Om}\left[\nabla K_\e\star\rho_2^\t+\nabla\Phi_1+v_1^\t +\nabla p^\t \right]\cdot\xi\rho_1^\t\dd x\dd t+\int_0^T\int_{\Om}\left[\nabla K_\e\star\rho_1^\t+\nabla\Phi_2+v_2^\t  +\nabla p^\t \right]\cdot \xi\rho_2^\t\dd x\dd t=0,
\end{align*}
rearranging the terms yields
\begin{equation}\label{eq:muskat_tau_eps}
\int_0^T\int_{\Om}\xi\cdot \left[(\nabla K_\e\star\rho_2^\t)\rho_1^\t+(\nabla K_\e\star\rho_1^\t)\rho_2^\t\right]+(\rho_1^\t (v_1^\t+\nabla\Phi_1)+\rho_2^\t (v_2^\t+\nabla\Phi_1))\cdot\xi   +\nabla p^\t \cdot \xi\dd x\dd t=0.
\end{equation}

Our aim is now to pass to the limit in this last expression \eqref{eq:muskat_tau_eps} as $\e_\t\da 0$ to recover \eqref{eq:goal}. 

\begin{theorem}\label{thm:limits}
Let $(\rho_i^\t,v_i^\t,p^\t)_{\t>0}$ be the piecewise constant interpolations constructed in \eqref{def:rho-tau}-\eqref{def:vpE-tau}. Then there exists $\rho_i\in L^1([0,T];BV(\Om;\{0,1\}))\cap {\rm{AC}}^2([0,T];\sP(\Om))$, $v_i\in L^2([0,T];L^2_{\rho_i}(\Om;\R^d))$ and $p\in{L^2([0,T]; (C^{0,\a}(\Om))^*)}$ such that, up to passing to a subsequence that we do not relabel, we have the following
\begin{itemize}
\item[(i)] $\rho_i^\t\to\rho_i$, as $\e_\t\da 0$, strongly in $L^1([0,T]\times\Om)$; 
\item[(ii)] $v_i^\t\rho_i^\t\weaklys v_i\rho_i$, as $\e_\t\da 0$, weakly-$\star$ in $\sM^d([0,T]\times\Om)$;
\item[(iii)] $\nabla p^\t\weaklys \nabla p$, as $\e_\t\da 0$, weakly-$\star$ in $L^2([0,T];(C^1(\Om))^\star);$ 
\end{itemize}
\end{theorem}

\begin{proof}
(i) is a consequence of Proposition \ref{prop:strong_limit} via the Aubin-Lions type argument.

\medskip

(ii) Let us notice that the estimates in Lemma \ref{lem:estimates}(i-iii) are independent of $\e>0$, therefore there exists $E_i\in\sM^d([0,T]\times\Om)$, $i=1,2$ such that $E_i^\t\weaklys E_i$ and $\tilde E_i^\t\weaklys E_i$ as $\e_\t\da 0$. Moreover, we have that $(\rho_i,E_i)$ solves $\partial_t+\nabla\cdot(E_i)=0$ in the sense of distributions and $\rho_i\in {\rm{AC}}^2([0,T];(\sP(\Om),W_2)).$ Therefore, there exists $v_i\in L^2([0,T];L^2_{\rho_i}(\Om;\R^d))$ such that $\partial_t\rho_i+\nabla\cdot(\rho_i v_i)=0$ in the sense of distributions.

\medskip

(iii) Let us notice first that from Lemma \ref{lem:estimates}(vii) we have that the sequence $(\nabla p^\t)_\t$ is uniformly bounded in $L^2([0,T];(C^1(\Om))^*)$, independently of $\e$. Then, the Banach-Bourbaki-Alaouglu theorem yields that it is sequentially pre-compact in that space, so we have that  there exists $\zeta\in L^2([0,T];(C^1(\Om))^*)$ such that up to passing to a subsequence $\nabla p^\t\weaklys\zeta$, as $\e_\t\da 0$. 

Thus, it only remains to show that $\zeta$ is a gradient. Let us notice that by construction of $p^\t$, 
$$\ds\int_0^T\int_{\Om}\nabla p^\t\cdot\xi\dd x\dd t=0,$$ for any incompressible smooth field $\xi$. Therefore, we also have that $\langle\zeta,\xi\rangle=0$ for any incompressible smooth field $\xi$. So we would be done, if one would have a Helmholtz decomposition in this corresponding space. This result is well known (see for instance \cite[Lemma 2.2.1]{Soh}), if $\zeta(t,\cdot)\in W^{-1,q}(\Om)^d$, for some $q\in(1,+\infty)$. However, our limit object has slightly worse regularity. 

\medskip

To overcome this issue, let us argue using the following claim. This is a consequence of classical result in the theory of elliptic equations and Schauder estimates (see for instance \cite[Theorem 5.23-Theorem 5.24]{GiaMar})

{\it Claim.} Let $\vphi\in C^{0,\a}(\Om)$. Then the problem $-\Delta u=\vphi$ (with homogeneous zero Neumann boundary condition, if $\int_\Om\vphi\dd x=0$) has a unique (modulo constants) solution $u\in C^{2,\a}(\Om)$ such that $\|u\|_{C^{2,\a}}\le \|\vphi\|_{C^{0,\a}}.$

\medskip

Now, let $\vphi\in C^{0,\a}(\Om)$ and $u\in C^{2,\a}(\Om)$ as in the claim. Let $t\in[0,T]$. Then we have
\begin{align*}
\Bigg{|}\int_{\Om}p^\t(t,\cdot)\vphi\dd x\Bigg{|}=\Bigg{|}\int_{\Om}p^\t(t,\cdot)\Delta u\dd x\Bigg{|}=\Bigg{|}\int_{\Om}\nabla p^\t(t,\cdot)\cdot\nabla u\dd x\Bigg{|}\le C\|\nabla u\|_{C^1}\le C\|u\|_{C^{2,\a}}\le C\|\vphi\|_{C^{0,\a}},
\end{align*}
where in the first inequality we have used the last uniform estimate from the proof of Lemma \ref{lem:estimates}(vii). This implies that $(p^\t(t,\cdot))_{\t}$ is uniformly bounded in $(C^{0,\a}(\Om))^*$.

To conclude the thesis of the lemma, we observe that (by possibly subtracting the mean) $(p^\t)_\t$ is also converging weakly-$\star$ to some $p$, therefore we will have that $\zeta=\nabla p$ in $\sD'((0,T)\times\Om).$

\end{proof}

To complete the proof of of Theorem \ref{thm:main}, it remains to show that limit of equation (\ref{eq:muskat_tau_eps}) converges to the weak formulation of the Muskat problem.    This result is proven in Proposition \ref{prop:limit_eq}, which in turn uses the results of Lemmas \ref{lem:pointwise_reduction} and \ref{lem:localized_hc_convergence}.  These are direct consequences of the corresponding results from \cite{LauOtt}. However, for completeness and to facilitate the reading, we collected them in the Appendix \ref{sec:appendix} below. These results allow to simply conclude this section with the proof of Theorem \ref{thm:MAIN}.
\begin{proof}[Proof of Theorem \ref{thm:MAIN}]
This proof is a direct consequence of Theorem \ref{prop:strong_limit} and Proposition \ref{prop:limit_eq}. Indeed, these two result allow us to pass to the limit in the equation \eqref{eq:muskat_tau_eps} to obtain \eqref{eq:goal}.
\end{proof}

\section{Numerics and equilibrium shapes}\label{sec:equilibrium} 

In this section we present several examples of the performance of our numerical scheme and a discussion of equilibrium shapes.

\subsection{Numerical implementation}\label{sec:numer}
The Muskat problem evolution can be simulated by discretizing the minimizing movements scheme (\ref{eq:hc_minimizing_movements}) onto a regular grid.    At first glance, numerically solving the discretized variational problem is not so simple. Indeed, Problem (\ref{eq:hc_minimizing_movements}) is not convex with respect to $\rho$ and the Wasserstein distance is challenging to work with numerically.   However, as noted in the introduction, the scheme can be substantially simplified by applying the heat content linearization trick used in \cite{EO}.  To that end, note that the convexity of the heat content gives us the upper bound
\begin{equation}\label{eq:concavity_of_hc}
\hc_{\epsilon}(\bm{\rho})\leq \hc_{\epsilon}(\bm{\rho}^n)+(\delta \hc_{\epsilon}(\bm{\rho}^n), \bm{\rho}-\bm{\rho}^n)
\end{equation}
where the second term is the linearization of the heat content about the previous iterate $\bm{\rho}^n$. 
 Thus, if we replace  (\ref{eq:hc_minimizing_movements}) by the linearized scheme,
 \begin{equation}\label{eq:linearized_scheme}
 \bm{\rho}^{n+1}=\argmin_{\bm{\rho}} \, \hc_{\epsilon}(\bm{\rho}^n)+(\delta \hc{E}_{\epsilon}(\bm{\rho}^n), \bm{\rho}-\bm{\rho}^n)+\cE_p(\bm{\rho})+ \Phi(\bm{\rho})+\sum_{i=1}^2\frac{b_i}{2\tau }W_2^2(\rho_i, \rho^n_i).
 \end{equation}
we obtain a convex variational problem, and inequality (\ref{eq:concavity_of_hc}) ensures that the scheme still posses the energy dissipation property
\[
E_{\epsilon}(\bm{\rho}^{n+1})+\sum_{i=1}^2\frac{b_i}{2\tau}W_2^2(\rho_i^{n+1}, \rho_i^n)\leq E_{\epsilon}(\bm{\rho}^{n}).
\] 
A nearly identical argument to the proof of Proposition \ref{prop:characteristic} shows that the set of minimizers of (\ref{eq:linearized_scheme}) scheme always contains a configuration where $\rho_1$ and $\rho_2$ are characteristic functions.  

To solve problem (\ref{eq:linearized_scheme}) we introduce the pressure as a Lagrange multiplier for the incompressibility constraint, and instead work with the corresponding dual problem.  Up to a constant, the dual problem has the form
\begin{equation}
\label{eq:dual}
p_{n+1}=\argmax_{p} \int_{\Omega} (p+\psi_1^n)^{c_1}(x)\rho_1^n+(p+\psi_2^n)^{c_2}(x)\rho_2^n(x)-p(x)\, dx
\end{equation}
where 
\[
\psi_i^n(x)=\big(\delta \hc_{\epsilon}(\rho_i^n)\big)(x)+\Phi_i(x),
\]
and $c_1$, $c_2$ denote the quadratic $c$-transform
\[
(p+\psi_i^n)^{c_i}(x)=\inf_{y\in\Omega} \; p(y)+\psi_i^n(y)+\frac{b_i}{2\tau}|y-x|^2
\]
 (note $c$-transforms play an essential role in optimal transport see for instance \cite{OTAM}). 
The dual problem is concave with respect to $p$, and can be solved using the recently introduced back-and-forth method \cite{JacLeg}, which efficiently solves optimal transport problems in dual form.  

Due to the two phase nature of the problem, the optimal densities $\rho_i^{n+1}$ are not a simple function of the optimal pressure $p^{n+1}$ (this is in contrast to one-phase incompressible fluid flow where the occupied region is the support of the pressure).  On the other hand, once we have solved for the optimal pressure in (\ref{eq:dual}), we can recover the velocities $v_i$ for each phase (c.f. equation (\ref{eq:muskat_tau_eps})).  Thus, in principle, one can recover the densities $\rho^{n+1}_i$ from $v_i$ and $\rho_i^n$ by solving the continuity equation for time $\tau$.  However, solving the continuity equation accurately is challenging due to the discontinuity of the densities at the phase boundary.  Luckily, since we know that the densities remain as characteristic functions, we can instead compute $\rho_i^{n+1}$ using the level set method \cite{OsherSethian}.   If we let $\vp$ be the signed distance function to the interface between $\rho_1^n$ and $\rho_2^n$, then by solving the transport equation 
\[
\partial_t \vp+\nabla \vp\cdot (v_i\rho_i)=0
\]
for time $\tau$, we can recover $\rho_i^{n+1}$ through the sign of $\vp$.   The advantage of this approach is that the transport equation with Lipschitz initial data can be solved much more accurately than the continuity equation with discontinuous initial data.

\subsection{Numerical Experiments}

We demonstrate the performance of the numerical scheme on 3 different examples in 2 dimensions.  In each experiment, we take our computational domain to be the unit square $[0,1]^2$ and set the surface tension constant to be $\sigma=0.15$. 

In the first two examples, shown in Figures \ref{fig:small_drop} and \ref{fig:large_drop},  we  and choose potentials $\Phi_i(x,y)=-w_iy$ where $w_1=5$ and $w_2=1$.   In Figure \ref{fig:small_drop}, the starting configuration for phase 1 is a small square and in Figure \ref{fig:large_drop}, the starting configuration for phase 1 is a large square.   In both cases, the square becomes round and falls to the bottom of the computational domain.  However, due to the difference in mass between examples 1 and 2, the equilibrium configurations are different.  In Figure \ref{fig:small_drop}, the equilibrium configuration is a half disc sitting at the bottom of the domain, while in Figure \ref{fig:large_drop} the equilibrium configuration is a flat strip.

In the last example, shown in Figure \ref{fig:ripped_drop}, we choose a different potential that leads to a topological change.  We set 
\[
\Phi_1(x,y)=\begin{cases}
\frac{1}{2}-|y-\frac{1}{2}| & \textrm{if} \; \; y>\frac{1}{2}\\
\frac{5}{4}\big(\frac{1}{2}-|y-\frac{1}{2}|\big) & \textrm{if} \; \; y\leq \frac{1}{2}
\end{cases}
\]
and $\Phi_2(x,y)=0$.   The potential encourages phase 1 to migrate to the top and the bottom of the computational domain, with a stronger force attracting the drop to the bottom.  Because the potential pulls the drop in opposite directions, ultimately the initial drop is ripped apart into two separate droplets.   Thanks to the scheme's implicit representation of the interface $\Gamma$, there is no difficulty in simulating topological changes.

\begin{figure}
\centering
\includegraphics[width=0.225 \textwidth]{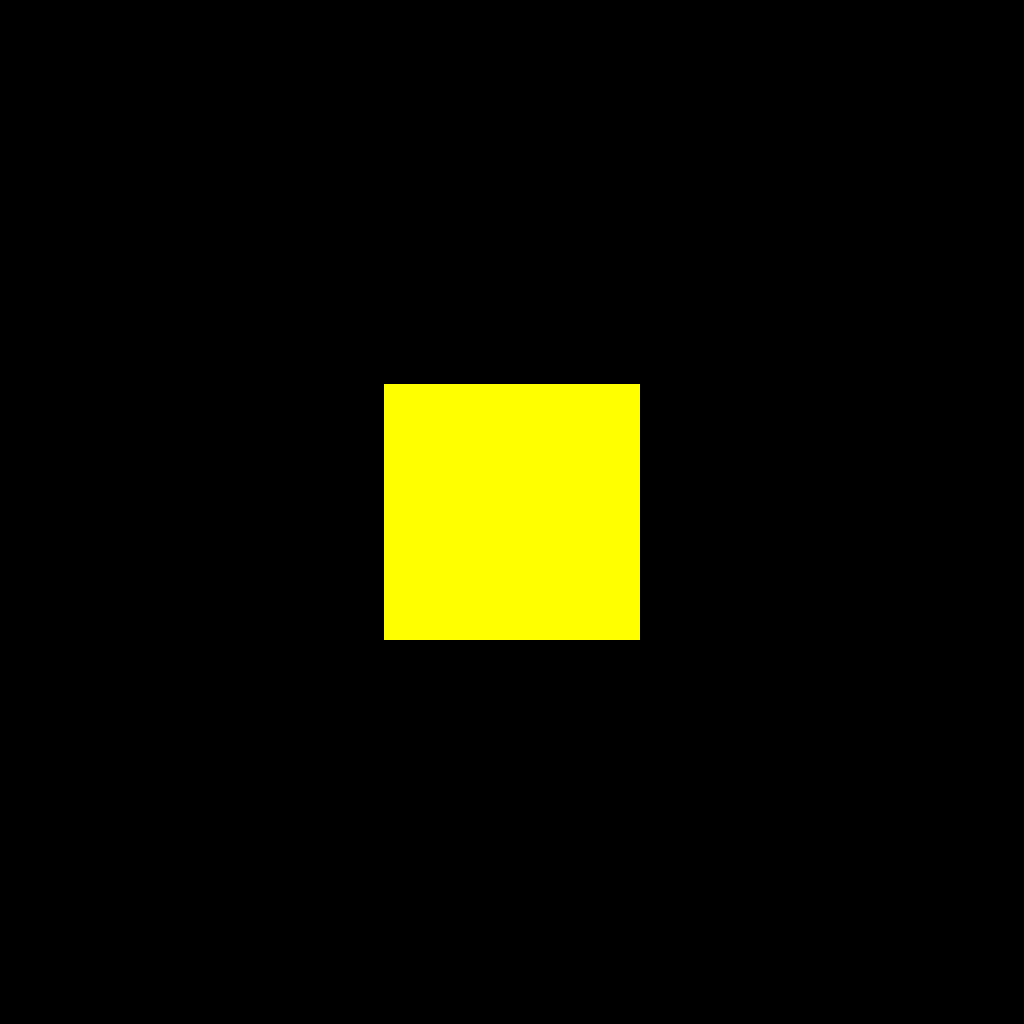}
\includegraphics[width=0.225 \textwidth]{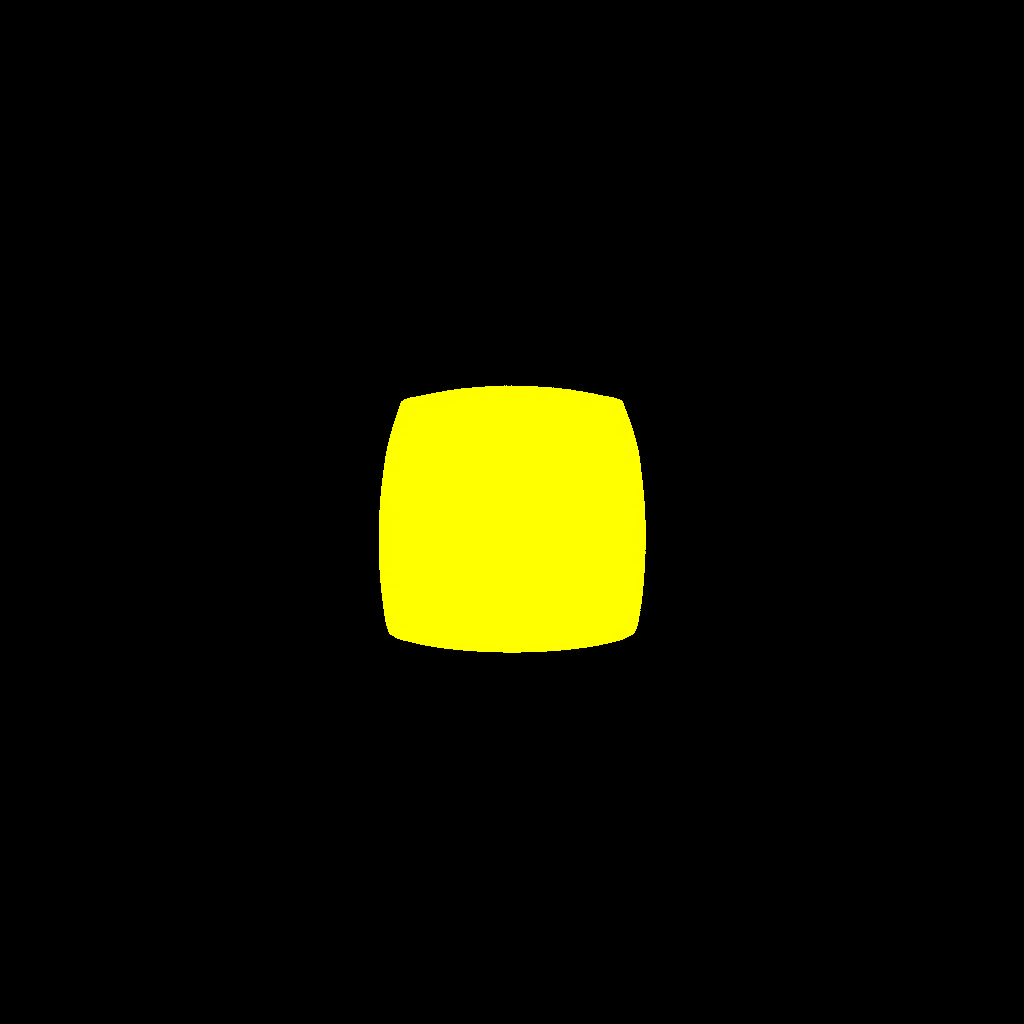}
\includegraphics[width=0.225 \textwidth]{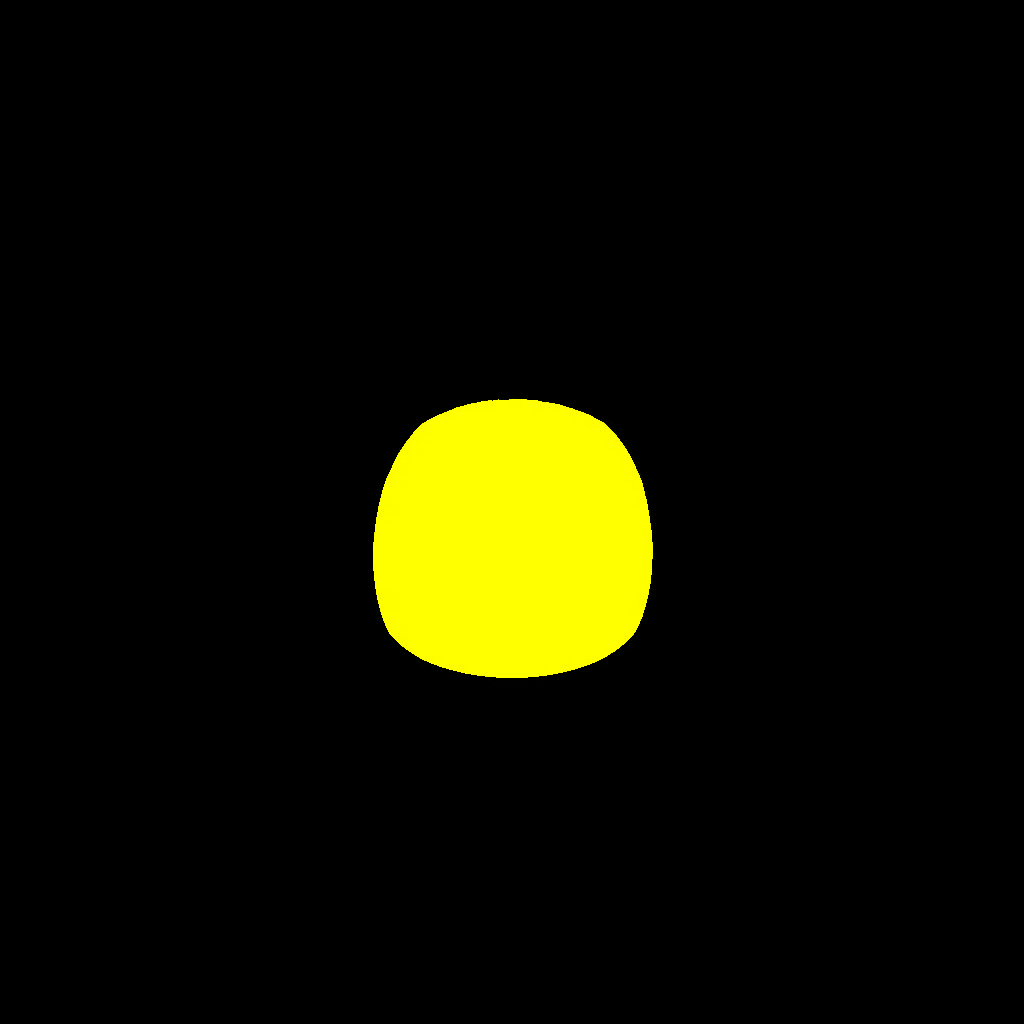}
\includegraphics[width=0.225 \textwidth]{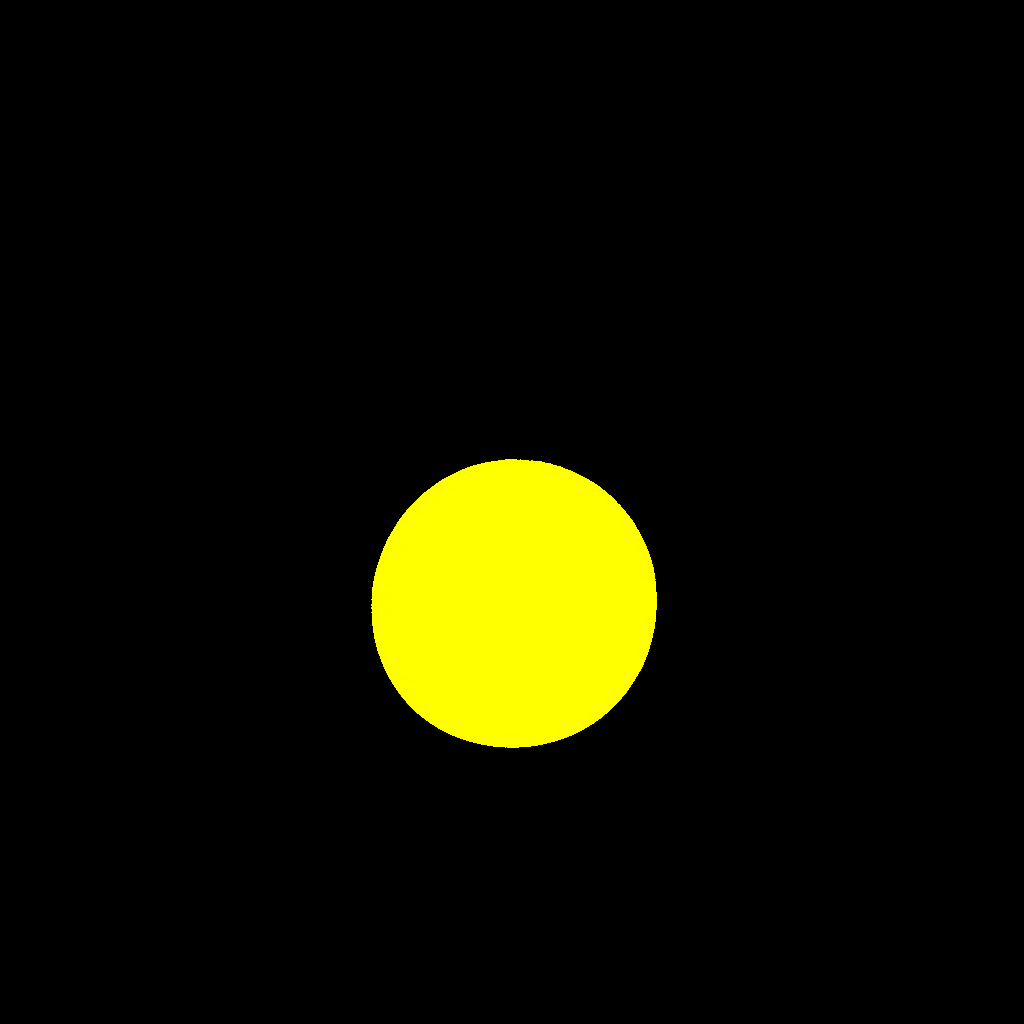}
\includegraphics[width=0.225 \textwidth]{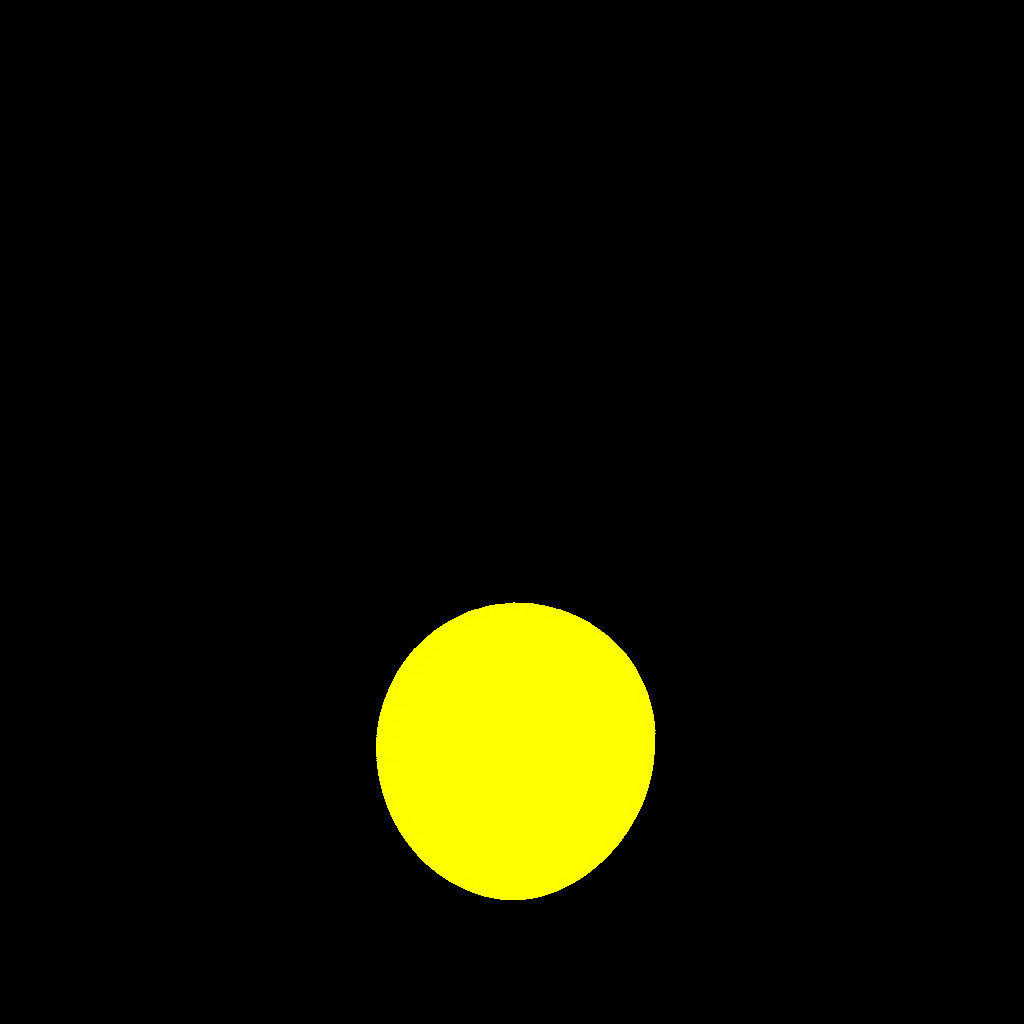}
\includegraphics[width=0.225 \textwidth]{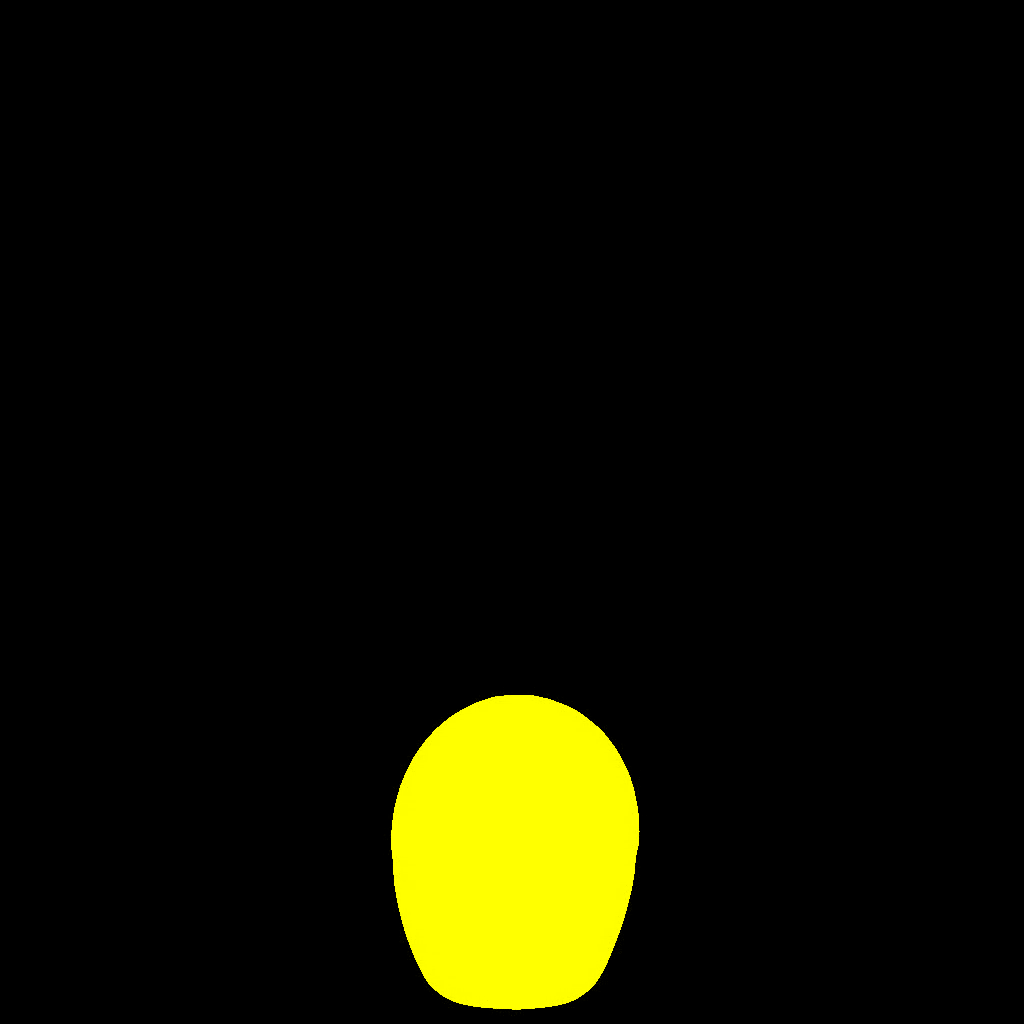}
\includegraphics[width=0.225 \textwidth]{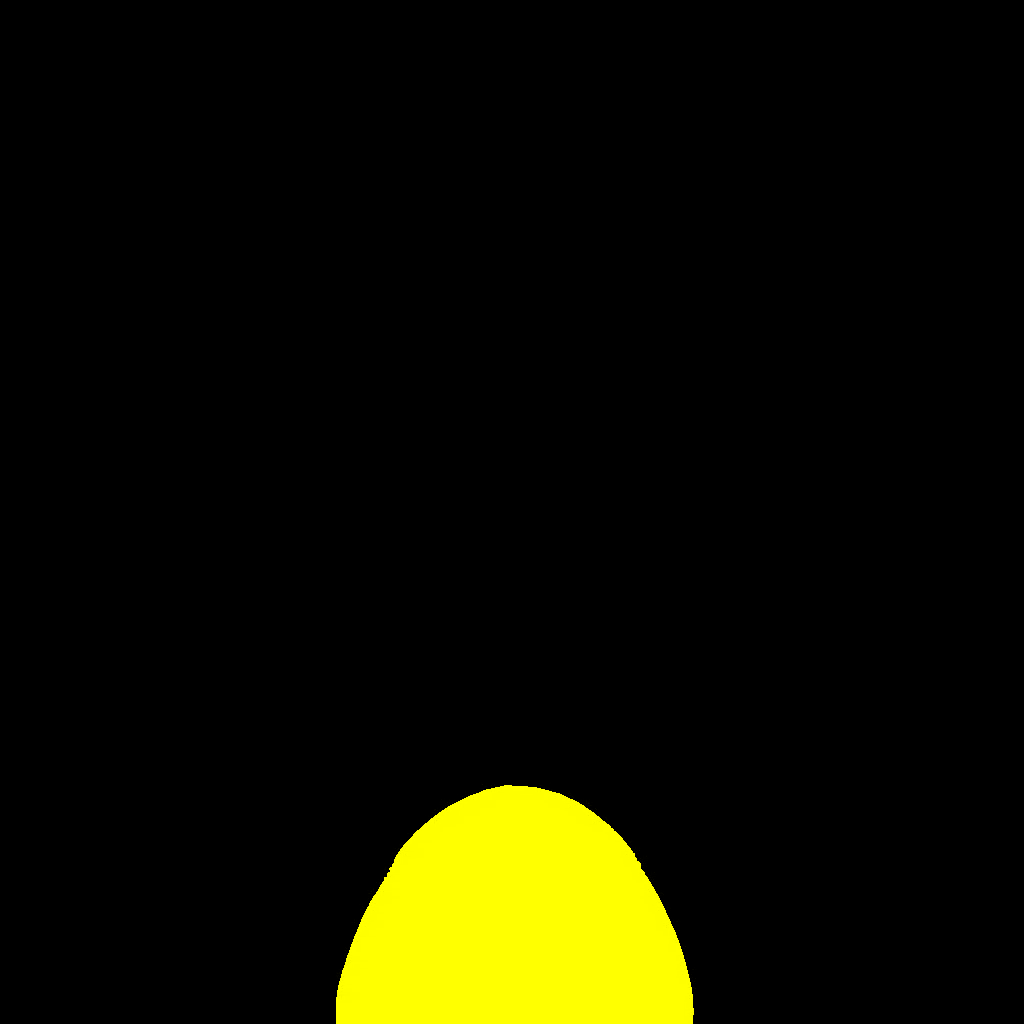}
\includegraphics[width=0.225 \textwidth]{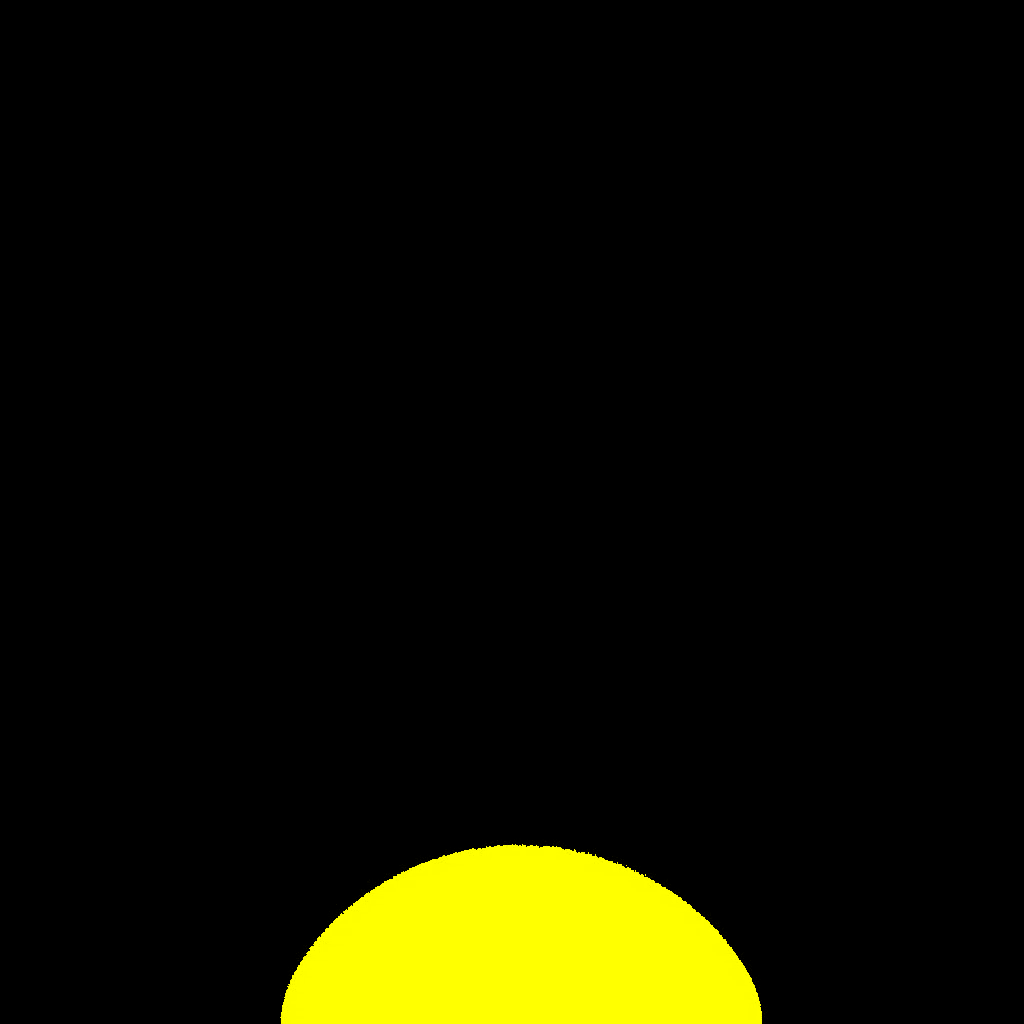}
\caption{ \footnotesize Numerical simulation of the Muskat problem evolution with a gravity potential.   The yellow phase is heavier than the black phase.  Under the influence of gravity and surface tension, the yellow phase becomes round and falls.  The images show snapshops of the evolution at several interesting times.  The final configuration is the stationary state.   \label{fig:small_drop} }
\end{figure}

\begin{figure}
\centering
\includegraphics[width=0.225 \textwidth]{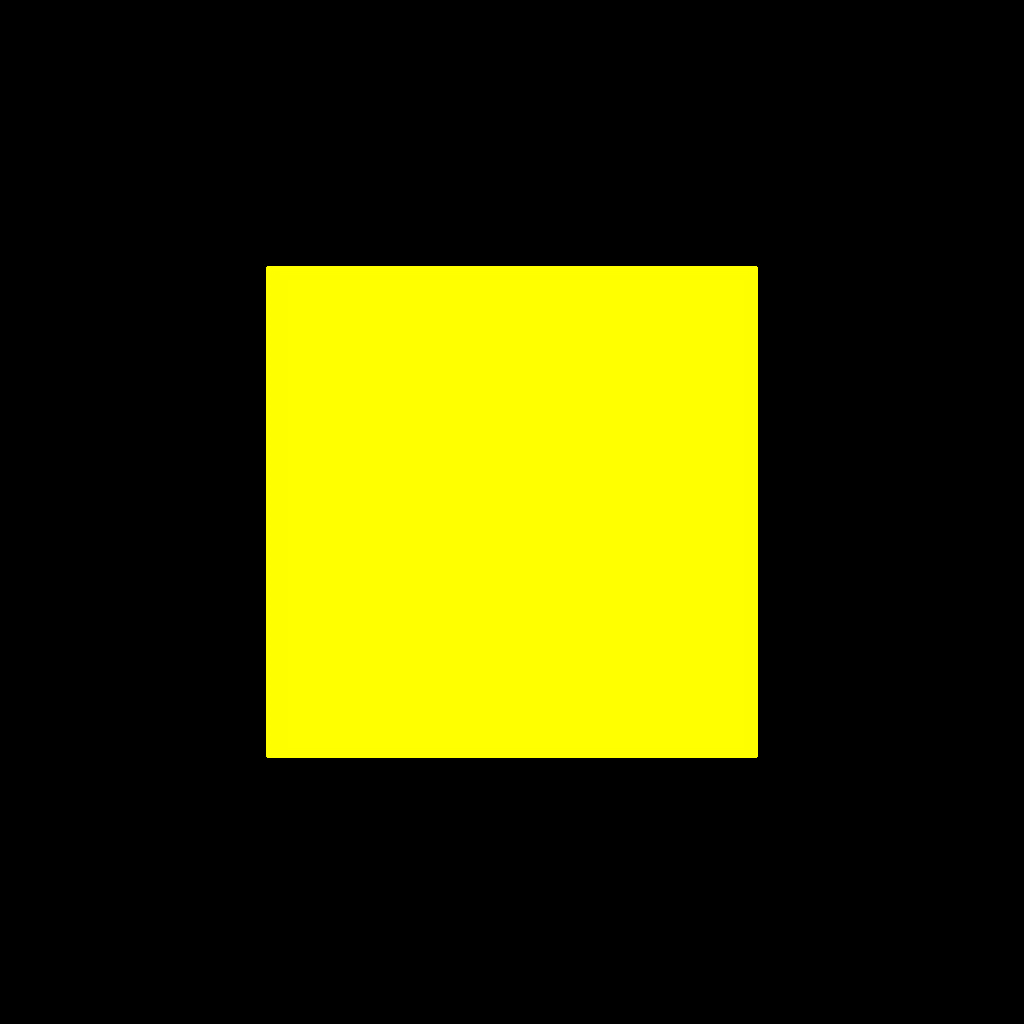}
\includegraphics[width=0.225 \textwidth]{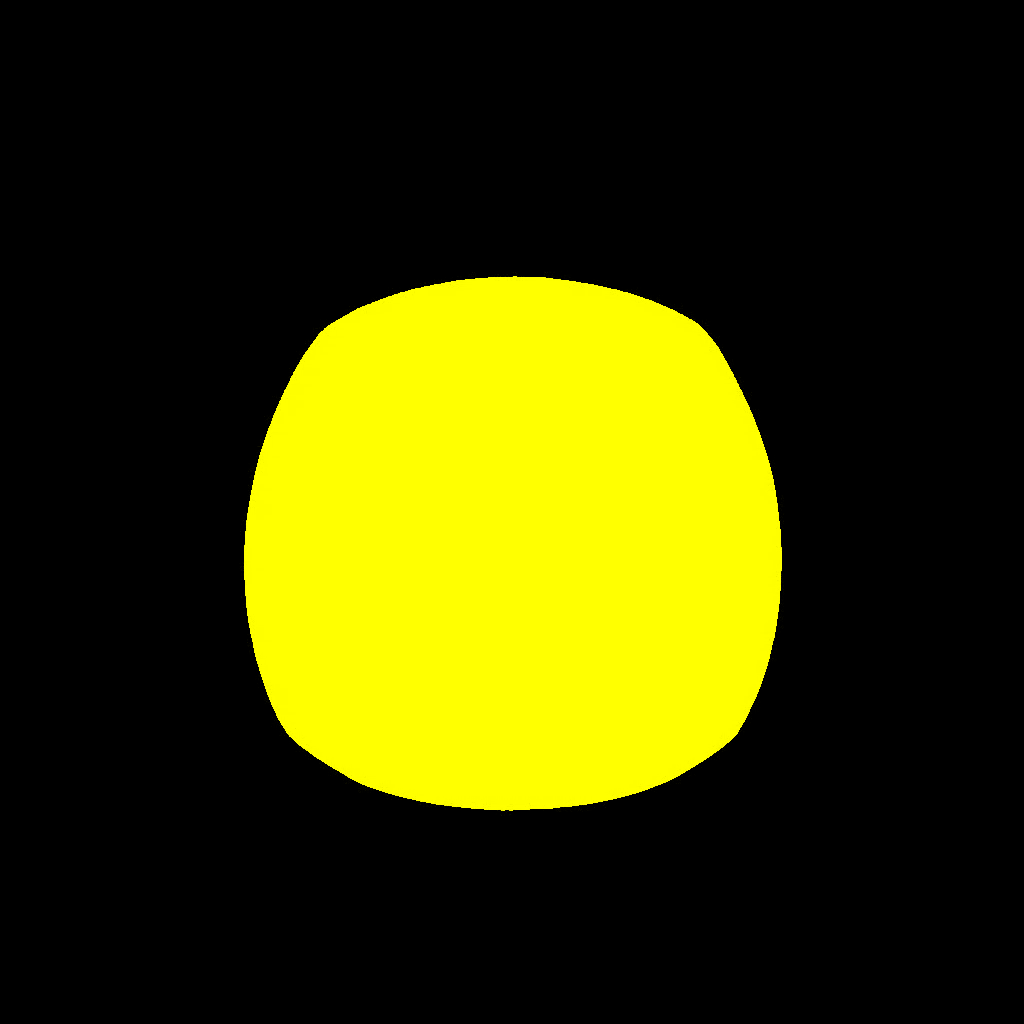}
\includegraphics[width=0.225 \textwidth]{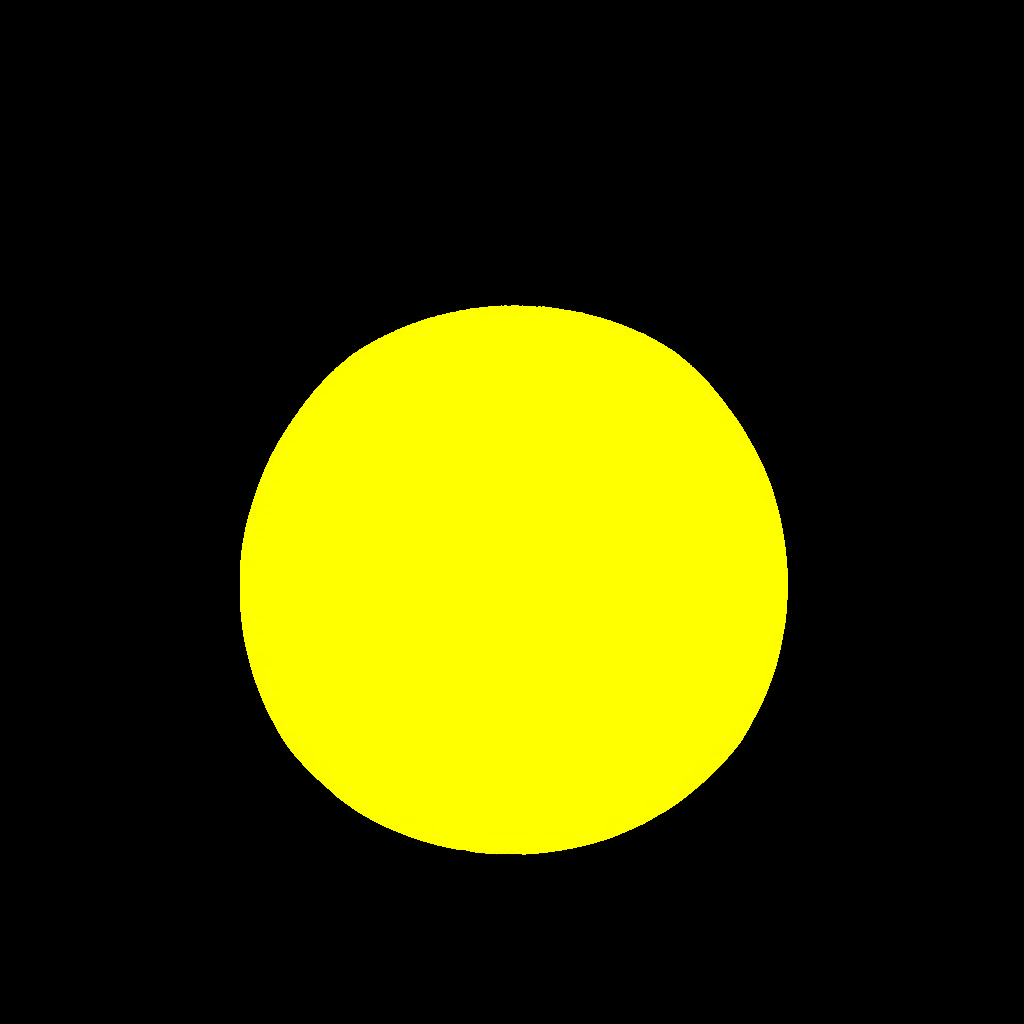}
\includegraphics[width=0.225 \textwidth]{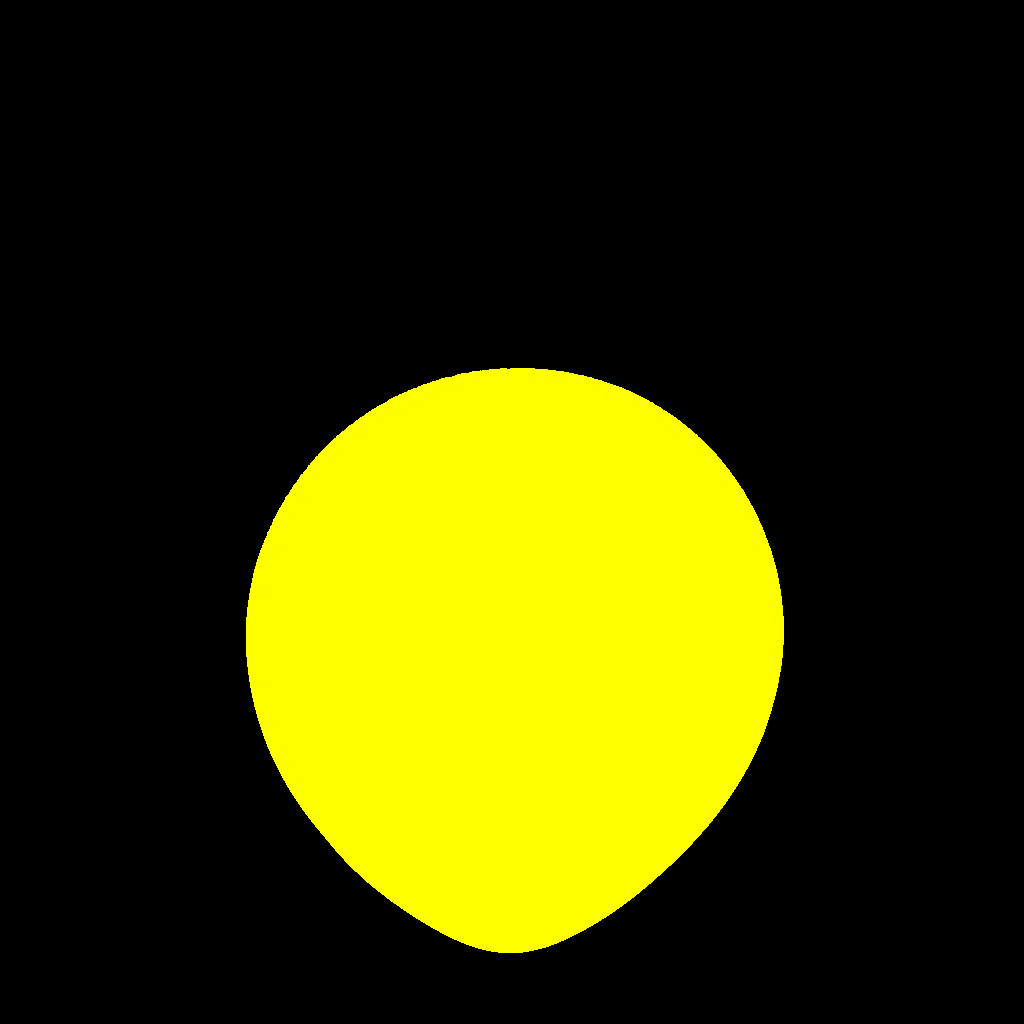}
\includegraphics[width=0.225 \textwidth]{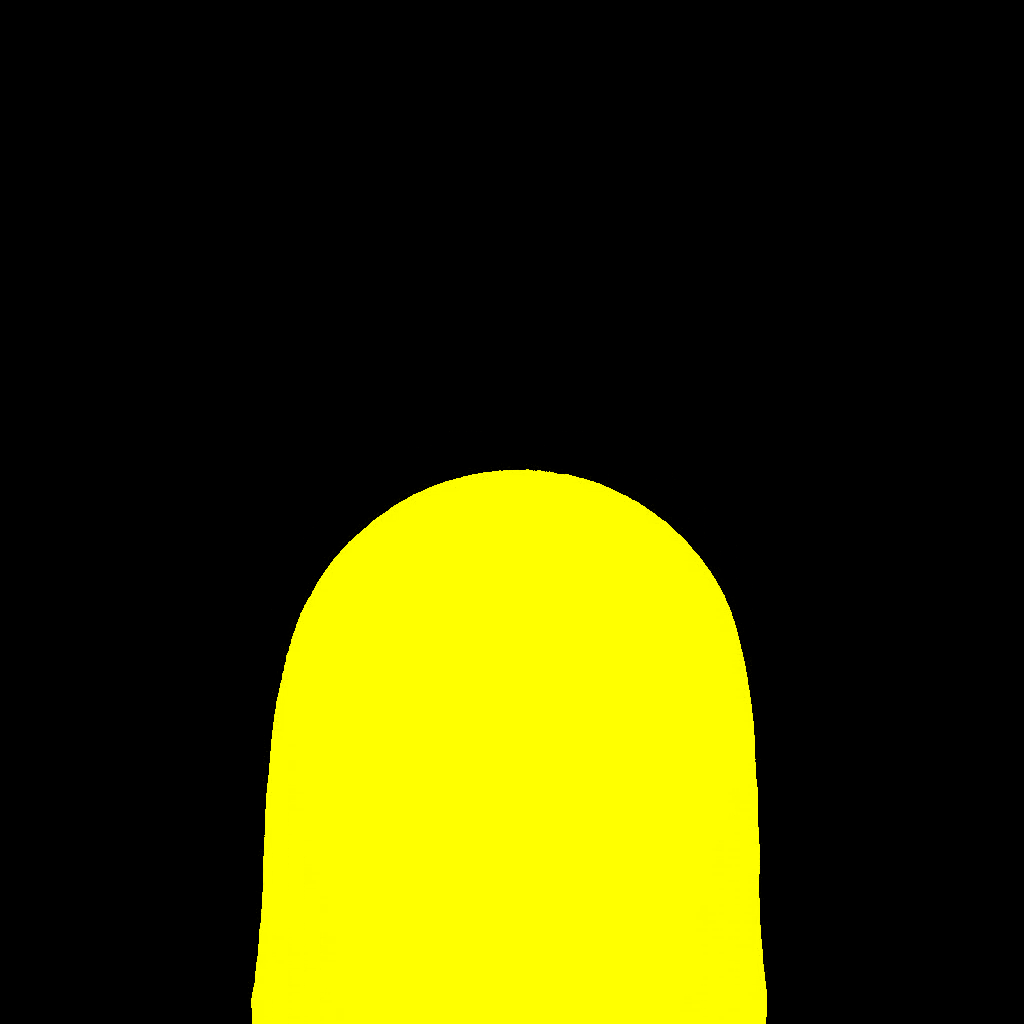}
\includegraphics[width=0.225 \textwidth]{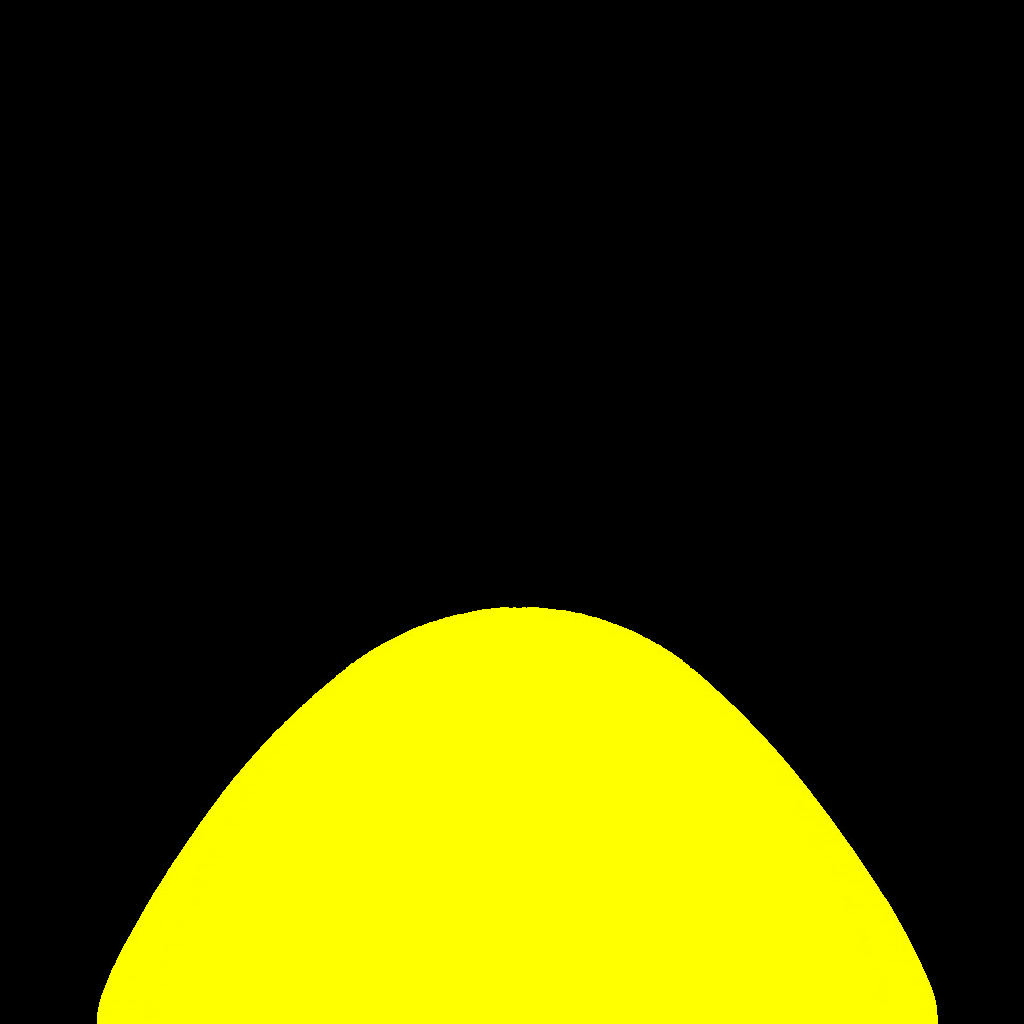}
\includegraphics[width=0.225 \textwidth]{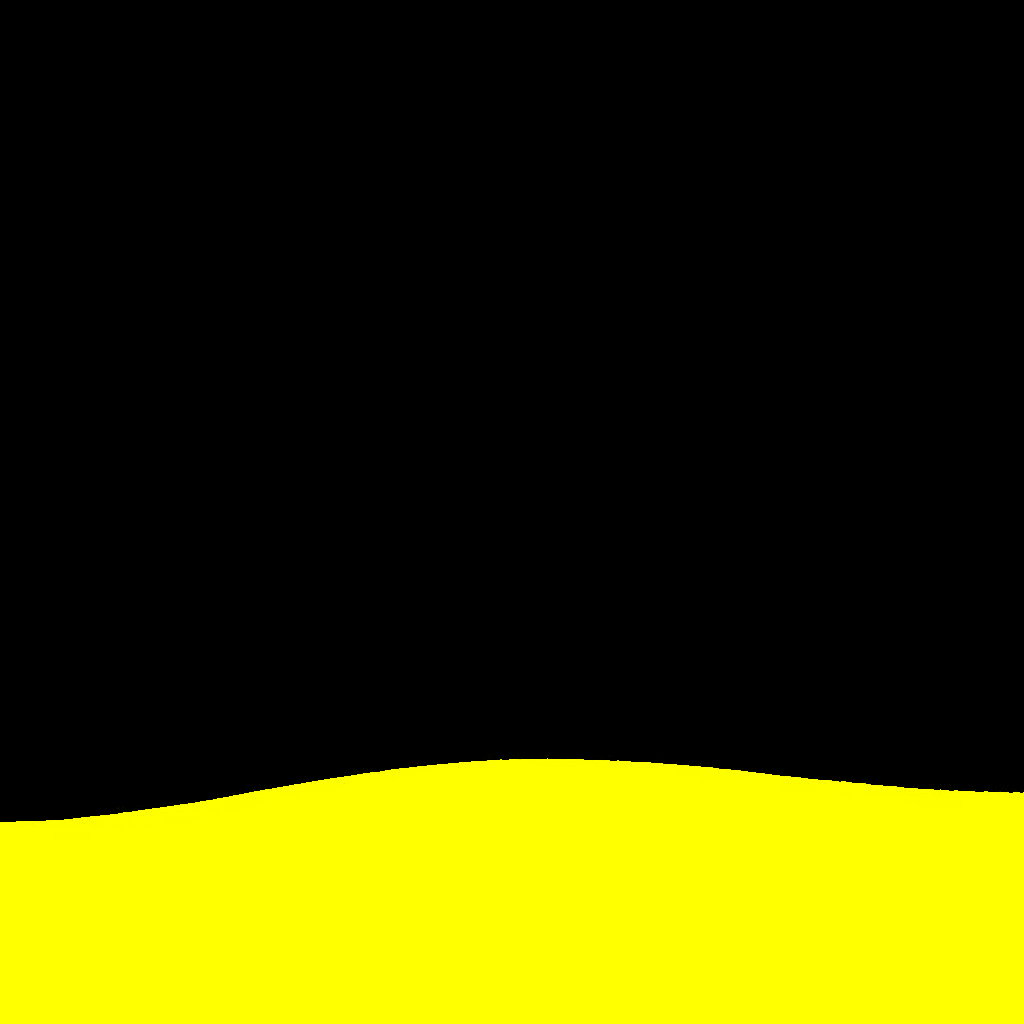}
\includegraphics[width=0.225 \textwidth]{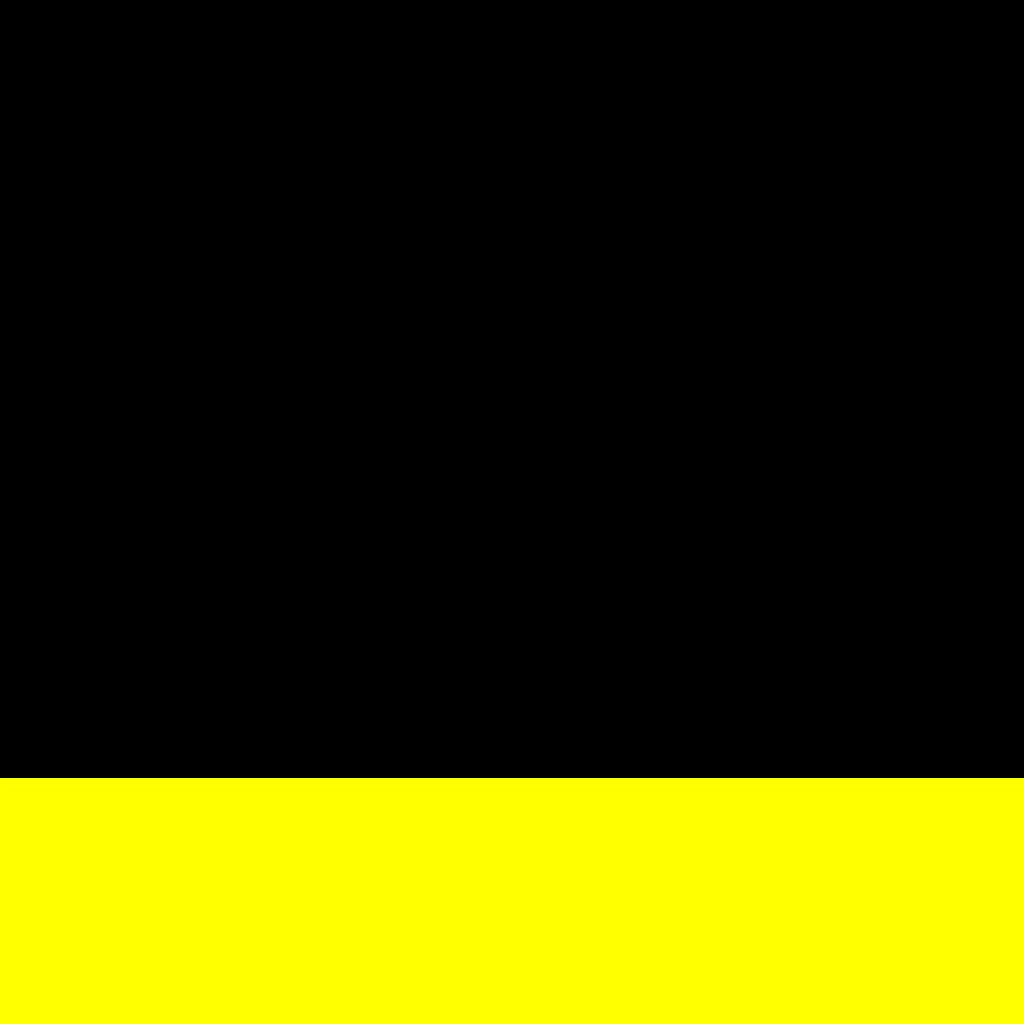}
\caption{\footnotesize Numerical simulation of the Muskat problem evolution with a gravity potential.  The setup is the same as in Figure \ref{fig:small_drop} except that the yellow region has a larger mass.  The images show snapshops of the evolution at several interesting times.  The final configuration is the stationary state, which is qualitatively different from the final configuration in Figure \ref{fig:small_drop} due to the larger mass.    \label{fig:large_drop} }
\end{figure}

\begin{figure}
\centering
\includegraphics[width=0.225 \textwidth]{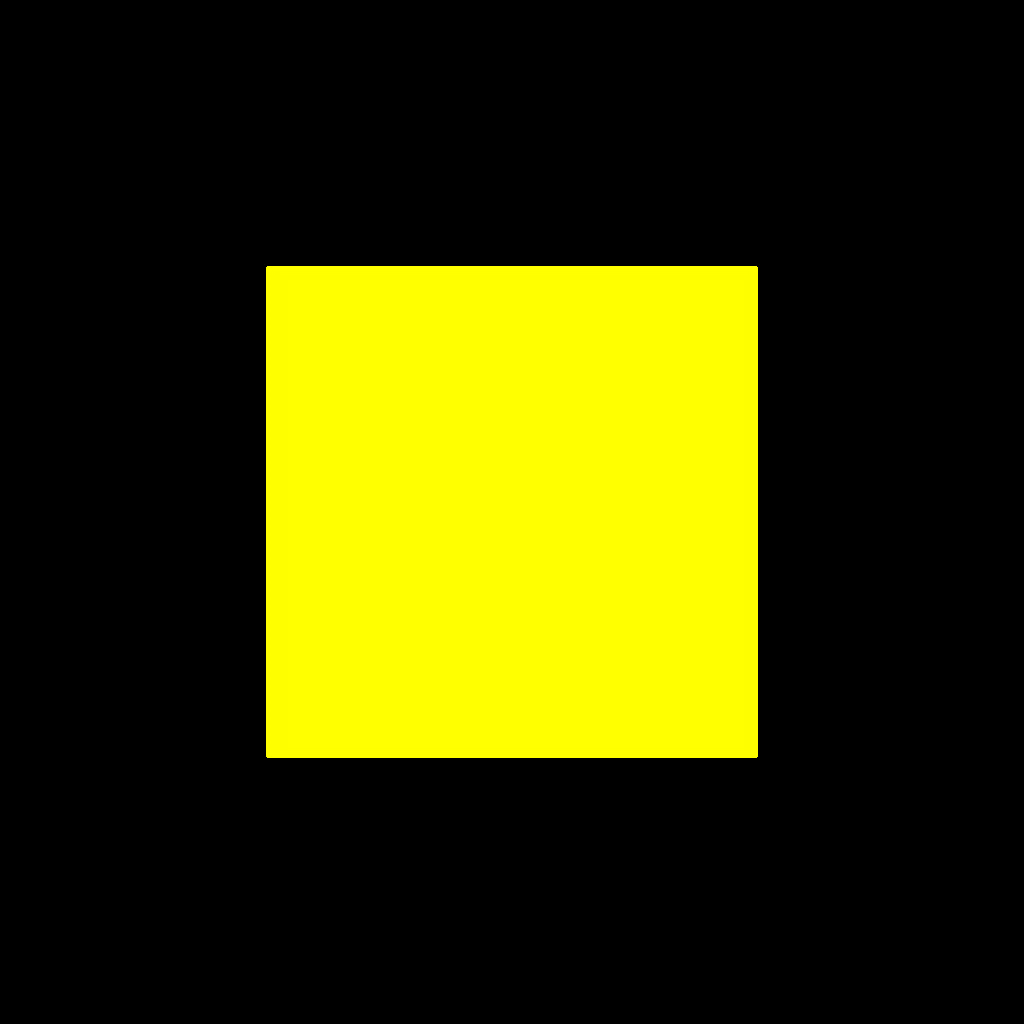}
\includegraphics[width=0.225 \textwidth]{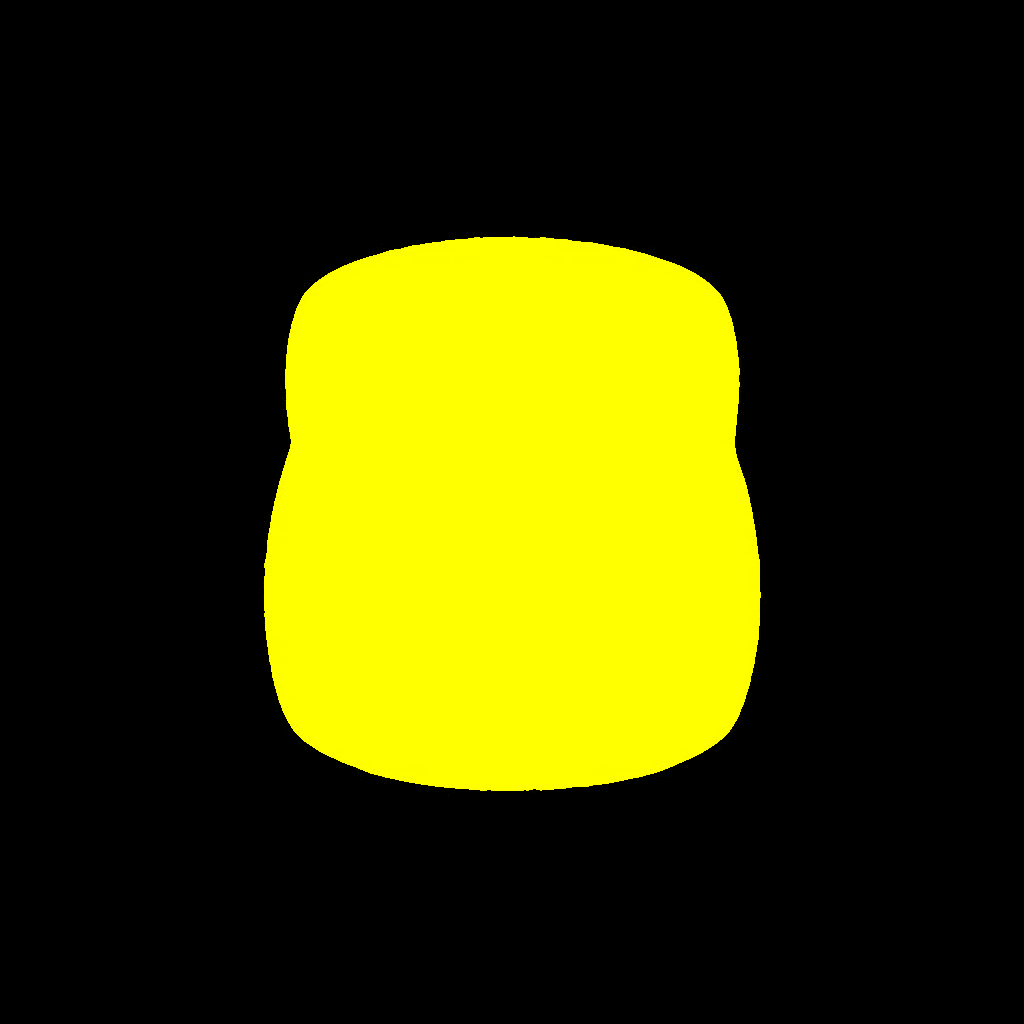}
\includegraphics[width=0.225 \textwidth]{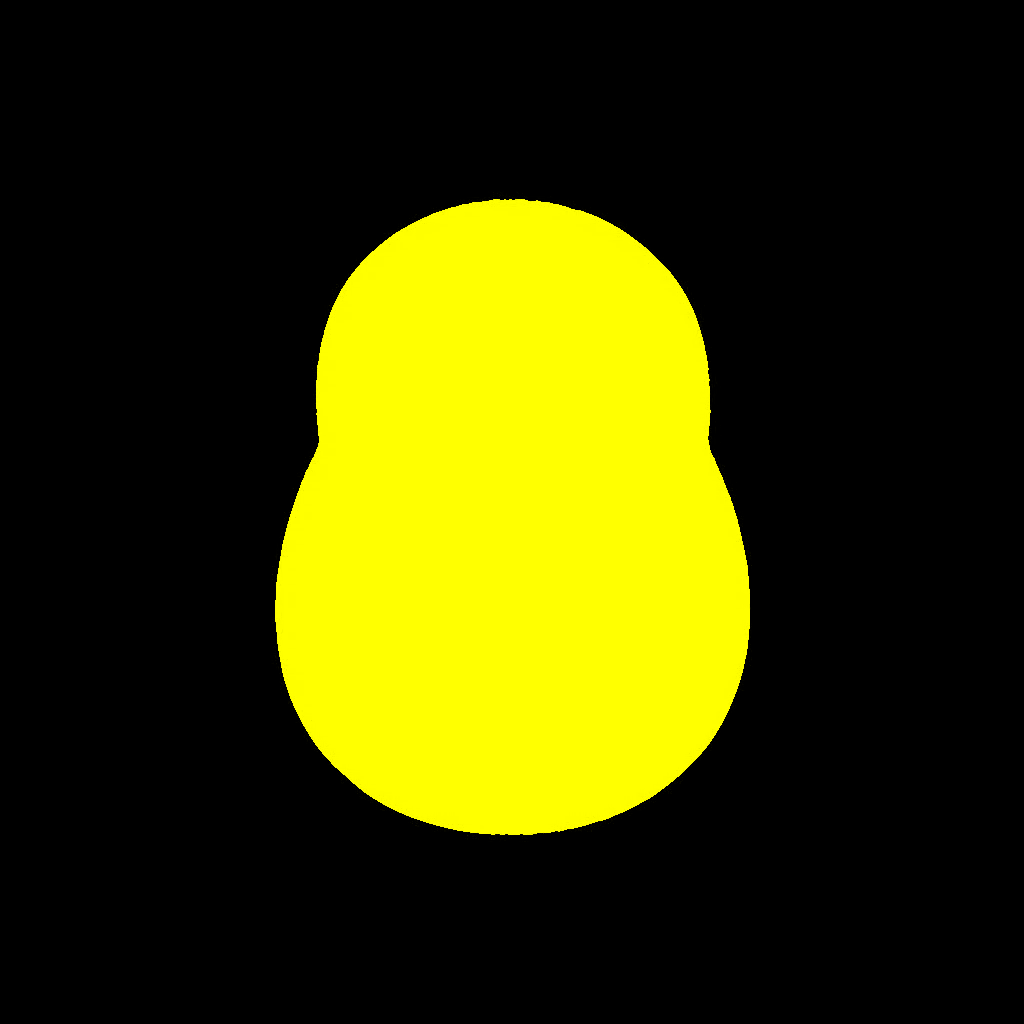}
\includegraphics[width=0.225 \textwidth]{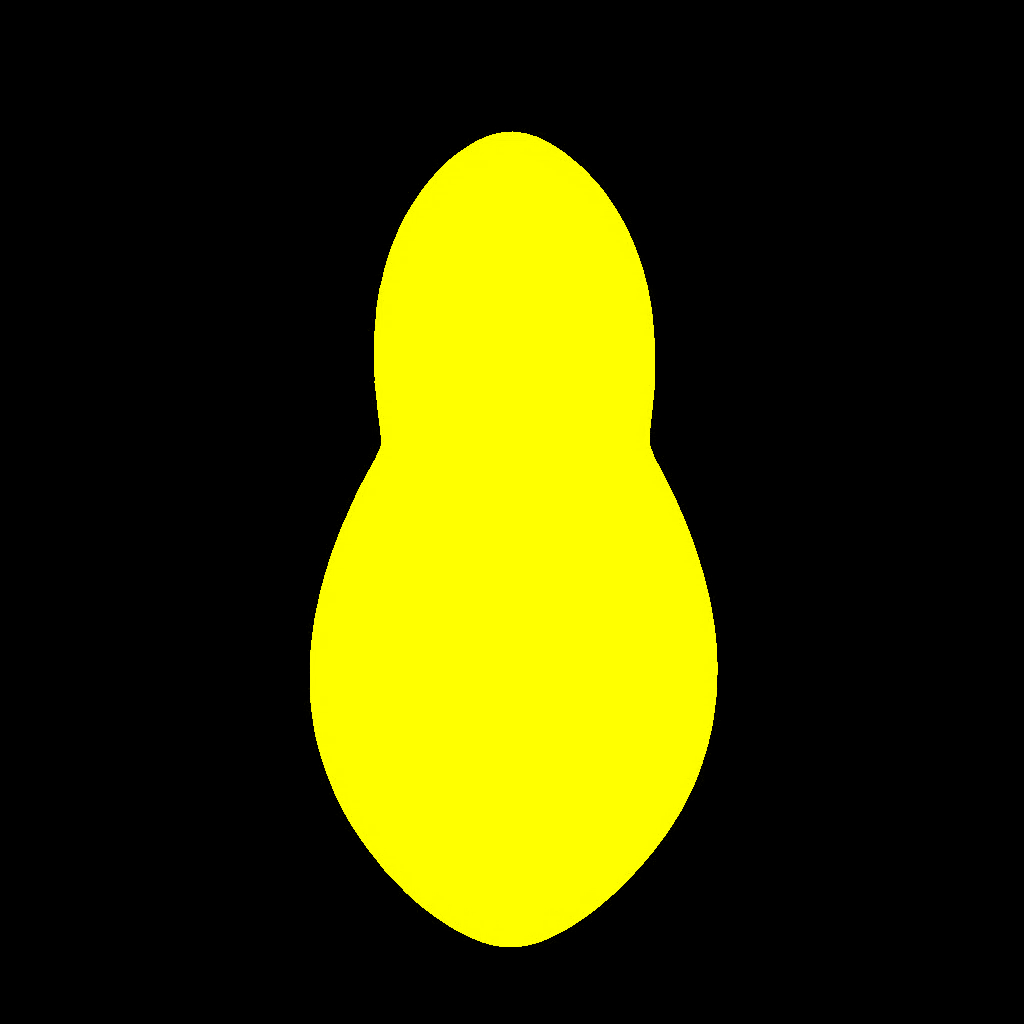}
\includegraphics[width=0.225 \textwidth]{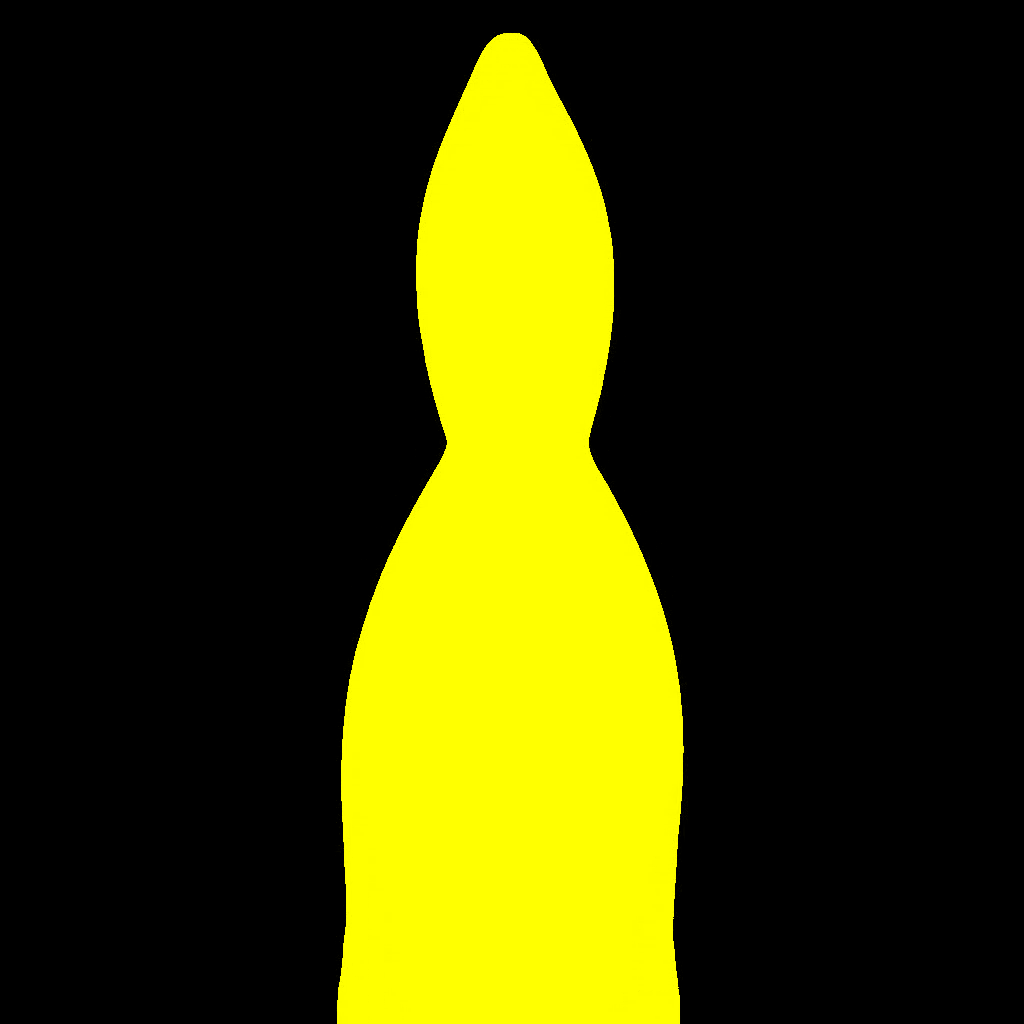}
\includegraphics[width=0.225 \textwidth]{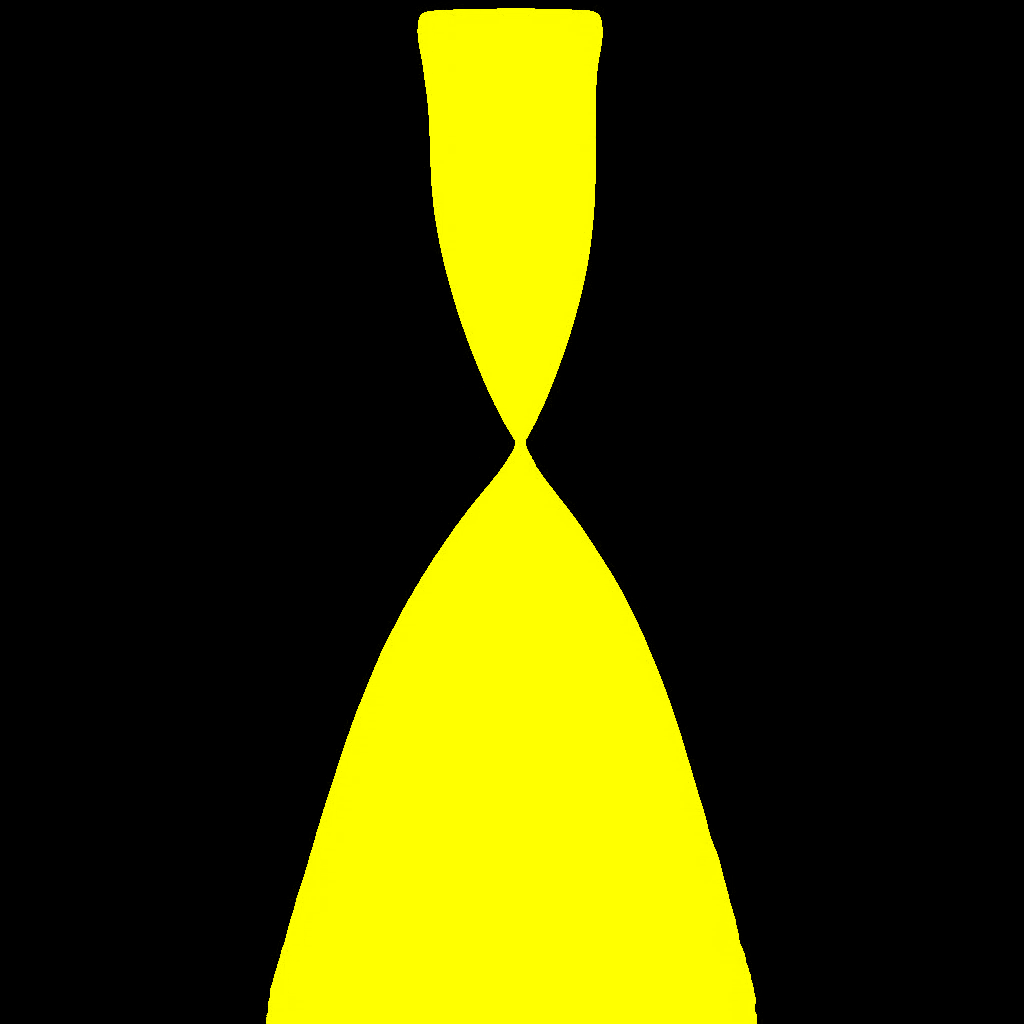}
\includegraphics[width=0.225 \textwidth]{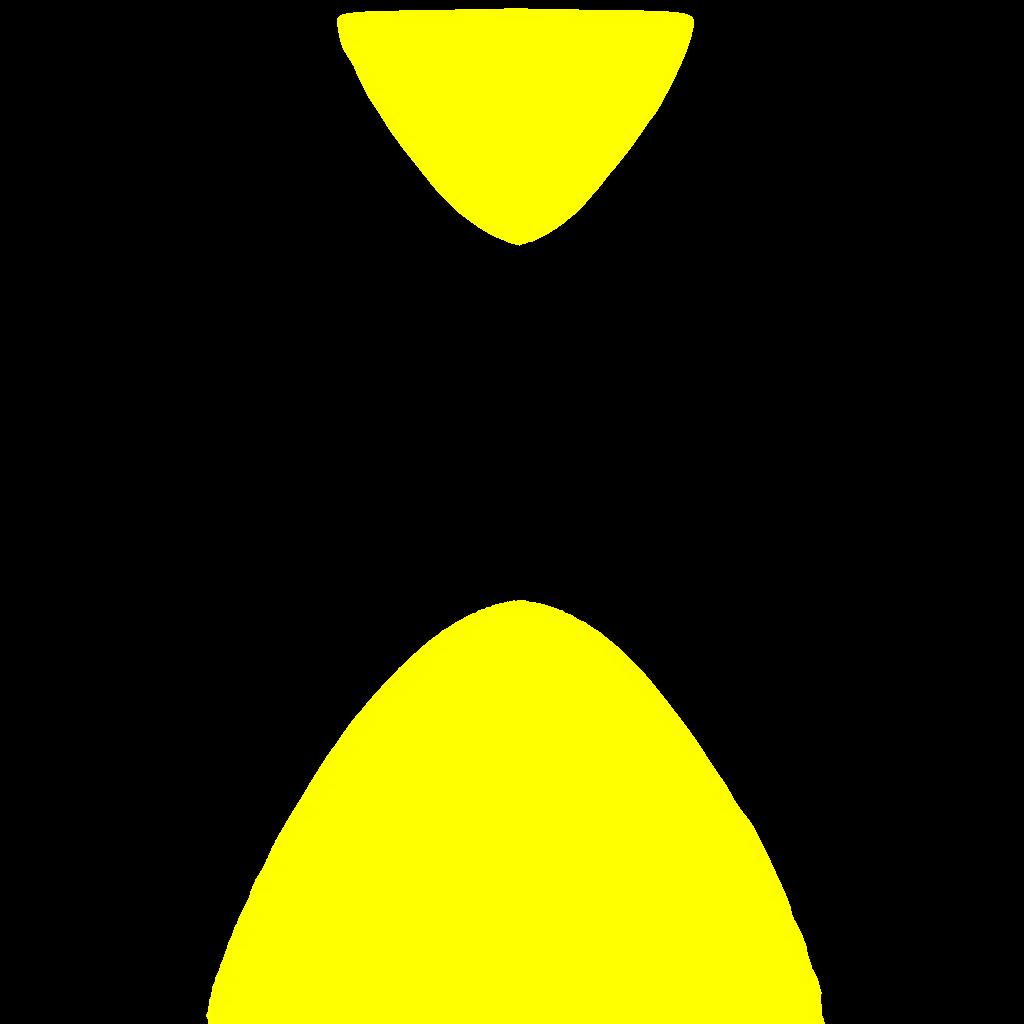}
\includegraphics[width=0.225 \textwidth]{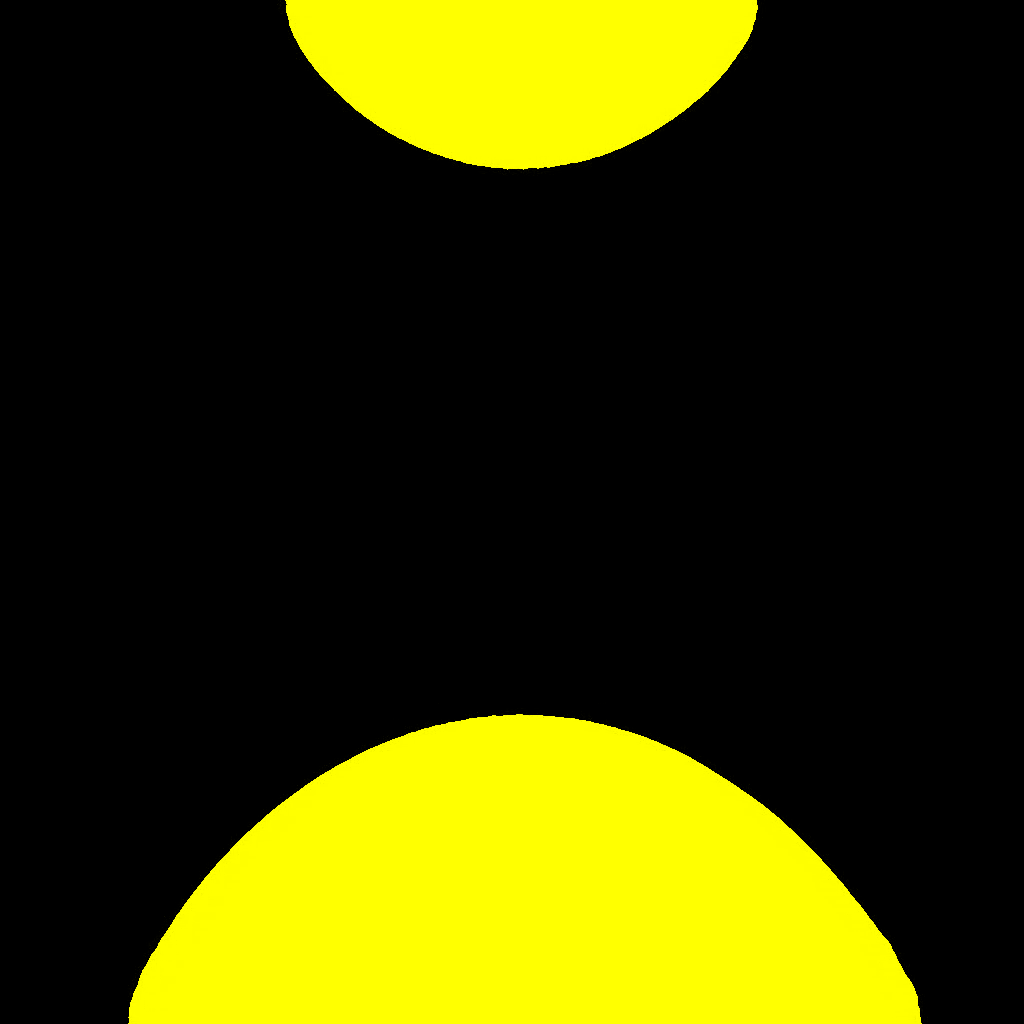}
\caption{\footnotesize Numerical simulation of the Muskat problem evolution with a vertical potential that attracts the yellow phase to the top and the bottom of the computational domain.  The potential is asymmetric and attracts the yellow phase more strongly to the bottom of the domain.   The yellow phase is pulled apart and undergoes a topological change. \label{fig:ripped_drop} }
\end{figure}

\color{black}

\subsection{Discussions on the structure of equilibrium shapes}

In this subsection we discuss global equilibrium of the energy with gravity potentials 
$$
\tilde\cE_\e(\rho_1,\rho_2) := \HC_\e(\rho_1, \rho_2) + \Phi(\rho_1,\rho_2).
$$
where $\Phi_i(x)= c_ix_d$, with $0<c_1<c_2$. The order of the constants denote that $\rho_1$ is the lighter fluid, where the vector $-e_d$ denotes the direction of gravity. The coordinate here is $x = (x', x_d)$ and for simplicity we consider a cylindrical domain, 
$$
\Sigma:= \{|x'|\leq 1\} \times [0,h]=B^{d-1}_1(0)\times[0,h].
$$
Here the convolution is taken with the extension of the density functions as zero outside of the domain,
with the density constraint $\rho_1+ \rho_2 = 1$ in $\Om$ and $\int_{\R^d} \dd\rho_i = M_i$. Since the densities are extended by zero outside $\Om$, this will produce a Neumann boundary condition for the interface $\Gamma$. In particular, we expect that $\Gamma$ intersects $\partial\Om$ orthogonally.

\medskip

Let us mention that away from the global equilibrium, there are diverse possibilities of stationary states for $\tilde\cE_\e$ even with zero potentials. For instance any choice of  characteristic functions $\rho_1$ and $\rho_2$ generating the interface as a disjoint union of spheres, $\{|x-a_i| = r_i\}$, is a stationary solution of the limit energy.   

\medskip

\medskip

 In the limit $\e\to 0$, the $\Gamma$-convergence properties indicate that the global equilibrium of the $\e$-energy converges to the limiting density pair $(\rho_1,\rho_2) = (\chi_{A}, \chi_{A^c})$, which is the global minimizer of the limit energy
$$
E_{\infty}(A):= {\rm{Per}}(A) + \int_A \Psi(x) \dd x, \hbox{ where }  \Psi(x):= (\Phi_1 - \Phi_2)(x). 
$$
under the volume constraint $\sL^d(A) = M$.  Away from the domain boundary, the classical minimal surface theory yields the $C^{2,\alpha}$ regularity of $\partial A$  with the Euler-Lagrange equation
$$
-\kappa -\Psi(x) =  \lambda  \hbox{ on } \partial A,
$$
where $\lambda$ is the Lagrange multiplier associated to the volume constraint and $\kappa$ stands for the curvature.

When $\partial A$ is away from the lateral and bottom portion of the cylinder, this corresponds to the classical pendant liquid droplet problem where the minimizer is known to be rotationally symmetric and convex (see \cite{Gonzalez}). When the droplet boundary touches the cylinder boundary, various shapes of drops are possible  (see Figure \ref{fig:eq-shapes}) and the complete description of possibly non-smooth global minimizers appear to be open. In general, when $\Sigma$ has $C^{1,\alpha}$ boundary, it is shown in \cite{Taylor} that $\partial A$ is $C^{1,\alpha}$ up to the boundary and meets $\partial \Sigma$ orthogonally.  For further discussion of available results we refer to \cite{MagMiha}. An ongoing work on numerical simulations for our flow suggest that the equilibrium states even for the $\e$-energy can be categorized as the ones on Figure \ref{fig:eq-shapes}. In dimension two, we could observe an additional equilibrium shape, as in Figure \ref{fig:eq-shapes-2d}.

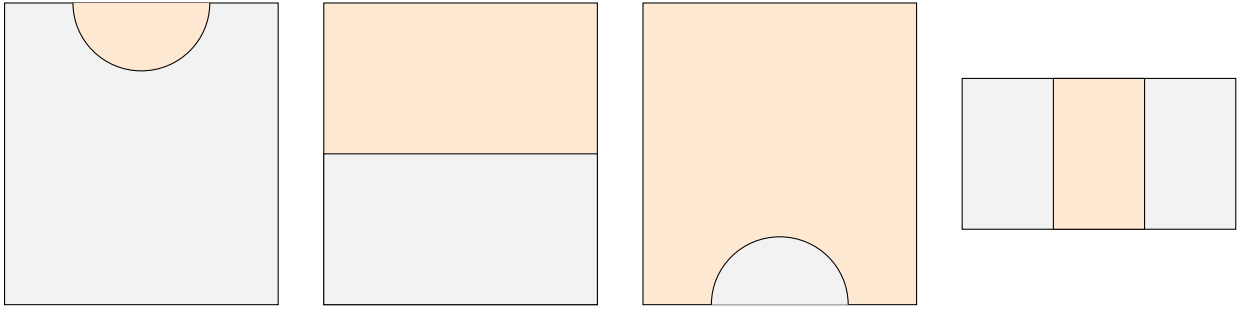
\begin{figure}
\begin{tikzpicture}
\filldraw [fill=gray!10!white, draw=black] (0,0) rectangle (3.6,4);
\filldraw [fill=orange!18!white, draw=black] (0.9,4) arc (180:360:0.9cm);
\filldraw [fill=orange!18!white, draw=black] (3.6+0.6,0) rectangle (4.2+3.6,4);
\filldraw [fill=gray!10!white, draw=black]  (3.6+0.6,0) rectangle (4.2+3.6,2);
\filldraw [fill=orange!18!white, draw=black] (7.8+0.6,0) rectangle (8.4+3.6,4);
\filldraw [fill=gray!10!white, draw=black] (7.8+0.6+2.7,0) arc (0:180:0.9cm);
\filldraw [fill=gray!10!white, draw=black] (12+0.6,1) rectangle (12.6+3.6,3);
\filldraw [fill=orange!18!white, draw=black] (12.6+1.2,1) rectangle (12.6+1.2+1.2,3);

\end{tikzpicture}\caption{Equilibrium shapes in any dimensions, depending on the proportion of the lighter vs. the heavier fluid: orange stands for the lighter fluid, while gray stands for the heavier fluid}
\label{fig:eq-shapes}
\end{figure}

\begin{center}
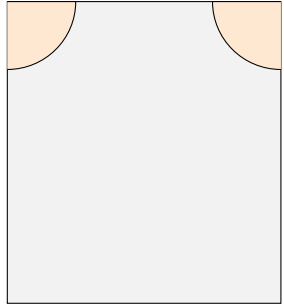
\begin{figure}
\begin{tikzpicture}

\filldraw [fill=gray!10!white, draw=black] (0,0) rectangle (3.6,4);
\filldraw [fill=orange!18!white, draw=black] (0,4) -- +(0:0.9) arc (0:-90:0.9);
\filldraw [fill=orange!18!white, draw=black] (3.6,4) -- +(180:0.9) arc (180:270:0.9);

\end{tikzpicture}\caption{An additional equilibrium shape that we predict only in dimension two}
\label{fig:eq-shapes-2d}
\end{figure}
\end{center}

{\sc Acknowledgement.} We thank Tim Laux for the helpful discussions and references. We thank Hugo Lavenant for his useful remarks, which in particular led us to fix a small issue in the proof of Proposition \ref{prop:characteristic}. We thank the two anonymous referees for their very careful reading of the manuscript. Their large amount of comments and remarks helped us to improve the manuscript significantly. 

I.K. is partially supported by NSF DMS-1566578.  M.J. is partially supported by  Simons Math + X Investigator - 510776, DARPA FA8750-18-2-0066, and NSF ATD-1737770. A.R.M was partially supported by the Air Force under the grant AFOSR MURI FA9550-18-1-0502.

\appendix

\section{Passing to the limit in the weak curvature equation based on \cite{LauOtt}}\label{sec:appendix}

Below we collected the results from \cite{LauOtt} that are used in the proof of Theorem \ref{thm:MAIN}. For two of the results we present their proofs as well, either because there was a minor difference between the result that we need and the one from \cite{LauOtt} or because we have found a shorter proof than the one in \cite{LauOtt}.

Let us note that the energy convergence assumption in Theorem \ref{thm:MAIN} is used only in Lemma \ref{lem:pointwise_reduction}.  The assumption plays a crucial role in the following arguments as it allows us to convert a limit requiring uniform convergence to one which only requires pointwise convergence.  In what follows, it will be useful for notational simplicity to introduce the following definition:
\begin{definition}\label{def:sigma_measure}
For a smooth rapidly decaying kernel $J:\R^d\to\R$ and a vector valued Radon measure $\nu\in\sM^d(\Om)$ we define the Radon measure $\sigma_J(\nu)\in\sM(\Om)$ as
$$\int_{\Om}\psi(x)\dd \sigma_J(\nu)(x):=\int_{\R^d}J(z)\int_{\Om} \psi(x)\Bigg{|}z\cdot \frac{\nu(x)}{|\nu(x)|}\Bigg{|}\dd|\nu|(x)\dd z,\ \forall\psi\in C(\Om).$$
\end{definition}

\begin{proposition}\label{prop:limit_eq}
Let $(\rho_i^\t)_{\t>0}$ be the piecewise constant interpolations constructed in \eqref{def:rho-tau}-\eqref{def:vpE-tau} and let $\rho_i\in L^1([0,T];BV(\Om;\{0,1\}))$, $i=1,2$ be their strong $L^1$ limits. If the assumption \eqref{ass:energy_conv} is fulfilled, then up to passing to a subsequence that we do not relabel, we have
\begin{align}\label{eq:convergence}
\int_0^T\int_{\Om}\xi\cdot [(\nabla K_{\e_\t}\star\rho_2^\t)\rho_1^\t&+(\nabla K_{\e_\t}\star\rho_1^\t)\rho_2^\t]\dd x\dd t\to \sum_i\int_0^T\int_{\Om}\frac{1}{4}\sum_{k}\zeta_k(x)\dd\sigma_{J}(D\rho_i)(x)\dd t\\
\nonumber&=\int_0^T\int_{\Om}\frac{\sigma}{2}\left( \frac{D\rho_1}{|D\rho_1|}\otimes \frac{D\rho_1}{|D\rho_1|} : \nabla \xi  + \nabla \cdot \xi \right) \left(|D \rho_1| +|D \rho_2|  \right)\dd x\dd t
\end{align}
${\rm{as\ }}\e_\t\da 0$, for any $\xi\in C^3([0,T]\times\Om;\R^d)$. Here, we denoted $J(z):=\sigma\sqrt{2\pi}|z\cdot e_1|^2G(z)$ and we used the decomposition $\nabla\xi^{{\rm{sym}}}(x)=\sum_k\zeta_k(x)n_k\otimes n_k$ with $\zeta_k\in C^\infty(\Om)$ and $n_k\in S^{d-1}$ such that $\{n_k\otimes n_k\}_{k=1}^{d(d+1)/2}$ is an appropriate basis of the the space of symmetric matrices.  
\end{proposition}


\begin{proof}
Let $\xi\in C^3([0,T]\times\Om;\R^d)$. Let us show that along a subsequence
\begin{align*}
\int_0^T\int_{\Om}\xi\cdot \left[(\nabla K_{\e_\t}\star\rho_2^\t)\rho_1^\t+(\nabla K_{\e_\t}\star\rho_1^\t)\rho_2^\t\right]\dd x\dd t\to  \sum_i\int_0^T\int_{\Om}\frac{1}{4}\sum_{k}\zeta_k(x)\dd\sigma_{J}(D\rho_i)(x)\dd t,
\end{align*}
${\rm{as\ }}\e_\t\da 0$.
This result is not straight forward, since  both the functional and the densities depend on $\tau$.  This difficulty is the key reason that we need the energy convergence assumption \eqref{ass:energy_conv}.  

Arguing exactly as in the proof of  \cite[Lemma 2.8]{LauOtt}, in order to show the convergence result \eqref{eq:convergence}, it is enough to show its time-independent version, i.e. 
$$\int_{\Om}\xi\cdot \left[(\nabla K_{\e_\t}\star\rho_2^\t)\rho_1^\t+(\nabla K_{\e_\t}\star\rho_1^\t)\rho_2^\t\right]\dd x\to \sum_i\int_{\Om}\frac{1}{4}\sum_{k}\zeta_k(x)\dd\sigma_{J}(D\rho_i)(x),
$$
${\rm{as\ }}\e_\t\da 0$, for $\sL^1$-a.e. $t\in[0,T].$ 

A computation similar to the one in the proof of Lemma \ref{lem:estimates}(vii) reveals 
\begin{align}\label{eq:taking_limit}
&\int_{\Om}\xi\cdot \left[(\nabla K_{\e_\t}\star\rho_2^\t)\rho_1^\t+(\nabla K_{\e_\t}\star\rho_1^\t)\rho_2^\t\right]\dd x\\
\nonumber&=\frac{\sigma\sqrt{2\pi}}{2\e_\t}\int_{\Om}\int_{\R^d}z G(z)\rho_2^{\t}(x+\sqrt{\e_\t}z)\rho_1^{\t}(x)\cdot\left[\xi(x)-\xi(x+\sqrt{\e_\t}z)\right]\dd z\dd x\\
\nonumber& = -\frac{\sigma\sqrt{2\pi}}{\sqrt{\epsilon_{\tau}}} \int_{\RR^d}  \int_{\Om} \rho_2^{\tau}(x+\sqrt{\epsilon_{\tau}} z)\rho_1^{\tau}(x)  \left[ z \otimes \nabla G(z) : \nabla \xi(x)\right]\dd x\ \dd t\\ 
\nonumber&+O\Big( \norm{D^2\xi}_{L^\infty} \int_{\RR^d}  |z|^2|\nabla G(z)| \int_{\Om}\rho_1^{\tau}(x)\rho_2^{\tau}(x+\sqrt{\epsilon_{\tau}} z)  \dd x\dd z\Big)
\end{align}
where, we were using a second order Taylor expansion (and the smoothness of $\xi$) in the last equality. Let us observe that the very last term in \eqref{eq:taking_limit} is converging to 0 as $\e_\t\da 0$ (due to the strong convergence $\rho_i^\t\to\rho_i$ in $L^1(\Om)$ and the fact that $\rho_1\rho_2\equiv 0$ a.e., see Proposition \ref{prop:strong_limit}).

\medskip

Therefore, it remains to prove 
$$
-\frac{\sigma\sqrt{2\pi}}{\sqrt{\epsilon_{\tau}}} \int_{\RR^d}  \int_{\Om} \rho_2^{\tau}(x+\sqrt{\epsilon_{\tau}} z)\rho_1^{\tau}(x)  \left[ z \otimes \nabla G(z) : \nabla \xi(x)\right]\dd x\dd z  \to \sum_i\int_{\Om}\frac{1}{4}\sum_{k}\zeta_k(x)\dd\sigma_{J}(D\rho_i)(x),
$$
as $\e_\t\da 0$. We note that
$$ -z\otimes \nabla G(z) : \nabla \xi= (z\otimes z)G(z) : \nabla \xi =(z\otimes z)G(z) : \nabla \xi^{\textrm{sym}} $$
where $ \nabla\xi^{\textrm{sym}}$ denotes the symmetric part of $\nabla \xi$. 
A basis for the space of $d\times d$ symmetric matrices is given by the $d+\frac{d(d-1)}{2}$ matrices 
$$\left\{e_i\otimes e_i\right\}_{i=1}^d \cup \left\{\frac{1}{2} (e_i+e_j)\otimes (e_i+e_j) \right\}_{1\leq i<j\leq d}$$
where $e_i\in \RR^d$ is the $i^{th}$ standard basis vector.
All of these basis matrices have the form $n\otimes n$ for some $n\in S^{d-1}$.  Thus, we may write
$$\nabla \xi^{\textrm{sym}}(x)=\sum_{k} \zeta_k(x) n_k\otimes n_k$$
where $\{n_k\}$ is any indexing of the above matrices and $\zeta_k(x) n_k\otimes n_k=P_k \nabla \xi^{\textrm{sym}}(x)$ where $P_k$ is the appropriate projection matrix.

 Therefore we have 
\begin{align*}
&\frac{\sigma\sqrt{2\pi}}{\sqrt{\epsilon_{\tau}}} \int_{\RR^d}  \int_{\Om} \rho_2^{\tau}(x+\sqrt{\epsilon_{\tau}} z)\rho_1^{\tau}(x)  \left[ z \otimes z\, G(z) : \nabla \xi(x)\right]\dd x\dd z\\
&=\frac{\sigma\sqrt{2\pi}}{\sqrt{\epsilon_{\tau}}}   \sum_{k}\int_{\RR^d} \int_{\Om}   \rho_2^{\tau}(x+\sqrt{\epsilon_{\tau}} z)\rho_1^{\tau}(x)  \zeta_k(x) |z\cdot n_k |^2 \,G(z) \dd x\dd z
\end{align*}

Since the functions $z\mapsto |z\cdot n_k|^2G(z)$ are rotation invariant, when integrating on $\R^d$, it is enough to consider only $n_k=e_1$, where $e_1$ is the first element of the standard basis on $\R^d$. Now, defining $J(z):=\sigma\sqrt{2\pi}|z\cdot e_1|^2G(z)$, we deduce the claim from Lemma \ref{lem:pointwise_reduction} and Lemma \ref{lem:localized_hc_convergence}.

\medskip

Now, let us remark that a computation completely parallel to \cite[Proof of Lemma 3.6.]{LauOtt} reveals furthermore that
\begin{equation}\label{to_show}
\frac{\sigma\sqrt{2\pi}}{\sqrt{\epsilon_{\tau}}} \int_{\RR^d}  \int_{\Om} \rho_2^{\tau}(x+\sqrt{\epsilon_{\tau}} z)\rho_1^{\tau}(x)   |z\cdot n |^2G(z) \zeta(x)\dd x\dd z \\
 \to \int_{\Om} \frac{\sigma}{2}\left(\Bigg{|}\frac{D\rho_1}{|D\rho_1|}\cdot n\Bigg{|}^2 +1\right)\zeta(x) \left(|D \rho_1| +|D \rho_2|  \right)\dd x
 \end{equation}
as  $\e_\t\to 0$, and so, in this case we would have
\begin{align*}
\frac{\sigma\sqrt{2\pi}}{\sqrt{\epsilon_{\tau}}} \sum_k \int_{\RR^d}  \int_{\Om} \rho_2^{\tau}(x+\sqrt{\epsilon_{\tau}} z)&\rho_1^{\tau}(x)   |z\cdot n_k |^2G(z) \zeta_k(x)\dd x\dd z \\ 
&\to  \int_{\Om} \frac{\sigma}{2} \sum_{k}\left( \frac{D\rho_1}{|D\rho_1|}\otimes \frac{D\rho_1}{|D\rho_1|}: n_k\otimes n_k +1\right)\zeta_k(x) \left(|D \rho_1| +|D \rho_2|  \right)\dd x.
\end{align*}
This would in fact complete the claim, as we can compute
\begin{align*}
\sum_{k}\left(\frac{D\rho_1}{|D\rho_1|}\otimes \frac{D\rho_1}{|D\rho_1|}: n_k\otimes n_k+1\right)\zeta_k(x)&=\frac{D\rho_1}{|D\rho_1|}\otimes \frac{D\rho_1}{|D\rho_1|}:\sum_{k}\zeta_k(x) n_k\otimes n_k +\sum_{k} \zeta_k(x)\tr(n_k\otimes n_k)\\
&=\frac{D\rho_1}{|D\rho_1|}\otimes \frac{D\rho_1}{|D\rho_1|}:\nabla\xi^{\textrm{sym}}(x)+\tr\big(\nabla\xi^{\textrm{sym}}(x)\big)\\
&=\frac{D\rho_1}{|D\rho_1|}\otimes \frac{D\rho_1}{|D\rho_1|}:\nabla\xi(x)+\nabla \cdot \xi(x)
\end{align*}
where we have used $\tr(n_k\otimes n_k)=1$ and the fact that the antisymmetric parts of $\nabla \xi$ are annihilated by the above operations.
\end{proof}

The conclusion in the previous proposition  is made by the following two lemmas.

\begin{lem}\label{lem:pointwise_reduction}  If $\bm{\rho}^{\tau} \to \bm{\rho}$ in $L^1(\Om)\times L^1(\Om)$ and $\HC_{\epsilon_{\tau}}(\bm{\rho}^{\tau})\to \cE(\bm{\rho})$ as $\e_\t\da 0$, and  $J$ is a nonnegative kernel with rapid decay (i.e. $J(z)\leq |P(z)|G(z)$ for some polynomial $P$) then
\begin{equation}  
\lim_{\tau\to 0} \frac{1}{\sqrt{\epsilon_{\tau}}}\int_{\RR^d} J(z)\Big| \int_{\Om} \Big(\rho_1^{\tau}(x+\sqrt{\epsilon_{\tau}} z)\rho_2^{\tau}(x)-\rho_1(x+\sqrt{\epsilon_{\tau}} z)\rho_2(x) \Big) \dd x\Big| \dd z=0 
\end{equation} 
\end{lem}
\begin{proof} 
The proof of this result is the same as the one of \cite[Lemma 3.7]{LauOtt}.
\end{proof}

Thanks to Lemma \ref{lem:pointwise_reduction} we just need the following \emph{pointwise} convergence result.  This applies to the multiphase case and is stronger than what is needed here.  Similarly to \cite[Lemma 2.8, Lemma 3.6]{LauOtt} we can formulate the following result. 

\begin{lem}\label{lem:localized_hc_convergence} Suppose that $\rho_1, \rho_2\in \rm{BV}(\Om;\{0,1\})$ such that $\rho_1(x)\rho_2(x)=0$ for a.e. $x\in\Om$.   Then for any smooth function $\zeta:\Om\to\R$ and any even nonnegative kernel $J:\R^d\to\R$ with rapid decay we have 
\begin{align*} 
&\lim_{\epsilon\to 0} \int_{\Om} \zeta(x)\left[ \rho_1(x) (J_{\epsilon} \star\rho_2 (x)\right]\dd x=\frac{1}{4}\left[\int_{\Om}\zeta(x)\dd\sigma_J(D\rho_1)(x)+\int_{\Om}\zeta(x)\dd\sigma_J(D\rho_2)(x)\right]\\
&-\frac{1}{4}\int_{\Om}\zeta(x)\dd\sigma_J(D(\rho_1+\rho_2))(x)\\
\end{align*}
where we define the Radon measure $\sigma_J(D\rho_i)\in\sM(\Om)$ as in Definition \ref{def:sigma_measure} and we have used the notation $J_\e(z):=J(z/\e).$
\end{lem}

The proof supplied below is different from the one by Laux-Otto in \cite[Lemma 2.8, Lemma 3.6]{LauOtt}. Instead of disintegrating on $\RR^d$ and using one-dimensional arguments we obtain upper and lower bounds using mollifiers. 

\begin{proof}
We begin by showing 
$$\lim_{\epsilon\to 0} \int_{\Om} \zeta(x)\left[ \rho_1(x) J_{\epsilon} \star\rho_2 (x)\right]\dd x=\lim_{\epsilon\to 0} \int_{\Om} \frac{1}{2}\zeta(x)\left[ \rho_1(x) J_{\epsilon} \star\rho_2 (x)+\rho_2(x) J_{\epsilon} \star\rho_1 (x)\right]\dd x$$
which amounts to showing that
$$\lim_{\epsilon\to 0} \int_{\Om} \zeta(x)\left[ \rho_1(x) J_{\epsilon} \star\rho_2 (x)- \rho_2(x) J_{\epsilon} \star\rho_1 (x)\right]\dd x=0.$$
Expanding out the convolution we have
$$\frac{1}{\epsilon}\int_{\Om}\int_{\RR^d} J(z) \left[ \rho_1(x)\rho_2(x+\epsilon z)\zeta(x) -\rho_1(x+\epsilon z)\rho_2(x)\zeta(x) \right]\dd z\dd x.$$
Changing variables $x\mapsto x-\epsilon z$ and then $z\mapsto -z$ in the second term of the integral we get
\begin{align*}
\Bigg{|}\frac{1}{\epsilon}\int_{\Om}\int_{\RR^d} J(z)\rho_1(x)\rho_2(x+\epsilon z) &\left[ \zeta(x)-\zeta(x+\epsilon z) \right]\dd z\dd x\Bigg{|}\\
&\le\|\nabla\zeta\|_\infty\Bigg{|}\int_{\Om}\int_{\RR^d}|z|J(z)\rho_1(x)\rho_2(x+\epsilon z)\dd z\dd x\Bigg{|}.\\
&=O\Big(\epsilon\norm{\nabla \zeta}_{\infty}\hc_{\epsilon}(\rho_1,\rho_2)\Big)
\end{align*}
Thus, the dominated convergence theorem with the fact that $\rho_1\rho_2=0$ a.e. yield that the quantity vanishes as $\epsilon \to 0$. 

Now we can restrict our attention to the limit
$$\lim_{\epsilon\to 0} \int_{\Om} \frac{1}{2}\zeta(x)\left[ \rho_1(x) J_{\epsilon} \star\rho_2 (x)+\rho_2(x) J_{\epsilon} \star\rho_1 (x)\right]\dd x$$
Since $\rho_i\in\rm{BV}(\Om;\{0,1\})$ and $\rho_1(x)\rho_2(x)=0$ a.e.,  
\begin{align*}
&\rho_1(x+\epsilon z)\rho_2(x)+\rho_1(x)\rho_2(x+\epsilon z)\\
&=\frac{1}{2}|\rho_1(x+\epsilon z)-\rho_1(x)| + \frac{1}{2}|\rho_2(x+\epsilon z)-\rho_2(x)|-\frac{1}{2}|(\rho_1+\rho_2)(x+\epsilon z)-(\rho_1+\rho_2)(x)|,
\end{align*}
which follows from directly by evaluating both sides.  Thus, it suffices to prove 
\begin{equation}\label{eq:to_show}
\lim_{\epsilon\to 0} \frac{1}{\epsilon} \int_{\RR^d} J(z) \int_{\Om}  \zeta(x)|\chi(x+\epsilon z)-\chi(x)|\dd x \dd z=\int_{\Om} \zeta(x)\dd\sigma_J(D\chi)(x)
\end{equation}
 for any $\chi \in \rm{BV}(\Om; \{0,1\})$.

For $\delta>0$ let $\eta_{\delta}$ be a smooth approximation to the identity, and set $\chi_{\delta}=\eta_{\delta} \star\chi$, such that $\chi_\d\to\chi$, as $\d\to 0$ in the sense of strict convergence of BV functions (i.e. $\chi_\d\to\chi$ in $L^1$ and $\int_\Om|D\chi_\d|\to\int_\Om|D\chi|$ as $\d\to 0$;  cf. \cite[Definition 3.14]{AmbFusPal}).

Then, we also have
\begin{equation}\label{eq:1}
\frac{1}{\epsilon} \int_{\RR^d} J(z) \int_{\Om}  \zeta(x)|\chi(x+\epsilon z)-\chi(x)| \dd x \dd z=\lim_{\delta \to 0} \int_{\RR^d} J(z) \int_{\Om}\zeta(x)  \frac{1}{\epsilon}|\chi_{\delta}(x+\epsilon z)-\chi_{\delta}(x)| \dd x \dd z.
\end{equation}
Without loss of generality, one may suppose that $\zeta\ge 0$. By Jensen's inequality, the above is 
\begin{align*}
&\leq \lim_{\delta \to 0} \int_{\RR^d} J(z) \int_{\Om} \zeta(x)   \int_{0}^1 \Big|z\cdot  \nabla \chi_{\delta}(x+\epsilon tz) \Big| \dd t \dd x \dd z\\
&=   \lim_{\delta \to 0} \int_0^1 \int_{\R^d} J(z) \int_{\Om}\zeta(x-\epsilon tz) | z\cdot \nabla \chi_{\delta}(x)| \dd x \dd z \dd t\\
&=\int_{\Om}\zeta(x) \dd\sigma_J(D\chi)(x) +O\bigg(\epsilon \norm{\nabla \zeta}_{\infty} \int_{\Om}  |D \chi| \bigg).
\end{align*}
Taking $\epsilon \to 0$ we get the desired upper bound for the limit. 

Conversely, for $\delta>0$ fixed we have from Jensen's inequality
\begin{align*}
&\frac{1}{\epsilon} \int_{\RR^d} J(z) \int_{\Om}  \zeta(x)|\chi(x+\epsilon z)-\chi(x)| \dd x \dd z \\
&\geq \int_{\RR^d}J(z) \int_{\Om}\zeta(x) \frac{1}{\epsilon} |\chi_{\delta}(x+\epsilon z)-\chi_{\delta}(x)| \dd x \dd z+O\bigg(\delta \norm{\nabla \zeta}_{\infty} \int_{\Om}  |D \chi| \bigg).
\end{align*}
Taking $\epsilon\to 0$ we get
$$ \int_{\RR^d} J(z) \int_{\Om}\zeta(x) |z\cdot \nabla \chi_{\delta}(x)|  \dd x \dd z+O\bigg(\delta \norm{\nabla \zeta}_{\infty} \int_{\Om}  |D \chi| \bigg). $$
Finally,
$$\lim_{\delta \to 0} \int_{\RR^d} J(z) \int_{\Om} \zeta(x) |z\cdot \nabla \chi_{\delta}(x)| \dd x \dd z +O\bigg(\delta \norm{\nabla \zeta}_{\infty} \int_{\Om}  |D \chi| \bigg) = \int_{\Om}\zeta(x)\dd \sigma_J(D\chi)(x).   $$
The result follows.

\end{proof}

\end{document}